\documentclass[a4paper]{article} %,draft
\usepackage{a4wide}
\usepackage[utf8]{inputenc}
\usepackage[ngerman, english]{babel}
\usepackage{amsmath,amsthm,amssymb,amsfonts}
\usepackage{mathrsfs}
\usepackage{ifthen}
\usepackage{enumerate}
\usepackage{comment}
\usepackage{subcaption}
\usepackage{hyphsubst}
\usepackage{hyperref}
\usepackage{hyphsubst}
\usepackage{cleveref}
\usepackage[dvipsnames]{xcolor}
\usepackage[skins]{tcolorbox}
\usepackage{calc}
\usepackage{graphicx}
\usepackage{tabularx}
\usepackage[toc,page]{appendix}
\usepackage[bbgreekl]{mathbbol}
\DeclareSymbolFontAlphabet{\mathbbm}{bbold}
\DeclareSymbolFontAlphabet{\mathbb}{AMSb}

\usepackage[utf8]{inputenc}
\usepackage[ngerman, english]{babel}
\usepackage{amsmath,amsthm,amssymb,amsfonts}
\usepackage{mathrsfs}
\usepackage{ifthen}
\usepackage{enumerate}
\usepackage{comment}
\usepackage{subcaption}
\usepackage{hyphsubst}
\usepackage{xcolor}
\usepackage{calc}
\usepackage{graphicx}
\usepackage{hyperref}
\usepackage[toc,page]{appendix}
\usepackage{cleveref}
\usepackage{leftindex}
\usepackage{pgfplots}
\pgfplotsset{compat=1.5.1}
\usepackage[bbgreekl]{mathbbol}

\makeatletter
\def\@hspace#1{\begingroup\setlength\dimen@{#1}\hskip\dimen@\endgroup}
\makeatother

\newtheorem{theorem}{Theorem}[section]
\newtheorem{lemma}[theorem]{Lemma}
\newtheorem{definition}[theorem]{Definition}

\newtheorem{remark}[theorem]{Remark}
\newtheorem{corollary}[theorem]{Corollary}

\newtheorem{proposition}[theorem]{Proposition}

\newtcbox{\mybox}{on line,
  colframe=white!10!white,colback=Goldenrod,
  boxrule=0.5pt,arc=4pt,boxsep=0pt,left=6pt,right=6pt,top=6pt,bottom=6pt}

\def\ljump{[\![}
\def\rjump{]\!]}
\def\WG{\Omega}
\def\domcs{\Gamma}
\def\intf{\Sigma}
\def\Rp{\mathbb{R}^+}

\def\Nz{\mathbb{N}_0}

\def\indL{N_{\domcs}}
\def\indn{n}
\def\indm{m}
\def\deltaem{\delta_{\epsilon,\mu}}
\def\deltaemt{\tilde\delta_{\epsilon,\mu}}
\def\epsmax{\epsilon_\mathrm{max}}
\def\epsmin{\epsilon_\mathrm{min}}
\def\mumax{\mu_\mathrm{max}}
\def\mumin{\mu_\mathrm{min}}
\def\gradF{\nabla_\mathrm{3D}}
\def\curlF{\curl_\mathrm{3D}}
\def\divF{\div_\mathrm{3D}}
\def\gradg{\nabla_{\domcs}}
\def\divg{\div_{\domcs}}
\def\curlg{\curl_{\domcs}}
\def\Curlg{\Curl_{\domcs}}
\def\nvf{\boldn}
\def\nvt{\hat\boldn}
\def\tv{\hat\boldt}
\def\rotm{R}
\def\xf{\boldx}
\def\xt{\hat\boldx}
\def\Eff{\mathcal{E}}
\def\Hff{\mathcal{H}}
\def\Uff{\mathcal{G}}
\def\Ef{\mathbb{E}}
\def\Hf{\mathbb{H}}
\def\Uf{\mathbb{G}}
\def\Et{\boldE_T}
\def\Ht{\boldH_T}
\def\Ut{\boldG_T}
\def\Ex{E_1}
\def\Ey{E_2}
\def\Ez{E_3}
\def\Hx{H_1}
\def\Hy{H_2}
\def\Hz{H_3}
\def\Ux{G_1}
\def\Uy{G_2}
\def\Uz{G_3}
\def\Etk{\boldE_{T,k}}
\def\Ezk{E_{3,k}}
\def\Htk{\boldH_{T,k}}
\def\Hzk{H_{3,k}}
\def\bL{\boldL}
\def\Lspace{\boldL} % vectorial Lebegues space
\def\Hspace{\boldH} % vectorial Sobolev space
\def\BLO{\mathcal{L}}
\def\Vspace{\boldV}
\def\Vpotspace{V^p}
\def\Wm{W_m}
\def\Vm{V_m}
\def\Xspace{X}
\def\Hcd{\Hspace(\curlg,\divg\mu^{-1}0;\Omega)}

\def\bv{\boldv}
\def\Deltamu{\Delta_{\mu}}
\def\Deltaeps{\Delta_{\epsilon}}

\def\Meps{M_\epsilon}
\def\Mmui{M_{\mu^{-1}}}
\def\DGrad{\mathbf{D}} % Grad
\def\embL{E_{L^2}}
\def\MepsVec{\mathbf{M}_\epsilon}
\def\MmuiVec{\mathbf{M}_{\mu^{-1}}}
\def\Dcurl{D} % curl
\def\embLVec{\mathbf{E}_{L^2}}
\def\Top{T}
\def\Aop{A}
\def\Bop{B}
\def\Kop{K}
\def\ases{a}
\def\bses{b}
\def\kses{k}
\def\Cdir{c_\mathrm{Dir}}
\def\Ctrhi{C_\mathrm{tr,1}}
\def\Ctrhii{C_\mathrm{tr,2}}
\def\CBop{C_{\Bop}}
\def\Cdet{C_{\ref{lem:det}}}
\def\Ceftrace{C_\mathrm{tr,3}}
\def\Cvec{C_\mathrm{vec}}
\def\Cab{C_\mathrm{ab}}
\def\CAiu{C_\mathrm{Ai}}
\def\cAil{c_\mathrm{Ai}}
\def\Cnl{c_\mathrm{ind}}
\def\Cnu{C_\mathrm{ind}}
\def\Cjrat{C_{\mathrm{rat}}}
\def\Col{C_{\mathrm{ol}}}
\def\clb{c_\mathrm{lb}}
\def\jtmp{j^\mathrm{tmp}}
\def\jttmp{\tilde j^\mathrm{tmp}}
\def\xtmp{x^\mathrm{tmp}}
\def\xttmp{\tilde x^\mathrm{tmp}}
\def\troot{t}
\def\ttroot{\tilde t}
\def\cgap{c_\mathrm{gap}}
\def\cgapt{\tilde c_\mathrm{gap}}
\def\cgapai{c_\mathrm{\Ai,gap}}
\def\Col{C_\mathrm{ol}} % overlap
\def\Crjy{C_{JY}}
\def\Crjj{C_{JJ}}
\def\Crrjjyy{C_{JJYY}}
\def\qclose{c}
\def\cgka{c_{\gamma,\kappa,\check\alpha}}
\def\Cgka{C_{\gamma,\kappa,\check\alpha}}
\def\MB{M_{\check\bses}}
\def\Amat{\ul{A}}
\def\AmatNeu{\Amat^\mathrm{Neu}}
\def\AmatDir{\Amat^\mathrm{Dir}}
\def\Bmat{\ul{B}}
\def\BmatNeu{\Bmat^\mathrm{Neu}}
\def\BmatDir{\Bmat^\mathrm{Dir}}
\def\vvec{\ul{v}}
\def\wvec{\ul{w}}

\def\Gammaf{Gamma}
\def\Ai{\operatorname{Ai}}
\def\Bi{\operatorname{Bi}}
\def\Aiz{a\hspace{-1.5mm}a}
\def\Aipz{\Aiz'}
\def\Aizc{\hat\Aiz}

\def\besselj{J}
\def\besseljz{j}
\def\besseljzc{\hat j}
\def\besseljpz{j'}
\def\bessely{Y}
\def\besselyz{y}

\def\mushift{\bbmu_m}
\def\omshift{\bbomega_m}
\def\phaseB{\mathcal{B}_m}
\def\rasy{q}
\def\tstar{t_*}
\def\tstari{t_{1*}}
\def\mstar{m_*}
\def\nstar{n_*}
\def\nshift{n_\mathrm{sh}}
\def\evsadir{\tilde j}
\def\evsaneu{\tilde j'}
\def\evsac{\hat{\tilde j}}

\newcommand{\tf}{^\dagger}
\newcommand{\dd}{\mathrm{d}}

\newcommand{\ol}[1]{\overline{#1}}
\newcommand{\ul}[1]{\underline{#1}}
\newcommand{\spl}{\langle}
\newcommand{\spr}{\rangle}
\newcommand{\bpm}{\begin{pmatrix}}
\newcommand{\epm}{\end{pmatrix}}

\renewcommand{\Re}{\operatorname{Re}}

\DeclareMathOperator{\curl}{curl}
\DeclareMathOperator{\Curl}{Curl}

\DeclareMathOperator{\diag}{diag}

\renewcommand{\div}{\operatorname{div}}
\DeclareMathOperator{\dom}{dom}

\DeclareMathOperator{\tr}{tr}

\newcommand{\setC}{\mathbb{C}}

\newcommand{\setN}{\mathbb{N}}

\newcommand{\setR}{\mathbb{R}}

\newcommand{\setZ}{\mathbb{Z}}

\newcommand{\boldn}{\mathbf{n}}

\newcommand{\boldt}{\mathbf{t}}
\newcommand{\boldu}{\mathbf{u}}
\newcommand{\boldv}{\mathbf{v}}

\newcommand{\boldx}{\mathbf{x}}

\newcommand{\boldE}{\mathbf{E}}

\newcommand{\boldG}{\mathbf{G}}
\newcommand{\boldH}{\mathbf{H}}

\newcommand{\boldL}{\mathbf{L}}

\newcommand{\boldV}{\mathbf{V}}

\newcommand{\boldY}{\mathbf{Y}}

\newcommand{\calH}{\mathcal{H}}
\newcommand{\calI}{\mathcal{I}}

\newcommand{\calP}{\mathcal{P}}

\newcommand{\calT}{\mathcal{T}}

\newcommand{\calX}{\mathcal{X}}

\definecolor{brickred}{rgb}{0.8, 0.25, 0.33}
\definecolor{bostonuniversityred}{rgb}{0.8, 0.0, 0.0}
\definecolor{cornellred}{rgb}{0.7, 0.11, 0.11}
\definecolor{corn}{rgb}{0.98, 0.93, 0.36}
\definecolor{schoolbusyellow}{rgb}{1.0, 0.85, 0.0}
\definecolor{TUblue}{rgb}{0,102,153}
\colorlet{TUbluelight}{TUblue!30!white}

\usepackage{fancyhdr}
\usepackage[backend=biber,
	giveninits=true,
	maxbibnames=8,
	date=year,
	doi=true,
	url=false,
	isbn=false,
	eprint=false]{biblatex}
\AtEveryBibitem{\clearlist{language}} % supress language field
\renewbibmacro{in:}{%
  \ifboolexpr{%
     test {\ifentrytype{article}}%
     or
     test {\ifentrytype{inproceedings}}%
  }{}{\printtext{\bibstring{in}\intitlepunct}}%
}
\title{Modal bases of coaxial electromagnetic step index fibers
\thanks{
The author acknowledges funding from Deutsche Forschungsgemeinschaft (DFG, German Research Foundation), projects 541433971 and 258734477 - SFB 1173.
}
}
\author{Martin Halla%
\thanks{%
Karlsruhe Institute of Technology,
Institute for Applied and Numerical Mathematicsa,
Englerstra\ss e 2,
76131 Karlsruhe, Germany
(martin.halla@kit.edu)
}
}

\begin{document}

% d/dx log((x+c)/(x-c))/(4c^3)+x/(2c^2(c^2-x^2)) = 1/(x^2-c^2)

\date{}
\maketitle
\begin{abstract}
\noindent
We consider the eigenvalue problem to find the modes of an electromagnetic coaxial step index fiber.
More specific, we consider a closed (meaning PEC boundary conditions) cylindrical waveguide with circular cross section $\domcs$, wave propagation modeled by the time-harmonic Maxwell's equations with frequency $\omega$, the permeability $\mu$ and the permittivity $\epsilon$ being scalar, uniformly positive, piece-wise constant and depending only on the radial variable of the cross section.
We prove that if the deviation from the homogeneous case is small, i.e., $\deltaem:=\|\epsilon-\epsilon_0\|_{L^\infty}+\|\mu-\mu_0\|_{L^\infty}\ll1$, then
the (appropriately scaled) tangential electric (magnetic) fields of the modes form Riesz bases in $\Hspace_{0}(\curl_{\domcs};\domcs)$ ($\Hspace(\curl_{\domcs};\domcs)$) and the natural trance space $\Hspace_{0}^{-1/2}(\curl_{\domcs};\domcs)$ ($\Hspace^{-1/2}(\curl_{\domcs};\domcs)$).
These results hold also for complex frequencies $\omega$.
In addition, if $\omega\in\setR$, then for small enough $\deltaem$ all wavenumbers are located on the axes and there exist no backward modes.
Key tools in the analysis are a particular reformulation of the eigenvalue problem, the perturbation theory for selfadjoint operators under a local subordinate condition and uniform properties of Bessel functions.
\\

\\
%35P05 General topics in linear spectral theory for PDEs
%35P25 Scattering theory for PDEs
%35R05 PDEs with low regular coefficients and/or low regular data
%35R30 Inverse problems for PDEs
%65N25 Numerical methods for eigenvalue problems for boundary value problems involving PDEs
%78M10 Finite element, Galerkin and related methods applied to problems in optics and electromagnetic theory
%78M35 Asymptotic analysis in optics and electromagnetic theory

\end{abstract}

%\tableofcontents

\section{Introduction}\label{sec:introduction}

Electromagnetic waveguides play an important role, e.g., for telecommunication~\cite{ImperialeJoly14}, %Imperiale, Joly ’14; Beck, Imperiale, Joly ’15, ’20; Hamad, Beck, Imperiale, Joly ’22
energy transfer~\cite{cmacs2003baseline}
%\cite{Leclerc25}
and as optical fiber amplifiers to produce highly coherent laser outputs \cite{JaureguiLimpertTunnermann13}.
A modal analysis and basis properties of the modes are important, e.g., for the construction of Dirichlet-to-Neumann/Calderon operators, the analysis of stability constants, and the justification of numerical transparent boundary conditions (in particular perfectly matched layers).
In this article we are interested in the modes of a closed (meaning PEC boundary conditions) cylindrical waveguide with circular cross section, wave propagation modeled by Maxwell’s equations, the permeability $\mu$ and the permittivity $\epsilon$ being scalar, uniformly positive, piece-wise constant and depending only on the radial variable of the cross section, i.e., a step index fiber.
In applications the deviation from a homogeneous index profile is usually small (see, e.g., \cite[Table~3]{MoraPDemkowiczEtal25}).
A particular motivation for this work is the study of \emph{long} optical fibers and we refer to the introduction of \cite{MelenkDemkowiczHenneking25} for references regarding the state of the art of simulating nonlinear optical fiber amplifiers.
The relevance of this topic is evident by the large number of books on \emph{optical fibers}, e.g., \cite{Marcuse13,SheSmirSmo22} to name only a few.
We also mention studies on twisted and open waveguides \cite{Leclerc25,BambergerBonnet90,JolyPoirier95,JolyPoirier99} and recent works on nonlinear fibers \cite{HennekingGrosekDemkowicz23,HennekingGrosekDemkowicz25,MelenkDemkowiczHenneking25,DrakeGopalakrishnanGoswamiGrosek20}, \cite{SheSmolSmi19,SheSmirSmo22,AbgaryanShestopalovRomanov23}.
\\

For homogeneous materials it is well known that the modal problem reduces to the study of scalar Helmholtz equations, see, e.g., \cite{Kim17}.
However, as soon as the material is heterogeneous the situation changes and the mathematical effort to analyze the modes dramatically increases.
Indeed, in this case two previously decoupled equations are now coupled and the selfadjoint nature (in the convenient sense) is violated.
The modal eigenvalue problem (EVP) can still be considered to be a selfadjoint, linear EVP in a Krein space or to be a quadratic EVP with selfadjoint operator coefficients in a Hilbert space, but these considerations do not yield much benefit.
Thus even the question to establish the completeness of modes \cite{SheSmirSmo22,SheSmi13b,Delitsyn00,HallaMonk24} poses severe mathematical challenges.
Nevertheless, the main result of this article is a much stronger property than completeness, namely a Riesz basis property of the modes.
% \cite[Conclusion]{MalykhSevastyanova19}
However, our analysis also indicates that for general waveguide configurations we cannot expect the modes to have basis properties and this observation led us to study the Calerd\'on operator in the companion article \cite{Halla26Calderon} with alternative techniques.
Let us also note that although our analysis requires significant effort and assumes certain non-trivial assumptions on the geometry of the waveguide (circular step index fiber), we mention that we do not even touch more sophisticated setups such as bent \cite{DemkowiczGopalaKrishnanHeuer24,DemkowiczHallaMelenk26,MoraPDemkowiczEtal25}, twisted \cite{GopalaNeunteufel25,Leclerc25} or open \cite{BambergerBonnet90,JolyPoirier95,JolyPoirier99} waveguides.\\

We are going to analyze the modal eigenvalue problem as a non-selfadjoint perturbation of a selfadjoint operator.
Spectral perturbation theory for non-selfadjoint operators is a rather challenging topic and we refer to the books \cite{GohbergKrein69,Markus88,Jeribi21}.
The convenient approach to this end is to assume that the perturbation satisfies a subordination condition and we refer to the introduction of \cite{MityaginSiegl24} for further references in this regard.
However, recently \cite{MityaginSiegl19} reported an important generalization of this notion, which exchanges the subordination condition with a \emph{local} version, and \cite{DemkowiczHallaMelenk26} gave some further extensions in this framework.
The latter two will be a main tool for our analysis.
In particular, the theory of \cite{MityaginSiegl19} assumes certain properties of the unperturbed selfadjoint operator.
In our case the selfadjoint operator is itself a perturbation of the operator for homogeneous materials.
Since this (selfadjoint) perturbation is not subordinate in any suitable sense, we are going to establish those necessary properties ``by hand'' by analyzing Bessel functions.
Here a crucial point is to derive all estimates in a uniform way.
\\

The remainder of this article is structured as follows.
In \Cref{sec:evp} we introduce our setting, the modal eigenvalue problem~\eqref{eq:evp-full} to be studied, and derive a suitable reformulation~\eqref{eq:evp} of the former.
In \Cref{sec:bessel} we study certain uniform properties of Bessel functions.
In \Cref{sec:perturbation-sa} we apply those properties to establish the necessary requirements for the selfadjoint perturbation.
Finally, in \Cref{sec:perturbation-nsa} we apply \cite{MityaginSiegl19,DemkowiczHallaMelenk26} and establish our main result in~\Cref{thm:main}.

\section{The modal eigenvalue problem}\label{sec:evp}

\subsection{Setting and assumptions}\label{subsec:setting}

We consider a cylindrical waveguide $\WG:=\domcs\times\Rp$ with bounded Lipschitz cross section $\domcs\subset\mathbb{R}^2$, and identify $\domcs$ with $\domcs\times\{0\}$.
Let $\nvf$ and $\nvt$ be the outward unit vectors of $\WG$ and $\domcs$ respectively, which are related by $\nvf=(\nvt,0)$ on $\partial\domcs\times\Rp$.
We denote the convenient differential operators in 3D as follows
\[
\gradF u:=\bpm \partial_x u\\\partial_y u\\ \partial_z u\epm,\quad
\curlF\bpm u_1\\u_2\\ u_3\epm:=\gradF\times\bpm u_1\\u_2\\ u_3\epm,\quad
\divF\bpm u_1\\u_2\\ u_3\epm:=\gradF\cdot\bpm u_1\\ u_2\\ u_3\epm.
\]
%\[
%\gradF u:=(\partial_x u,\partial_y u, \partial_z u)^\top,\quad
%\curlgF\boldu:=\gradF\times\boldu,\quad
%\divgF\boldu:=\gradF\cdot\boldu.
%\]
On $\domcs$ we introduce the rotation operator
\[
\rotm:=\left(\begin{array}{cc} 0&-1\\1&0\end{array}\right),
\]
and the surface differential operators
\begin{align*}
\curlg(u_1,u_2)^\top&:=\partial_x u_{2}-\partial_y u_{1},&
\Curlg u&:=(\partial_y u,-\partial_x u)^\top,\\
\divg(u_1,u_2)^\top&:=\partial_x u_{1}+\partial_y u_{2},&
\gradg u&:=(\partial_x u,\partial_y u)^\top,
\end{align*}
for which we note that
\begin{align*}
\rotm^2=-I,\quad
\gradg u=\rotm\Curlg u,\quad
\rotm\gradg u=-\Curlg u,\quad
\curlg (\rotm\boldu)=\divg \boldu,\quad
\divg (\rotm\boldu)=-\curlg \boldu
\end{align*}
for $\boldu=(u_1,u_2)^\top$.
%$\rotm^2=-I$, $\curlg (\rotm\boldu)=\divg \boldu$, $\gradg u=\rotm\Curlg u$ and $\rotm\gradg u=-\Curlg u$, $\divg (\rotm\boldu)=-\curlg \boldu$ for $\bv=(u_1,u_2)^\top$.
In addition, let $\tv:=-\rotm\nvt$ be the unit tangential vector on $\partial\domcs$.
Let $B_r(x_0):=\{x\colon |x-x_0|<r\}$ be the ball with radius $r>0$ and center $x_0$, where the default space (e.g., $\mathbb{C}$, $\mathbb{R}^2$, some Hilbert space $X$) for $x,x_0$ will always be clear from the context.
Further let $B_r:=B_r(0)$.
The magnetic permeability $\mu$ and the electric permeability $\epsilon$ of the material in the waveguide are assumed to be real valued, bounded and uniformly positive scalar functions which depend only on $x$ and $y$, i.e., $\mu,\epsilon,\mu^{-1},\epsilon^{-1}\in L^\infty(\domcs;\Rp)$.
Let $\epsmax:=\sup_{\xt\in\domcs}\epsilon(\xt)$,
$\epsmin:=\inf_{\xt\in\domcs}\epsilon(\xt)$,
$\mumax:=\sup_{\xt\in\domcs}\mu(\xt)$,
$\mumin:=\inf_{\xt\in\domcs}\mu(\xt)$.
Let $(r,\theta)$ be polar coordinates for $\xt\in\setR^2$ and $(r,\theta,z)$ be cylindrical coordinates for $\xf\in\setR^3$.
We assume that $\epsilon$ and $\mu$ depend only on $r$ and use the overloaded notation $\epsilon(\xf)=\epsilon(\xt)=\epsilon(r)$ with $\xf\in\setR^3,\xt\in\setR^2,r>0$, etc..
Let
$\domcs:=\{(x,y)\in\mathbb{R}^2\colon x^2+y^2<1\}$,
$\indL\in\mathbb{N}$,
$0=:r_1<\dots<r_{\indL+1}:=1$,
$\epsilon_0,\epsilon_1,\dots,\epsilon_{\indL+1}>0$,
$\mu_0,\mu_1,\dots,\mu_{\indL+1}>0$,
$\epsilon|_{(r_{l},r_{l+1})}=\epsilon_l$,
$\mu|_{(r_{l},r_{l+1})}=\mu_l$, $l=1,\dots,\indL$,
where $\epsilon_0$ and $\mu_0$ take the role of reference values to measure the deviation of $\epsilon$ and $\mu$ from homogeneous materials.
To quantify the deviation from the homogeneous material we abbreviate
\begin{align*}
\deltaem:=\|\epsilon_0-\epsilon\|_{L^\infty(\domcs)}+\|\mu_0-\mu\|_{L^\infty(\domcs)}
\quad\text{and}\quad
\deltaemt:=\|\epsilon_0\mu_0-\epsilon\mu\|_{L^\infty(\domcs)}.
\end{align*}
To simplify constants appearing in the forthcoming analysis we assume further that
\begin{align*}
	\frac{\epsilon_0}{2}\leq \epsilon \leq 2\epsilon_0, \qquad
	\frac{\mu_0}{2}\leq \mu \leq 2\mu_0.
\end{align*}
We introduce the annuli $\domcs_l:=B_{r_{l+1}}\setminus\ol{B_{r_l}}$, $l=1,\dots,\indL$ and the interfaces $\intf_l:=\partial B_{r_{l+1}}$, $l=1,\dots,\indL-1$, $\intf:=\bigcup_{l=1}^{\indL-1}\intf_l$ with the convention that the unit normal vector $\nvt$ on $\intf_l$ points to infinity.
For a piece-wise smooth function $u\colon \domcs\to\mathbb{C}$ let $\ljump u\rjump:=u|_{\domcs_{l+1}}-u|_{\domcs_l}$ on $\intf_l$.
The temporal frequency will be denoted as $\omega\in\mathbb{C}\setminus\{0\}$, the wavenumber in longitudinal direction $(0,0,1)$ as $\beta\in\setC$, $\alpha:=-i\beta$, $\nu:=\sqrt{\alpha^2+\omega^2\epsilon_0\mu_0}$, and the circumferential index as $m\in\setZ$.\\
\fbox{
\begin{minipage}{\textwidth}
\textbf{We consider $\omega$, $\indL$, $(r_l)_{l=1,\dots,\indL+1}$ and $\epsilon_0,\mu_0$ to be fixed, i.e., constants may depend on these parameters.
	On the other hand all appearing constants will be independent of $\nu$ (and $\alpha$, $\beta$), $(\epsilon_l,\mu_l)_{l=1,\dots,\indL}$, $m$ and exemptions will be indicated by indices!
}
\end{minipage}
}
\begin{definition}\label{def:cut-off}
The frequency $\omega$ is called a cut-off frequency, if one of the following two equations
\begin{align*}
-\divg(\mu^{-1}\gradg u)-\omega^2\epsilon u &=0 \quad\text{in }\domcs,\quad\phantom{\partial_{\nvt}}u=0\quad\text{on }\partial\domcs,\\
-\divg(\epsilon^{-1}\gradg u)-\omega^2\mu u &=0 \quad\text{in }\domcs,\quad \partial_{\nvt}u=0\quad\text{on }\partial\domcs,
\end{align*}
admits a non-trivial solution.
\end{definition}
Throughout the article we assume that $\omega$ is not a cut-off frequency for the homogeneous case $\epsilon=\epsilon_0$, $\mu=\mu_0$.
\\
The complex valued scalar products and norms of $L^2(\domcs)$ and $\Lspace^2(\domcs):=(L^2(\domcs))^2$ are simultaneously denoted as $\spl\cdot,\cdot\spr_{\domcs}$ and $\|\cdot\|_{\domcs}$ respectively;
and in the same fashion $\spl\cdot,\cdot\spr_\intf$ and $\|\cdot\|_\intf$ refer to the scalar products and norms of $L^2(\intf)$ and $\Lspace^2(\intf):=(L^2(\intf))^2$ respectively.
%\mhnote{$L^2_{\sigma*}$, $\Hspace_0^s(\curlg;\domcs)$}
We use a convenient notation for Sobolev spaces, e.g., $\Hspace(\curlg;\domcs)$, $\Hspace_0(\curlg;\domcs)$, $H^1(\domcs)$, $H^1_0(\domcs)$ and $H^1_{\sigma*}(\domcs):=\{u\in H^1(\domcs)\colon \spl \sigma u,1\spr_{\domcs}=0\}$ for $\sigma\in L^\infty(\domcs)$.
%$H^1_*(\domcs)$ \mhrev{be equipped with the scalar products $\spl u,u\tf\spr_{H^1_0(\domcs)}:=\spl \gradg u,\gradg u\tf \spr_{\domcs}$, $u,u\tf\in H^1_0(\domcs)$ and $\spl v,v\tf\spr_{H^1_*(\domcs)}:=\spl \gradg v,\gradg v\tf \spr_{\domcs}$, $v,v\tf\in H^1_*(\domcs)$ respectively.}
All function spaces will be considered over the scalar body $\mathbb{C}$.
For generic Hilbert spaces $X,Y$ let $\BLO(X,Y)$ be the vector spaces of bounded, linear operators from $X$ to $Y$, and set $\BLO(X):=\BLO(X,X)$.
The product $X\times Y$ of two Hilbert spaces $X,Y$ is equipped with the convenient scalar product $\spl (u,v),(u\tf,v\tf)\spr_{X\times Y}:=\spl u,u\tf\spr_{X}+\spl v,v\tf\spr_{Y}$.

\subsection{The modal eigenvalue problem}

Let $\Eff, \Hff$ denote the three-dimensional time-harmonic electric and magnetic fields, solving the Maxwells equation
\begin{subequations}\label{eq:MW-3D}
\begin{eqnarray}
\epsilon \partial_t \Eff &=& \curlF \Hff \quad\text{in }\WG,\\
\mu \partial_t \Hff &=& -\curlF \Eff \quad\text{in }\WG,
\end{eqnarray}
together with the boundary conditions on $\partial\domcs\times\Rp$
\begin{eqnarray}
\nvf\times\Eff &= 0 \quad\text{on }\partial\domcs\times\Rp,\\
\nvf\cdot\mu\Hff &= 0 \quad\text{on }\partial\domcs\times\Rp.
\end{eqnarray}
\end{subequations}
Let $\xt:=(x,y)$ and $\xf:=(x,y,z)$.
For a given angular frequency $\omega\in\mathbb{C}$ we seek solutions of the form
\[
\Eff(t,\xf)=\Ef(\xt)\exp(i\beta z)\exp(-i\omega t),
\qquad
\Hff(t,\xf)=\Hf(\xt)\exp(i\beta z)\exp(-i\omega t),
\]
where $\beta\in\mathbb{C}$ is the unknown wavenumber in longitudinal direction.
Henceforth we will work mainly with the rotated wavenumber
\[\alpha:=-i\beta\]
instead, because it will be more convenient for our analysis.
Note that if
$\Uff(\xf)=\Uf(\xt)\exp(-\alpha z)$
%$\Uff(\xf)=\Uf(\xt)\exp(i\beta z)$
with $\Uf=(\Ux,\Uy,\Uz)^\top$ and $\Ut:=(\Ux,\Uy)^\top$, then
\[
\curlF\Uff=\bpm \partial_y \Uz+\alpha \Uy\\-\partial_x \Uz-\alpha \Ux\\ \partial_x \Uy-\partial_y \Ux\epm \exp(-\alpha z)
=\bpm \Curlg \Uz-\alpha R\Ut  \\ \curlg\Ut\epm \exp(-\alpha z).
%\curlgF\Uff=\bpm \partial_y \Uz-i\beta \Uy\\-\partial_x \Uz+i\beta \Ux\\ \partial_x \Uy-\partial_y \Ux\epm \exp(i\beta z)
%=\bpm \Curlg \Uz+i\beta R\Ut  \\ \curlg\Ut\epm \exp(i\beta z).
\]
So with $\Et:=(\Ex,\Ey)^\top$, $\Ht:=(\Hx,\Hy)^\top$ \eqref{eq:MW-3D} is equivalent to
\begin{subequations}\label{eq:modal-full}
\begin{align}
%i\omega\epsilon \bpm\Et\\\Ez\epm &= -\bpm \Curlg \Hz+i\beta \rotm\Ht \\ \curlg\Ht\epm \quad\text{in }\domcs,\\
%i\omega\mu \bpm\Ht\\\Hz\epm &= \bpm \Curlg \Ez+i\beta \rotm\Et \\ \curlg\Et\epm \quad\text{in }\domcs,\\
\label{eq:modal-full-Et}
-i\omega\epsilon \Et &= \Curlg \Hz-\alpha\rotm\Ht \quad\text{in }\domcs,\\
%i\omega\epsilon \Et &= -\Curlg \Hz+i\beta\rotm\Ht \quad\text{in }\domcs,\\
\label{eq:modal-full-Ez}
-i\omega\epsilon \Ez &= \curlg\Ht \quad\text{in }\domcs,\\
\label{eq:modal-full-Ht}
-i\omega\mu \Ht &= -\Curlg \Ez+\alpha \rotm\Et \quad\text{in }\domcs,\\
%i\omega\mu \Ht &= \Curlg \Ez+i\beta \rotm\Et \quad\text{in }\domcs,\\
\label{eq:modal-full-Hz}
-i\omega\mu \Hz &= -\curlg\Et \quad\text{in }\domcs,\\
\tv\Ez &=0 \quad\text{on }\partial\domcs,\\
\tv\cdot\Et &=0 \quad\text{on }\partial\domcs,\\
%\partial_{\nvt}\Hz &=0 \quad\text{on }\partial\domcs,\\
\nvt\cdot\mu\Ht &=0 \quad\text{on }\partial\domcs.
\end{align}
\end{subequations}
Altogether we are interested in the eigenvalue problem to
\begin{align}\label{eq:evp-full}
\text{find } (\alpha,(\Et,\Ez,\Ht,\Hz)) \in\mathbb{C}\times(\Lspace^2(\domcs)\times L^2(\domcs)\times\Lspace^2(\domcs)\times L^2(\domcs))\setminus\{0\} \text{ which solves } \eqref{eq:modal-full}.
%\text{find } \alpha\in\mathbb{C} \text{ and } (\Et,\Ez,\Ht,\Hz)\neq0 \text{ which solve } \eqref{eq:modal-full}.
\end{align}
%Note that for a cross section $\domcs$ with radius $r_{\indL+1}\neq1$ a change of variable $r\to r/r_{\indL+1}$ leads to \eqref{eq:modal-full} on the unit cross section with rescaled parameters $\omega\to r_{\indL+1}\omega$ and  $\alpha\to r_{\indL+1}\alpha$.

\subsection{The tangential electric field formulation}

Our intention is to study the modal eigenvalue problem \eqref{eq:evp-full} as a perturbation of an eigenvalue problem for a compact selfadjoint operator.
To this end we first need to reformulate \eqref{eq:modal-full} into a suitable form.
The first step is to reduce the number of unknowns.
Among the many possibilities we choose to eliminate $\Ht$ and $\Hz$, see, e.g.~\cite{HallaMonk24}.
We plug \eqref{eq:modal-full-Hz}, \eqref{eq:modal-full-Ht} into \eqref{eq:modal-full-Et}; and apply $\curlg\mu^{-1}$ to \eqref{eq:modal-full-Ht} and plug in \eqref{eq:modal-full-Ez}.
Thus we obtain
\begin{subequations}
\label{eq:modal-EtEz}
\begin{align}
\Curlg\mu^{-1}\curlg\Et-\omega^2\epsilon\Et-\alpha^2\mu^{-1}\Et-\alpha\mu^{-1}\gradg\Ez &= 0 \quad\text{in }\domcs,\\
-\divg\mu^{-1}\gradg\Ez-\omega^2\epsilon\Ez-\alpha\divg(\mu^{-1}\Et) &= 0 \quad\text{in }\domcs,\\
\tv\cdot\Et &=0 \quad\text{on }\partial\domcs,\\
\Ez &=0 \quad\text{on }\partial\domcs.
\end{align}
\end{subequations}
We formulate \eqref{eq:modal-EtEz} in weak form:
\begin{align}\label{eq:modal-EtEz-weak}
\begin{aligned}
&(\Et,\Ez)\in \Hspace_0(\curlg;\domcs)\times H^1_0(\domcs) \quad\text{solves}\\
&\qquad\qquad\spl\mu^{-1}\curlg\Et,\curlg\Et\tf\spr_{\domcs}
-\omega^2\spl\epsilon\Et,\Et\tf\spr_{\domcs}
-\alpha^2\spl\mu^{-1}\Et,\Et\tf\spr_{\domcs}\\
&\qquad\qquad+\spl\mu^{-1}\gradg\Ez,\gradg\Ez\tf\spr_{\domcs}
-\omega^2\spl\epsilon\Ez,\Ez\tf\spr_{\domcs}
+\alpha\spl\mu^{-1}\Et,\gradg\Ez\tf\spr_{\domcs}
-\alpha\spl\mu^{-1}\gradg\Ez,\Et\tf\spr_{\domcs} = 0\\
&\text{for all } (\Et\tf,\Ez\tf)\in \Hspace_0(\curlg;\domcs)\times H^1_0(\domcs).
\end{aligned}
\end{align}
Since the space $\Hspace_0(\curlg;\domcs)$ does not admit a compact embedding into $\bL^2(\domcs)$ the current setting of function spaces is not suitable yet for our intended analysis.
Hence we follow \cite{HallaMonk24} and consider the topological decomposition (which is orthogonal in an equivalent weighted scalar product)
\begin{align*}
	\Hspace_0(\curlg;\domcs)&=\Vspace\oplus^\calT\gradg H^1_0(\domcs),\\
\Vspace&:=\Hcd:=\{\Et\in \Hspace_0(\curlg;\domcs)\colon\divg(\mu^{-1}\Et)=0 \text{ in }\domcs\}.
\end{align*}
With $\Et=\bv+\gradg \tilde w$, $\Et\tf=\bv\tf+\gradg \tilde w\tf$, $\bv,\bv\tf\in \Vspace, \tilde w,\tilde w\tf\in H^1_0(\domcs)$ \eqref{eq:modal-EtEz-weak} reads
\begin{align}\label{eq:modal-EtwEz-weak}
\begin{aligned}
&(\bv,\tilde w,\Ez)\in \Vspace % \Hspace_0(\curlg,(\divg\mu^{-1})^0;\domcs)
\times H^1_0(\domcs) \times H^1_0(\domcs) \quad\text{solves}\\
&\hspace{10mm}\spl\mu^{-1}\curlg\bv,\curlg\bv\tf\spr_{\domcs}
-\omega^2\spl\epsilon(\bv+\gradg \tilde w),(\bv\tf+\gradg \tilde w\tf)\spr_{\domcs}\\
&\hspace{10mm}-\alpha^2\spl\mu^{-1}\bv,\bv\tf\spr_{\domcs}
-\alpha^2\spl\mu^{-1}\gradg \tilde w,\gradg \tilde w\tf\spr_{\domcs}\\
&\hspace{10mm}+\spl\mu^{-1}\gradg\Ez,\gradg\Ez\tf\spr_{\domcs}
-\omega^2\spl\epsilon\Ez,\Ez\tf\spr_{\domcs}
+\alpha\spl\mu^{-1}\gradg \tilde w,\gradg\Ez\tf\spr_{\domcs}
-\alpha\spl\mu^{-1}\gradg\Ez,\gradg \tilde w\tf\spr_{\domcs} = 0\\
&\text{for all } (\bv\tf,\tilde w\tf,\Ez\tf)\in \Vspace
%\Hspace_0(\curlg,(\divg\mu^{-1})^0;\domcs)
\times H^1_0(\domcs) \times H^1_0(\domcs).
\end{aligned}
\end{align}
Note that we will later change to the rescaled variables $w=\omega\tilde w,w\tf=\omega\tilde w\tf$.
Eq.~\eqref{eq:modal-EtwEz-weak} is quadratic in the eigenvalue parameter $\alpha$.
Hence the next step is to regain a problem, which is affine in $\alpha^2$.
To this end let
\begin{align*}
\Dcurl &\in\BLO(\Vspace,L^2(\domcs)),\quad \Dcurl\Et:=\curlg\Et,\\
\mathbf{M}_\sigma&\in\BLO(\bL^2(\domcs)),\quad \mathbf{M}_\sigma\Et:=\sigma\Et,\quad\text{ for }\sigma\in L^\infty(\domcs),\\
\embLVec&\in\BLO(\Vspace,\bL^2(\domcs)),\quad \embLVec\Et:=\Et,\\
\DGrad&\in\BLO(H^1_0(\domcs),\bL^2(\domcs)),\quad \DGrad\Ez:=\gradg\Ez,\\
\embL&\in\BLO(H^1_0(\domcs),L^2(\domcs)),\quad \embL\Ez:=\Ez,\\
M_\sigma&\in\BLO(L^2(\domcs)),\quad M_\sigma\Ez:=\sigma\Ez,\quad\text{ for }\sigma\in L^\infty(\domcs).
%\MepsVec&\in\BLO(\bL^2(\domcs)),\quad \MepsVec\Et:=\epsilon\Et,\\
%\MmuiVec&\in\BLO(\bL^2(\domcs)),\quad \MmuiVec\Et:=\mu^{-1}\Et,\\
%\Meps&\in\BLO(L^2(\domcs)),\quad \Meps\Ez:=\epsilon\Ez,\\
%\Mmui&\in\BLO(L^2(\domcs)),\quad \Mmui\Ez:=\mu^{-1}\Ez.
\end{align*}
Then \eqref{eq:modal-EtwEz-weak} reads in operator form:
\begin{align}\label{eq:modal-EtwEz-operator}
\begin{aligned}
\bpm \mathbf{0}\\0\\0 \epm &=
\left[\bpm
\Dcurl^* \Mmui \Dcurl-\omega^2 \embLVec^* \MepsVec \embLVec
%-\alpha^2 \embLVec^* \MmuiVec \embLVec
& -\omega^2 \embLVec^*\MepsVec\DGrad
& 0\\
-\omega^2 \DGrad^*\MepsVec\embLVec
& -\omega^2\DGrad^* \MepsVec \DGrad %-\alpha^2\DGrad^* \MmuiVec \DGrad
& 0\\%-\alpha\DGrad^* \MmuiVec \DGrad\\
0 & 0%\alpha\DGrad^* \MmuiVec \DGrad
& \DGrad^* \MmuiVec \DGrad -\omega^2 \embL^* \Meps \embL
\epm \right.\\
&\hspace{-5mm} \left.+\alpha
\bpm
0%\Dcurl^* \Mmui \Dcurl-\omega^2 \embLVec^* \MepsVec \embLVec
%-\alpha^2 \embLVec^* \MmuiVec \embLVec
& 0%-\omega^2 \embLVec^*\MepsVec\DGrad
& 0\\
0%-\omega^2 \DGrad^*\MepsVec\embLVec
& 0%-\omega^2\DGrad^* \MepsVec \DGrad %-\alpha^2\DGrad^* \MmuiVec \DGrad
& -\DGrad^* \MmuiVec \DGrad\\
%& -\alpha\DGrad^* \MmuiVec \DGrad\\
0 & \DGrad^* \MmuiVec \DGrad
%0 & \alpha\DGrad^* \MmuiVec \DGrad
& 0%\DGrad^* \MmuiVec \DGrad -\omega^2 \embL^* \Meps \embL
\epm 
-\alpha^2
\bpm
%\Dcurl^* \Mmui \Dcurl-\omega^2 \embLVec^* \MepsVec \embLVec
\alpha^2 \embLVec^* \MmuiVec \embLVec
& 0%-\omega^2 \embLVec^*\MepsVec\DGrad
& 0\\
%-\omega^2 \DGrad^*\MepsVec\embLVec
& %-\omega^2\DGrad^* \MepsVec \DGrad %-\alpha^2
\DGrad^* \MmuiVec \DGrad
& 0\\%-\alpha\DGrad^* \MmuiVec \DGrad\\
0 & 0%\alpha\DGrad^* \MmuiVec \DGrad
& 0%\DGrad^* \MmuiVec \DGrad -\omega^2 \embL^* \Meps \embL
\epm \right]
\bpm \bv\\ \tilde w\\\Ez \epm.
\end{aligned}
\end{align}
%
%\begin{align}\label{eq:modal-EtwEz-operator}
%\bpm
%\Dcurl^* \Mmui \Dcurl-\omega^2 \embLVec^* \MepsVec \embLVec
%-\alpha^2 \embLVec^* \MmuiVec \embLVec
%& -\omega^2 \embLVec^*\MepsVec\DGrad
%& 0\\
%-\omega^2 \DGrad^*\MepsVec\embLVec
%& -\omega^2\DGrad^* \MepsVec \DGrad -\alpha^2\DGrad^* \MmuiVec \DGrad
%& -\alpha\DGrad^* \MmuiVec \DGrad\\
%0 & \alpha\DGrad^* \MmuiVec \DGrad
%& \DGrad^* \MmuiVec \DGrad -\omega^2 \embL^* \Meps \embL
%\epm
%\bpm \bv\\w\\\Ez \epm = \bpm \mathbf{0}\\0\\0 \epm.
%\end{align}
\begin{lemma}\label{lem:cutoff}
Let the assumptions formulated in \Cref{subsec:setting} be satisfied.
Then there exists $\delta_0>0$ such that $\omega$ is not a cut-off frequency (see \Cref{def:cut-off}) for each $\epsilon,\mu$ that satisfy $\deltaem<\delta_0$.
\end{lemma}
\begin{proof}
Follows from basic perturbation theory.
\end{proof}
Justified by \Cref{lem:cutoff} we assume henceforth that $\omega$ is not a cut-off frequency.
Thus we can build in \eqref{eq:modal-EtwEz-operator} the Schur complement with respect to $\Ez$ and \eqref{eq:modal-EtwEz-operator} becomes
%\begin{align}\label{eq:modal-Etw-operator}
%\bpm
%\Dcurl^* \Mmui \Dcurl-\omega^2 \embLVec^* \MepsVec \embLVec
%-\alpha^2 \embLVec^* \MmuiVec \embLVec
%& -\omega^2 \embLVec^*\MepsVec\DGrad\\
%-\omega^2 \DGrad^*\MepsVec\embLVec
%& -\omega^2\DGrad^* \MepsVec \DGrad -\alpha^2\DGrad^* \MmuiVec \DGrad
%+\alpha^2 \DGrad^* \MmuiVec \DGrad (\DGrad^* \MmuiVec \DGrad -\omega^2 \embL^* \Meps \embL)^{-1}
%\DGrad^* \MmuiVec \DGrad
%\epm
%\bpm \bv\\w \epm = \bpm \mathbf{0}\\0 \epm.
%\end{align}
%
\begin{align}\label{eq:modal-Etw-operator}
\begin{aligned}
&\bpm
\Dcurl^* \Mmui \Dcurl-\omega^2 \embLVec^* \MepsVec \embLVec
& -\omega^2 \embLVec^*\MepsVec\DGrad\\
-\omega^2 \DGrad^*\MepsVec\embLVec
& -\omega^2\DGrad^* \MepsVec \DGrad
\epm
\bpm \bv\\ \tilde w \epm=\\
&\alpha^2 \bpm
\embLVec^* \MmuiVec \embLVec &\\
& \DGrad^* \MmuiVec \DGrad-\DGrad^* \MmuiVec \DGrad (\DGrad^* \MmuiVec \DGrad -\omega^2 \embL^* \Meps \embL)^{-1}
\DGrad^* \MmuiVec \DGrad
\epm
\bpm \bv\\ \tilde w \epm.
\end{aligned}
\end{align}
Let us set
\begin{align*}
\Xspace:=\Vspace \times H^1_0(\domcs),
\end{align*}
which will be our working Hilbert space.
The convenient form of an eigenvalue problem admits $L^2$-terms as operator coefficients of $\alpha^2$.
While the first diagonal entry in the right hand-side of \eqref{eq:modal-Etw-operator} already has a suitable form, the actual properties of the second diagonal entry are not obvious yet and shall be worked out next.
We abbreviate $\Aop_{\mu^{-1}}:=\DGrad^* \MmuiVec \DGrad$
and $\Kop_\epsilon=\omega^2 \embL^* \Meps \embL$.
Applying the formula $\Aop-\Aop(\Aop-\Kop)^{-1}\Aop=\Aop[I-(\Aop-\Kop)^{-1}\Aop]=-\Aop(\Aop-\Kop)^{-1}\Kop$ to $\Aop=\Aop_{\mu^{-1}}$ and $\Kop=\Kop_\epsilon$ the second diagonal entry in the right hand-side of \eqref{eq:modal-Etw-operator} simplifies to
\begin{align*}
-\Aop_{\mu^{-1}} (\Aop_{\mu^{-1}}-\Kop_\epsilon)^{-1} \Kop_\epsilon=
-\DGrad^* \MmuiVec \DGrad
(\DGrad^* \MmuiVec \DGrad-\omega^2 \embL^* \Meps \embL)^{-1}
\omega^2 \embL^* \Meps \embL.
\end{align*}
Note that also $\Aop-\Aop(\Aop-\Kop)^{-1}\Aop=-\Kop-\Kop(\Aop-\Kop)^{-1}\Kop$, which is the expression used in \cite{HallaMonk24}.
However, at this point we deviate from \cite{HallaMonk24} and multiply \eqref{eq:modal-Etw-operator} with 
\begin{align*}
\diag\Big(I,-(\Aop_{\mu^{-1}} (\Aop_{\mu^{-1}}-\Kop_\epsilon)^{-1})^{-1}\Big)
=\diag\Big(I,-(\Aop_{\mu^{-1}}-\Kop_\epsilon) \Aop_{\mu^{-1}}^{-1}\Big)
=\diag\Big(I,-I +\Kop_\epsilon \Aop_{\mu^{-1}}^{-1}\Big)
\end{align*}
to obtain
\begin{align}\label{eq:modal-Etw-operator-II}
(\tilde\Aop+\tilde\Bop) \bpm \bv\\ \tilde w \epm &= \alpha^2 \tilde\Kop \bpm \bv\\ \tilde w \epm
\end{align}
where
\begin{subequations}
\begin{align}
%&A&=\bpm
%\Dcurl^* \Mmui \Dcurl-\omega^2 \embLVec^* \MepsVec \embLVec
%& -\omega^2 \embLVec^*\MepsVec\DGrad\\
%\omega^2 \DGrad^*\MepsVec\embLVec
%-\omega^4 \embL^* \Meps \embL (\DGrad^* \MmuiVec \DGrad)^{-1} \DGrad^*\MepsVec\embLVec
%& \omega^2\DGrad^* \MepsVec \DGrad
%-\omega^4 \embL^* \Meps \embL (\DGrad^* \MmuiVec \DGrad)^{-1} \DGrad^* \MepsVec \DGrad
%\epm,\\
%
\tilde\Aop&:=\bpm
\Dcurl^* \Mmui \Dcurl-\omega^2\epsilon_0\mu_0 \embLVec^* \MmuiVec \embLVec
& \\
& \omega^2\DGrad^* \MepsVec \DGrad
-\omega^4 \epsilon_0\mu_0 \embL^* \Meps \embL 
\epm,\\
\tilde\Bop&:=\bpm
-\omega^2 \embLVec^* (\MepsVec-\epsilon_0\mu_0\MmuiVec) \embLVec
& -\omega^2 \embLVec^*\MepsVec\DGrad\\
\omega^2 \DGrad^*\MepsVec\embLVec
-\omega^4 \embL^* \Meps \embL (\DGrad^* \MmuiVec \DGrad)^{-1} \DGrad^*\MepsVec\embLVec
& \tilde\Bop_{22}
\epm,\\
\tilde\Bop_{22}&:=-\omega^4 \embL^* \Meps \embL (\DGrad^* \MmuiVec \DGrad)^{-1} \big( \DGrad^* \MepsVec \DGrad -\epsilon_0\mu_0 \DGrad^* \MmuiVec \DGrad \big),\\
\tilde\Kop&:=\bpm
\embLVec^* \MmuiVec \embLVec &\\
& \omega^2 \embL^* \Meps \embL\epm.
\end{align}
\end{subequations}
Now we switch to $w:=\omega \tilde w$, multiply \eqref{eq:modal-Etw-operator-II} with $\diag(I,\omega^{-1} I)$ and perform an eigenvalue shift $\nu^2:=\alpha^2+\omega^2\epsilon_0\mu_0$ to obtain
the eigenvalue problem to
\begin{align}\label{eq:evp}
\text{find} \quad \bigg(\nu,\bpm \bv\\w \epm\bigg)\in\setC\times \Xspace\setminus\{0\}
\quad\text{such that}\quad
(\Aop_{\epsilon,\mu}+\Bop_{\epsilon,\mu}) \bpm \bv\\w \epm &= \nu^2  \Kop_{\epsilon,\mu} \bpm \bv\\w \epm
\end{align}
with
\begin{subequations}
\begin{align}
%&A&=\bpm
%\Dcurl^* \Mmui \Dcurl-\omega^2 \embLVec^* \MepsVec \embLVec
%& -\omega^2 \embLVec^*\MepsVec\DGrad\\
%\omega^2 \DGrad^*\MepsVec\embLVec
%-\omega^4 \embL^* \Meps \embL (\DGrad^* \MmuiVec \DGrad)^{-1} \DGrad^*\MepsVec\embLVec
%& \omega^2\DGrad^* \MepsVec \DGrad
%-\omega^4 \embL^* \Meps \embL (\DGrad^* \MmuiVec \DGrad)^{-1} \DGrad^* \MepsVec \DGrad
%\epm,\\
%
\Aop_{\epsilon,\mu}&:=\bpm
\Dcurl^* \Mmui \Dcurl
& \\
& \DGrad^* \MepsVec \DGrad
\epm,\\
\Bop_{\epsilon,\mu}&:=\bpm
-\omega^2 \embLVec^* (\MepsVec-\epsilon_0\mu_0\MmuiVec) \embLVec
& -\omega \embLVec^*\MepsVec\DGrad\\
\omega \DGrad^*\MepsVec\embLVec
-\omega^3 \embL^* \Meps \embL (\DGrad^* \MmuiVec \DGrad)^{-1} \DGrad^*\MepsVec\embLVec
& \Bop_{22}
\epm,\\
\Bop_{\epsilon,\mu,22}&:=-\omega^2 \embL^* \Meps \embL (\DGrad^* \MmuiVec \DGrad)^{-1} \big( \DGrad^* \MepsVec \DGrad -\epsilon_0\mu_0 \DGrad^* \MmuiVec \DGrad \big),\\
\Kop_{\epsilon,\mu}&:=\bpm
\embLVec^* \MmuiVec \embLVec &\\
& \embL^* \Meps \embL\epm.
\end{align}
\end{subequations}
Note that in the above notation we made the dependency of the operators on the coefficients $\epsilon,\mu$ explicit by means of indices.
In cases, where there is no ambiguity we abbreviate
\begin{align*}
\Aop:=\Aop_{\epsilon,\mu},\qquad
\Kop:=\Kop_{\epsilon,\mu},\qquad
\Bop:=\Bop_{\epsilon,\mu},\qquad
\Bop_{22}:=\Bop_{\epsilon,\mu,22}.
\end{align*}
We collect our transformations in the following lemma.
\begin{lemma}\label{lem:evp-equi}
	If $(\alpha,(\Et,\Ez,\Ht,\Hz))\in \setC\times(\Lspace^2(\domcs)\times L^2(\domcs)\times \Lspace^2(\domcs)\times L^2(\domcs))\setminus\{0\}$ is a solution to \eqref{eq:evp-full}, then $(\nu:=\pm\sqrt{\alpha^2+\omega^2\epsilon_0\mu_0},(\bv,w)^\top)\in \setC\times \Xspace\setminus\{0\}$ with $\bv+\omega^{-1}\nabla w=\Et$ is a solution to \eqref{eq:evp}.
Vice-versa, if $(\nu,(\bv,w)^\top)\in \setC\times \Xspace\setminus\{0\}$ is a solution to \eqref{eq:evp}, then there exist two linear independent solutions $(\alpha:=\pm\sqrt{\nu^2-\omega^2\epsilon_0\mu_0},(\Et,\Ez^{(\pm)},\Ht^{(\pm)},\Hz^{(\pm)}))\in \setC\times(\Lspace^2(\domcs)\times L^2(\domcs)\times \Lspace^2(\domcs)\times L^2(\domcs))\setminus\{0\}$ to \eqref{eq:evp-full} with $\Et=\bv+\omega^{-1}\nabla w$ and $(\Ez^{(\pm)},\Ht^{(\pm)},\Hz^{(\pm)})$ being uniquely determined by $\alpha$ and $\Et$.
\end{lemma}
\begin{proof}
Follows with basic algebraic manipulations.
%Let $(\alpha,(\Et,\Ez,\Ht,\Hz))\in \setC\times(\Lspace^2(\domcs)\times L^2(\domcs)\times \Lspace^2(\domcs)\times L^2(\domcs))\setminus\{0\}$ be a solution to \eqref{eq:evp-full}.
%Then it follows from \eqref{eq:modal-full} that indeed
%$(\Et,\Ez,\Ht,\Hz))\in \Hspace_0(\curlg;\domcs)\times H^1_0(\domcs)\times \Hspace(\curlg;\domcs)\times H_{\mu*}^1(\domcs)$.
%Since $\Et=0$ and \eqref{eq:modal-full} would imply $\Ez=0,\Ht=0,\Hz=0$, it also follows that $\Et\neq0$.
%The previous analysis in this section then shows that $(\lambda^2:=\alpha^2+\omega^2\epsilon_0\mu_0,\Et)$ solves \eqref{eq:evp}.
%Vice-versa, let $(\lambda^2,\Et)\in \setC\times \Hspace_0(\curlg;\domcs)\setminus\{0\}$ be a solution to \eqref{eq:evp}.
\end{proof}
The structure of \eqref{eq:evp} now invites us to analyze \eqref{eq:evp} as a perturbation of the selfadjoint eigenvalue problem
\begin{align}\label{eq:evp-sa}
\text{find} \quad \bigg(\nu,\bpm \bv\\w \epm\bigg)\in\setC\times \Xspace\setminus\{0\}
\quad\text{such that}\quad
\Aop_{\epsilon,\mu} \bpm \bv\\w \epm &= \nu^2 \Kop_{\epsilon,\mu} \bpm \bv\\w \epm.
\end{align}
We are going to perform such an analysis in \Cref{sec:perturbation-nsa} and as a preparation we state the upcoming \Cref{lem:Bbound}.
However, this analysis requires certain strong properties of the eigenfunctions and eigenvalue gaps of the unperturbed problem \eqref{eq:evp-sa}.
To this end we consider \eqref{eq:evp-sa} as a perturbation of the homogeneous problem [\eqref{eq:evp-sa} with $\epsilon=\epsilon_0$, $\mu=\mu_0$] and perform this analysis in \Cref{sec:perturbation-sa}.
Since this perturbation is not small in any subordinate sense we will exploit certain properties of Bessel functions, which we are going to study in \Cref{sec:bessel}.\\

In preparation of \Cref{lem:Bbound} we introduce certain constants.
Let $\Cdir^2>0$ be the coercivity constant of the Dirichlet Laplacian:
\begin{align*}
\|\gradg w\|_{\bL^2(\domcs)} \geq \Cdir \|w\|_{H^1(\domcs)}
\quad\text{for all }w\in H^1_0(\domcs).
\end{align*}
Let $\Ctrhi>0$ be a trace continuity constant in $H^1(\domcs)$:
\begin{align*}
\|w\|_{L^2(\intf)} \leq \Ctrhi \|w\|_{H^1(\domcs)}
\quad\text{for all } w\in H^1_0(\domcs);
\end{align*}
and $\Ctrhii>0$ be a trace continuity constant in a broken space:
\begin{align*}
\|\gradg w\|_{L^2(\intf)} \leq \Ctrhii (\|w\|_{H^1(\domcs)}+(\sum_{l=1}^{\indL} |v|_{H^2(\domcs_l)})^{1/2})
\end{align*}
for all $w\in \{w \in H^1(\domcs)\colon w|_{\domcs_l}\in H^2(\domcs_l)\,\, \forall l=1,\dots,\indL\}$.
\begin{lemma}\label{lem:Bbound}
It holds that
\begin{align*}
\Big|\Big\spl \Bop_{\epsilon,\mu} \bpm\bv\\w\epm, \bpm\bv\tf\\w\tf\epm \Big\spr_{\Xspace}
\Big|&\leq
\CBop\deltaemt \Big(
\|\bv\|_{\domcs} \|\bv\tf\|_{\domcs}
+\|w\|_{\intf} \|\nvt\cdot\mu^{-1}\bv\tf\|_{\intf}\\
&
+\|\nvt\cdot\mu^{-1}\bv\|_{\intf} \|w\tf\|_{\intf}
+\|\nvt\cdot\mu^{-1}\bv\|_{\intf} \|w\tf\|_{\domcs}
+\|w\|_{\domcs} \|w\tf\|_{\domcs}
+\|w\|_{\intf} \|w\tf\|_{\domcs} \Big).
\end{align*}
for all $\bv,\bv\tf\in\Vspace$, $w,w\tf\in H^1_0(\domcs)$
with
\[
\CBop:=\max\left\{
2|\omega|^2\mu_0^{-1},
2|\omega|,
4|\omega|^3 \epsilon_0\mu_0 \Cdir^{-1} \Ctrhi,
2|\omega|^2 \epsilon_0,
2|\omega|^2 \Ctrhii (2\epsilon_0\mu_0\Cdir^{-1}+4\epsilon_0\mu_0) \right\}.
\]
\end{lemma}
\begin{proof}
We apply the triangle inequality and estimate step by step each summand in the expansion of
$\Big\spl \Bop_{\epsilon,\mu} \bpm\bv\\w\epm, \bpm\bv\tf\\w\tf\epm \Big\spr_{\Xspace}$.
To start with
\begin{align*}
|\omega|^2 |\spl \embLVec^* (\MepsVec-\epsilon_0\mu_0\MmuiVec) \embLVec \bv,\bv\tf\spr_{\Vspace}|
&=|\omega|^2 |\spl \mu^{-1}(\epsilon\mu-\epsilon_0\mu_0) \bv,\bv\tf\spr_{\bL^2(\domcs)}|\\
&\leq 2|\omega|^2\mu_0^{-1} \deltaemt \|\bv\|_{\domcs} \|\bv\tf\|_{\domcs}.
\end{align*}
For the next term we exploit that $\divg(\mu^{-1}\bv)=0$ for $\bv\in\Vspace$ and that $\epsilon,\mu$ are piece-wise constant:
\begin{align*}
|\omega \spl \embLVec^*\MepsVec\DGrad w,\bv\tf \spr_{\Vspace}|
&=|\omega \spl \epsilon \gradg w,\bv\tf \spr_{\bL^2(\domcs)}|\\
&=|\omega \spl \epsilon\mu \gradg w,\mu^{-1}\bv\tf \spr_{\domcs}|\\
&=|\omega \spl (\epsilon\mu-\epsilon_0\mu_0) \gradg w,\mu^{-1}\bv\tf \spr_{\domcs}|\\
&=|\omega \spl \ljump\epsilon\mu-\epsilon_0\mu_0\rjump w,\nvt\cdot\mu^{-1}\bv\tf \spr_{\intf}|\\
&=2|\omega| \deltaemt \|w\|_{\intf} \|\nvt\cdot\mu^{-1}\bv\tf\|_{\intf}.
\end{align*}
Like-wise $|\omega \spl \DGrad^*\MepsVec\embLVec\bv, w\tf \spr_{\Vspace}|\leq 2|\omega| \deltaemt \|\nvt\cdot\mu^{-1}\bv\|_{\intf} \|w\tf\|_{\intf}$.
To treat the remaining terms we first analyze $(\DGrad^* \MmuiVec \DGrad)^{-1} \embL^* \Meps \embL$.
Since
\[
H^1_0(\domcs)\ni u:=(\DGrad^* \MmuiVec \DGrad)^{-1} \embL^* \Meps \embL w
\]
solves $-\divg(\mu^{-1}\gradg u)=\epsilon w$ in $\domcs$, it follows that $\|u\|_{H^1(\domcs)}\leq 2\epsilon_0\mu_0\Cdir^{-1} \|w\|_{L^2(\domcs)}$.
Further, we have that
\begin{align*}
\spl \mu \divg(\mu^{-1}\gradg u),\divg(\mu^{-1}\gradg u) \spr_{\domcs}
&=\spl \mu \partial_r (r^{-1}\mu^{-1} \partial_r u), \partial_r (r^{-1}\mu^{-1} \partial_r u) \spr_{\domcs}
+\spl \mu^{-1} r^{-1} \partial_{\theta}^2 u, r^{-1} \partial_{\theta}^2 u \spr_{\domcs}\\
&+\spl \partial_r (r^{-1}\mu^{-1} \partial_r u), r^{-1} \partial_{\theta}^2 u \spr_{\domcs}
+\spl r^{-1} \partial_{\theta}^2 u, \partial_r (r^{-1}\mu^{-1} \partial_r u) \spr_{\domcs}
\end{align*}
and
$\spl \partial_r (r^{-1}\mu^{-1} \partial_r u), r^{-1} \partial_{\theta}^2 u \spr_{\domcs}
+\spl r^{-1} \partial_{\theta}^2 u, \partial_r (r^{-1}\mu^{-1} \partial_r u) \spr_{\domcs}
=2\spl \mu^{-1}r^{-1} \partial_r \partial_{\theta} u, r^{-1} \partial_r \partial_{\theta} u \spr_{\domcs}$.
Thus $\sum_{l=1}^{\indL} |u|_{H^2(\domcs_l)}^2 \leq 16\epsilon_0^2\mu_0^2 \|w\|_{L(^2(\domcs)}^2$.
Now we are read to estimate
\begin{align*}
|\omega^3 \embL^* \Meps \embL (\DGrad^* \MmuiVec \DGrad)^{-1} & \DGrad^*\MepsVec\embLVec\bv,w\tf\spr_{H^1_0(\domcs)}|\\
&=|\omega^3\MepsVec\embLVec\bv, \DGrad (\DGrad^* \MmuiVec \DGrad)^{-1} \embL^* \Meps \embL w\tf\spr_{\bL^2(\domcs)}|\\
&=|\omega^3 \epsilon\bv, \gradg (\DGrad^* \MmuiVec \DGrad)^{-1} \embL^* \Meps \embL w\tf\spr_{\domcs}|\\
&=|\omega^3 (\epsilon\mu-\epsilon_0\mu_0) \mu^{-1}\bv, \gradg (\DGrad^* \MmuiVec \DGrad)^{-1} \embL^* \Meps \embL w\tf\spr_{\domcs}|\\
&\leq 2|\omega|^3 \deltaemt \|\nvt\cdot\mu^{-1}\bv\|_{\intf} \|(\DGrad^* \MmuiVec \DGrad)^{-1} \embL^* \Meps \embL w\tf\|_{\intf}\\
&\leq 4|\omega|^3 \epsilon_0\mu_0 \Cdir^{-1} \Ctrhi \deltaemt \|\nvt\cdot\mu^{-1}\bv\|_{\intf} \|w\tf\|_{\domcs},
\end{align*}
where again we have used that $\divg(\mu^{-1}\bv)=0$ and that $\epsilon,\mu$ are piece-wise constant.
It remains to estimate the term stemming from $\Bop_{\epsilon,\mu,22}$.
Again, it follows with integration by parts that
\begin{align*}
|\omega^2 &\embL^* \Meps \embL (\DGrad^* \MmuiVec \DGrad)^{-1}
(\DGrad^* \MepsVec \DGrad -\epsilon_0\mu_0 \DGrad^* \MmuiVec \DGrad )w, w\tf \spr_{H^1_0(\domcs)}|\\
&= |\omega|^2 |\spl ( \MepsVec \DGrad -\epsilon_0\mu_0 \MmuiVec \DGrad) w, \DGrad (\DGrad^* \MmuiVec \DGrad)^{-1}\embL^* \Meps \embL w\tf \spr_{\bL^2(\domcs)}|\\
&= |\omega|^2 |\spl (\epsilon\mu-\epsilon_0\mu_0) \gradg w, \mu^{-1}\gradg (\DGrad^* \MmuiVec \DGrad)^{-1}\embL^* \Meps \embL w\tf \spr_{\domcs}|\\
&\leq |\omega|^2 |\spl (\epsilon\mu-\epsilon_0\mu_0) w, \epsilon w\tf \spr_{\domcs}|
+|\omega|^2 |\spl \ljump \epsilon\mu-\epsilon_0\mu_0 \rjump w,\nvt\cdot \mu^{-1}\gradg (\DGrad^* \MmuiVec \DGrad)^{-1}\embL^* \Meps \embL w\tf \spr_{\intf}\\
&\leq 2|\omega|^2 \epsilon_0 \deltaemt \|w\|_{\domcs} \|w\tf\|_{\domcs}
+2|\omega|^2 \Ctrhii (2\epsilon_0\mu_0\Cdir^{-1}+4\epsilon_0\mu_0) \deltaemt  \|w\|_{\intf} \|w\tf\|_{\domcs}.
\end{align*}
\end{proof}

\subsection{Coaxial separation of modes}\label{subsec:separation}

%For $\mu\Curlg v\in\Vspace$, $w\in H^1_0(\domcs)$, $\tilde v\in H^1_{\mu*}(\domcs)$ let us introduce the Fourier coefficients
%\begin{align*}
%u(re^{i\theta m})=\sum_{m\in\setZ} u_m(r) (2\pi)^{-1/2} e^{im\theta},\quad
%u_m(r):=(2\pi)^{-1/2}\int_0^{2\pi} u(re^{i\theta m}) e^{-im\theta}\,d\theta,\quad
%u \in\{v,w,\tilde v\}.
%\end{align*}
%%\begin{align*}
%%v(re^{i\theta m})=\sum_{m\in\setZ} v_m(r) (2\pi)^{-1/2} e^{im\theta},\quad
%%v_m(r):=(2\pi)^{-1/2}\int_0^{2\pi} v(re^{i\theta m}) e^{-im\theta}\,d\theta,\\
%%w(re^{i\theta m})=\sum_{m\in\setZ}w_m(r) (2\pi)^{-1/2} e^{im\theta},\quad
%%w_m(r):=(2\pi)^{-1/2}\int_0^{2\pi} w(re^{i\theta m}) e^{-im\theta}\,d\theta,\\
%%\tilde v(re^{i\theta m})=\sum_{m\in\setZ} \tilde v_m(r) (2\pi)^{-1/2} e^{im\theta},\quad
%%\tilde v_m(r):=(2\pi)^{-1/2}\int_0^{2\pi} \tilde v(re^{i\theta m}) e^{-im\theta}\,d\theta.
%%\end{align*}
For any scalar function $u\in L^2(\domcs)$ we introduce the Fourier coefficients 
\begin{align*}
u(re^{i\theta m})=\sum_{m\in\setZ} u_m(r) (2\pi)^{-1/2} e^{im\theta},\quad
u_m(r):=(2\pi)^{-1/2}\int_0^{2\pi} u(re^{i\theta m}) e^{-im\theta}\,d\theta.
\end{align*}
For any scalar function space $Y\subset L^2(\domcs)$ let
\begin{align*}
	Y_m:=\{u\in Y\colon u_{m'}=0 \,\,\forall m'\in\setZ\setminus\{m\}\}.
\end{align*}
For any vectorial function space $\boldY$ characterized by a potential space $\boldY=\{\sigma\Curlg u\colon u\in Y\}$, $\sigma=\mu,\epsilon$ let
\begin{align*}
	\boldY_m:=\{\sigma\Curlg u\colon u\in Y_m\}.
\end{align*}
Note that
\begin{align*}
	\Vspace=\{\mu\Curlg u\colon u\in \Vpotspace\},\qquad
	\Vpotspace:=\{u\in H^1(\domcs)\colon \mu\Curlg u\in\Hspace_0(\curlg;\domcs)\}.
\end{align*}
Further set
\begin{align*}
	V:=H^1_{\mu^{-1}*}(\domcs) \quad\text{and}\quad W:=H^1_0(\domcs),
\end{align*}
and in particular
\begin{align*}
	\Xspace_m:=\Vspace_m\times\Wm.
\end{align*}
%Since $\domcs$ is simply connected it holds that
%\begin{align*}
%	\Vspace&=\{\mu\Curlg v\colon v\in\Vpotspace\},\\
%	\Vpotspace&:=\{v\in H^1_{\mu*}(\domcs),\,\,
%	\curlg(\mu\Curlg v)\in L^2(\domcs),\,\,
%	\mu\partial_{\nvt}v=0 \text{ on }\partial\domcs\},\\
%	\|v\|_{\Vpotspace}^2&:=\|v\|^2_{H^1(\domcs)}+\|\curlg(\mu\Curlg v)\|^2_{\domcs}.
%\end{align*}
%Thence let
%\begin{align*}
%%	\Vspace_m&:=\{\mu\Curlg v\colon v\in H^1_{\mu*}(\domcs),\,
%%	\curlg(\mu\Curlg v)\in L^2(\domcs),\,
%%	\mu\partial_{\nvt}v=0 \text{ on }\partial\domcs,\,
%%	v_{m'}=0 \,\,\forall m'\in\setZ\setminus\{m\}\},\\
%%\Vspace_m&:=\{\mu\Curlg v\colon v\in\Vpotspace,\,\,
%%	v_{m'}=0 \,\,\forall m'\in\setZ\setminus\{m\}\},\\
%\Vspace_m&:=\{\mu\Curlg v_m\colon v_m\in\Vpotspace_m\},\\
%\Vpotspace_m&:=\{v\in\Vpotspace\colon v_{m'}=0 \,\,\forall m'\in\setZ\setminus\{m\}\},\qquad\\
%\Wm&:=\{w\in H^1_0(\domcs)\colon w_{m'}=0 \,\,\forall m'\in\setZ\setminus\{m\}\},\qquad\\
%\Vm&:=\{\tilde v\in H^1_{\mu*}(\domcs)\colon \tilde v_{m'}=0 \,\,\forall m'\in\setZ\setminus\{m\}\},\qquad\\
%\Xspace_m&:=\Vspace_m\times\Wm.
%\end{align*}
Since $\epsilon$ and $\mu$ depend only on the radial variable it follows that the spaces  $\Xspace_m, m\in\setZ$ are invariant with respect to $\Aop,\Bop,\Kop$ and the eigenvalue problems \eqref{eq:evp} and \eqref{eq:evp-sa} decouple into the following families ($m\in\setZ$) of problems
\begin{align}\label{eq:evp-m}
\text{find} \quad \bigg(\nu,\bpm \bv_m\\w_m \epm\bigg)\in\setC\times \Xspace_m\setminus\{0\}
\quad\text{such that}\quad
(\Aop_m+\Bop_m) \bpm \bv_m\\w_m \epm &= \nu^2  \Kop_m \bpm \bv_m\\w_m \epm
\end{align}
and
\begin{align}\label{eq:evp-sa-m}
\text{find} \quad \bigg(\nu,\bpm \bv_m\\w_m \epm\bigg)\in\setC\times \Xspace_m\setminus\{0\}
\quad\text{such that}\quad
\Aop_m \bpm \bv_m\\w_m \epm &= \nu^2  \Kop_m \bpm \bv_m\\w_m \epm,
\end{align}
where $\Aop_m:=\Aop|_{X_m}, \Bop_m:=\Bop|_{X_m}, \Kop_m:=\Kop|_{X_m} \in \BLO(\Xspace_m)$.
Hence it suffices to conduct the perturbation analysis for the problems \eqref{eq:evp-m} and \eqref{eq:evp-sa-m} separately for each $m\in\setZ$.
However, the respective perturbation estimates have to be uniform in $m\in\setZ$ in order to yield a meaningful result for the full space $X$.
To be more precise, we introduce the following statements.
\begin{definition}
A family of vectors $(\phi_n)_{n\in\mathbb{N}}$ in a Hilbert space $\Phi\ni \phi_n$ is called a Riesz basis, if each $\phi\in \Phi$ admits a unique representation $\phi=\sum_{n\in\setN} c_n \phi_n$, $c_n\in\setC$ and there exist constants $c,C>0$ such that
\begin{align*}
	c \sum_{n\in\setN} |c_n|^2 \leq \|\phi\|_\Phi^2 \leq C \sum_{n\in\setN} |c_n|^2
	\qquad\text{for all } \phi\in \Phi.
\end{align*}
We call $c,C$ the Gram constants of the Riesz basis $(\phi_n)_{n\in\mathbb{N}}$.
\end{definition}
\begin{lemma}\label{lem:compoundbasis}
Let the Hilbert space $\Phi$ be separated into a discrete family of orthogonal subspaces, i.e., $\calI$ is a discrete index set and $\Phi=\bigoplus_{m\in\calI}^\bot \Phi_m$.
Let $(\phi_{n,m})_{n\in\mathbb{N}}$ be a Riesz basis of $\Phi_m$ with Gram constants $c_m,C_m>0$ for each $m\in\calI$.
If there exist constants $c,C>0$ such that $c<c_m, C_m<C$ for all $m\in\calI$, then $(\phi_{n,m})_{n\in\mathbb{N},m\in\calI}$ forms a Riesz basis of $\Phi$ with Gram constants $c,C$.
\end{lemma}
\begin{proof}
Follows with basic considerations.
\end{proof}

\section{Properties of Bessel functions}\label{sec:bessel}

% other usefull results on Bessel functions not applied in this paper:
% Dunster24,FilonovLevitinPolterovich24,Nemes21
% Landau99
% Reif24
% Paris84
% Kowalenko02
% Nemes15
% Nemes17
% ChungLeePark24b 
% LangowskiNowak25
% Paris04
% ElbertLaforgia85
% Olenko06

Let $\besselj_m$, $\bessely_m$ be the Bessel functions of the first and second kind, $(\besseljz_{n,m})_{n\in\setN}$ be the non-decreasing sequence of positive roots of $\besselj_m$, $m\in\Nz$ and $(\besseljz'_{n,m})_{n\in\setN}$ be the non-decreasing sequence of positive roots of the derived Bessel function $\besselj'_m$, $m\in\Nz$.
It is known that $(\besseljz_{n,m})_{n\in\setN}$ and $(\besseljz'_{n,m})_{n\in\setN}$ are interlacing \cite[Thm.~1]{BariczKokologiannakiPogany18} for each $m\in\Nz$, i.e.,
\begin{align}\label{eq:interlacing}
\besseljz_{1,0}<\besseljz'_{1,0}<\besseljz_{2,0}<\besseljz'_{2,0}<\dots \quad\text{and}\quad
\besseljz'_{1,m}<\besseljz_{1,m}<\besseljz'_{2,m}<\besseljz_{2,m}<\dots,\, m\in\setN
\end{align}
We introduce the compact notation
\begin{align*}
(\besseljzc_{n,0})_{n\in\setN}:=\besseljz_{1,0},\besseljz'_{1,0},\besseljz_{2,0},\besseljz'_{2,0},\dots \quad\text{and}\quad
(\besseljzc_{n,m})_{n\in\setN}:=
\besseljz'_{1,m},\besseljz_{1,m},\besseljz'_{2,m},\besseljz_{2,m},\dots,\,  m\in\setN.
\end{align*}
Throughout the manuscript we will repeatedly apply the Wronskian identity \cite[Eq.~10.5.2]{NISTDLMF}
\begin{align}\label{eq:Wronskian}
\besselj_m'(r)\bessely_m(r)-\besselj_m(r)\bessely'_m(r)=\frac{2}{\pi r},
\quad r>0, m\in\Nz.
\end{align}
Let $\Ai$ and $\Bi$ be the Airy functions of the first and second kind and
\begin{align*}
	(\Aiz_n)_{n\in\setN}, (\Aipz_n)_{n\in\setN}, \quad \Aiz_n>0,
\end{align*}
be the ascendant sequence of the positive roots of $\Ai(-\cdot)$ and $\Ai'(-\cdot)$ respectively.
Note that $(\Aiz_n)_{n\in\setN}$ and $(\Aipz_n)_{n\in\setN}$ are interlacing, i.e.,
\begin{align*}
	0<\Aipz_1<\Aiz_1<\Aipz_2<\Aiz_2<\dots,
\end{align*}
see, e.g., \cite[(2.2), Thm.~(T3.1)]{Baldwin85}.
We introduce the compact notation
\begin{align*}
(\Aizc_n)_{n\in\setN}:=\Aipz_1,\Aiz_1,\Aipz_2,\Aiz_2,\dots.
\end{align*}

\subsection{Asymptotic approximations}

To establish \Cref{lem:Crjy,lem:Crjj,lem:Crrjjyy} (which will be used for \Cref{lem:det} and \Cref{prop:evgap}) we will repeatedly use (incremental adaptations/generalizations) of the Bessel function asymptotics provided by \cite{Sher23,Krasikov14a}.
To make the manuscript more self-contained we formulate those in the following two \Cref{lem:besselTransZone,lem:besselOscZoneI}.
We start with the behavior in the transition zone.
\begin{lemma}\label{lem:besselTransZone}
For each compact interval $[t_1,t_2]\subset\setR$ there exists a constant $C_{[t_1,t_2]}>0$ such that
\begin{subequations}
\begin{align}
\label{eq:SherJ}
\sup_{t\in [t_1,t_2]} |\besselj_m(m+m^{1/3}t)-\frac{2^{1/3}}{m^{1/3}}\Ai(-2^{1/3}t)|&\leq \frac{C_{[t_1,t_2]}}{m},\\
\label{eq:SherY}
\sup_{t\in [t_1,t_2]} |\bessely_m(m+m^{1/3}t)+\frac{2^{1/3}}{m^{1/3}}\Bi(-2^{1/3}t)|&\leq \frac{C_{[t_1,t_2]}}{m},\\
\label{eq:SherJp}
\sup_{t\in [t_1,t_2]} |\besselj'_m(m+m^{1/3}t)+\frac{2^{2/3}}{m^{2/3}}\Ai'(-2^{1/3}t)|&\leq \frac{C_{[t_1,t_2]}}{m^{4/3}},\\
\label{eq:SherYp}
\sup_{t\in [t_1,t_2]} |\bessely'_m(m+m^{1/3}t)-\frac{2^{2/3}}{m^{2/3}}\Bi'(-2^{1/3}t)|&\leq \frac{C_{[t_1,t_2]}}{m^{4/3}},\\
\label{eq:SherJpp}
\sup_{t\in [t_1,t_2]} |\besselj''_m(m+m^{1/3}t)-\frac{2}{m}\Ai(-2^{1/3}t)|&\leq \frac{C_{[t_1,t_2]}}{m^{5/3}},\\
\label{eq:SherYpp}
\sup_{t\in [t_1,t_2]} |\bessely''_m(m+m^{1/3}t)+\frac{2}{m}\Bi(-2^{1/3}t)|&\leq \frac{C_{[t_1,t_2]}}{m^{5/3}},
\end{align}
\end{subequations}
for all $m\in\Nz$.
\end{lemma}
\begin{proof}
The asymptotics are in principal well-known \cite[Sect.~10.19]{NISTDLMF}.
The actual bounds are due to \cite[Prop.~2.9]{Sher23}.
\end{proof}
In order to stick to the convenient notation \cite{Krasikov14a}, but to avoid a double use of symbols we use a blackboard font style for the following constants:
\begin{align*}
\mushift:=|m^2-1/4|, \qquad \omshift:=(2m+1)\pi/4.
\end{align*}
Further, let
\begin{align*}
\phaseB(r):=\sqrt{r^2-\mushift}+\sqrt{\mushift}\arcsin\frac{\sqrt{\mushift}}{r},\qquad
\phaseB'(r)=\frac{\sqrt{r^2-\mushift}}{r},\qquad
r\geq\sqrt{\mushift}.
\end{align*}
\begin{lemma}\label{lem:besselOscZoneI}
For each $m\in\setN$ and $r>\sqrt{\mushift}$ it holds that
\begin{subequations}
\begin{align}
\label{eq:KrasikovJ}
|\besselj_m(r)-\sqrt{\frac{2}{\pi}} (r^2-\mushift)^{-1/4} \cos(\phaseB(r)-\omshift)| &\leq \frac{13\mushift}{12\sqrt{2\pi}(r^2-\mushift)^{7/4}},\\
\label{eq:KrasikovY}
|\bessely_m(r)-\sqrt{\frac{2}{\pi}} (r^2-\mushift)^{-1/4} \sin(\phaseB(r)-\omshift)| &\leq \frac{13\mushift}{12\sqrt{2\pi}(r^2-\mushift)^{7/4}},\\
\label{eq:KrasikovJp}
|\besselj'_m(r)+\sqrt{\frac{2}{\pi}} \frac{(r^2-\mushift)^{1/4}}{r} \sin(\phaseB(r)-\omshift)| &\leq
%\frac{3r}{\sqrt{8\pi}(r^2-\mushift)^{5/4}}\\
%\frac{6\mushift}{\sqrt{8\pi}(r^2-\mushift)^{1/4}}
%\int_r^{+\infty} (z^2-\mushift)^{-2} \,\dd z\\
\frac{5}{\sqrt{8\pi}} \frac{r}{(r^2-\mushift)^{5/4}}
+\frac{13\mushift r}{24\sqrt{2\pi}(r^2-\mushift)^{11/4}},\\
\label{eq:KrasikovYp}
|\bessely'_m(r)-\sqrt{\frac{2}{\pi}} \frac{(r^2-\mushift)^{1/4}}{r} \cos(\phaseB(r)-\omshift)| &\leq
\frac{5}{\sqrt{8\pi}} \frac{r}{(r^2-\mushift)^{5/4}}
+\frac{13\mushift r}{24\sqrt{2\pi}(r^2-\mushift)^{11/4}}.
%%
%\label{eq:KrasikovJpp}
%|\besselj''_m(r)-\sqrt{\frac{2}{\pi}} (r^2-\mushift)^{-1/4} \cos(\phaseB(r)-\omshift)| &\leq \frac{13\mushift}{12\sqrt{2\pi}(r^2-\mushift)^{7/4}},\\
%%
\end{align}
\end{subequations}
\end{lemma}
\begin{proof}
The first equation \eqref{eq:KrasikovJ} is merely \cite[Thm.~5]{Krasikov14a}.
To obtain \eqref{eq:KrasikovY} we repeat the proof of \cite[Thm.~5]{Krasikov14a} with minor adaptations:
The proof applies the estimate $(r^2-\mushift)^{1/4}|\bessely_m(r)|\leq \sqrt{2/\pi}$ for $r>\mushift$, which follows as in the first case in the proof of \cite[Thm.~3]{Krasikov14a} by replacing $\besselj_m$ with $\bessely_m$.
Furthermore, the constants $c_1,c_2$ of the homogeneous solution in the proof \cite[Thm.~5]{Krasikov14a} change and are deduced from the asymptotic of $\bessely_m(r)$ for $r\to+\infty$.
To obtain \eqref{eq:KrasikovJp} we note that \cite{Krasikov14a} actually yields
\begin{align*}
(r^2-\mushift)^{1/4} &\besselj_m(r)-\sqrt{\frac{2}{\pi}} \cos(\phaseB(r)-\omshift) \\
&= \int_r^{+\infty} \frac{3(\phaseB''(t))^2-2\phaseB'(t)\phaseB'''(t)}{4(\phaseB'(t))^4} \phaseB'(t) (t^2-\mushift)^{1/4}\besselj_m(t) \sin(\phaseB(r)-\phaseB(t))\,\dd t
\end{align*}
and
\begin{align*}
\int_r^{+\infty} \left|\frac{3(\phaseB''(t))^2-2\phaseB'(t)\phaseB'''(t)}{4(\phaseB'(t))^4} \phaseB'(t) (t^2-\mushift)^{1/4} \besselj_m(t)\right| \,\dd t
\leq \frac{13\mushift}{12\sqrt{2\pi}(r^2-\mushift)^{7/4}} \quad\text{for } r>\sqrt{\mushift}.
\end{align*}
Multiplying the former identity by $(r^2-\mushift)^{-1/4}$ and deriving it we obtain that the argument inside the absolute value bars in the left hand-side of \eqref{eq:KrasikovJp} equals
\begin{align*}
&-\frac{1}{\sqrt{2\pi}} \frac{r}{(r^2-\mushift)^{5/4}} \cos(\phaseB(r)-\omshift)\\
&-\frac{1}{2}\frac{r}{(r^2-\mushift)^{5/4}}
\int_r^{+\infty} \frac{3(\phaseB''(t))^2-2\phaseB'(t)\phaseB'''(t)}{4(\phaseB'(t))^4} \phaseB'(t) (t^2-\mushift)^{1/4}\besselj_m(t) \sin(\phaseB(r)-\phaseB(t))\,\dd t\\
&+\frac{1}{(r^2-\mushift)^{1/4}}
\int_r^{+\infty} \frac{3(\phaseB''(t))^2-2\phaseB'(t)\phaseB'''(t)}{4(\phaseB'(t))^4} (\phaseB'(t))^2 (t^2-\mushift)^{1/4}\besselj_m(t) \cos(\phaseB(r)-\phaseB(t))\,\dd t.
\end{align*}
We estimate the absolute value of the first two terms by
\begin{align*}
\frac{1}{\sqrt{2\pi}} \frac{r}{(r^2-\mushift)^{5/4}}
+\frac{13\mushift r}{24\sqrt{2\pi}(r^2-\mushift)^{11/4}}.
\end{align*}
Exploiting computations from \cite[p.~222]{Krasikov14a} the third term can be estimated as
% log((x+c)/(x-c))/(4c^3)+x/(2c^2(c^2-x^2))
% https://www.wolframalpha.com/input?i=log%28%28x%2Bc%29%2F%28x-c%29%29%2F%284c%5E3%29%2Bx%2F%282c%5E2%28c%5E2-x%5E2%29%29
\begin{align*}
\frac{6\mushift}{\sqrt{8\pi}(r^2-\mushift)^{1/4}}
\int_r^{+\infty} (z^2-\mushift)^{-2} \,\dd z
&=\frac{6\mushift}{\sqrt{8\pi}(r^2-\mushift)^{1/4}}
\Big( \frac{r}{2\mushift(r^2-\mushift)}-\frac{\log\frac{r+\sqrt{\mushift}}{r-\sqrt{\mushift}}}{4\mushift^{3/2}} \Big)\\
&\leq \frac{3r}{\sqrt{8\pi}(r^2-\mushift)^{5/4}}.
\end{align*}
Hence \eqref{eq:KrasikovJp} follows.
Finally, \eqref{eq:KrasikovYp} follows by combining the above adaptations.
\end{proof}
\begin{lemma}\label{lem:besselOscZoneII}
For each $m\in\setN$ and $r>m+m^{1/3}\tstar$, $\tstar>1$ it holds that
\begin{subequations}
\begin{align}
\label{eq:KrasikovJbound}
|(r^2-\mushift)^{1/4} \besselj_m(r)-\sqrt{\frac{2}{\pi}} \cos(\phaseB(r)-\omshift)|
&\leq %\frac{13\mushift}{12\sqrt{2\pi}(r^2-\mushift)^{3/2}} \leq
\frac{13}{12\sqrt{2\pi}\tstar^{3/2}},\\
\label{eq:KrasikovYbound}
|(r^2-\mushift)^{1/4} \bessely_m(r)-\frac{\sqrt{2}}{\sqrt{\pi}} \sin(\phaseB(r)-\omshift)|
&\leq
\frac{13}{12\sqrt{2\pi}\tstar^{3/2}},\\
\label{eq:KrasikovJpbound}
|\frac{r}{(r^2-\mushift)^{1/4}} \besselj'_m(r)+\sqrt{\frac{2}{\pi}} \sin(\phaseB(r)-\omshift)| &\leq
%\frac{5}{\sqrt{8\pi}} \frac{r^2}{(r^2-\mushift)^{3/2}}
%+\frac{13\mushift r^2}{24\sqrt{2\pi}(r^2-\mushift)^{3}}\\
%&\leq
%\frac{5}{\sqrt{8\pi}} \frac{1}{(r^2-\mushift)^{1/2}}
%+\frac{5}{\sqrt{8\pi}} \frac{\mushift}{(r^2-\mushift)^{3/2}}\\
%&+\frac{13\mushift}{24\sqrt{2\pi}(r^2-\mushift)^{2}}
%+\frac{13\mushift^2}{24\sqrt{2\pi}(r^2-\mushift)^{3}}\\
%&\leq
%\frac{5}{\sqrt{8\pi}} \frac{1}{m^{2/3}r_*^{1/2}}
%+\frac{5}{\sqrt{8\pi}} \frac{1}{r_*^{3/2}}\\
%&+\frac{13}{24\sqrt{2\pi}m^{2/3}r_*^2}
%+\frac{13}{24\sqrt{2\pi}r_*^3}\\
%&\leq
%\frac{5}{\sqrt{2\pi}} \frac{1}{r_*^{1/2}}
%+\frac{13}{12\sqrt{2\pi}r_*^2}
%\leq
\frac{7}{\sqrt{2\pi}} \frac{1}{\tstar^{1/2}}\\
\label{eq:KrasikovYpbound}
|\frac{r}{(r^2-\mushift)^{1/4}} \bessely'_m(r)-\sqrt{\frac{2}{\pi}} \cos(\phaseB(r)-\omshift)| &\leq
\frac{7}{\sqrt{2\pi}} \frac{1}{\tstar^{1/2}}.
\end{align}
\end{subequations}
\end{lemma}
\begin{proof}
To establish \eqref{eq:KrasikovJbound} we multiply \eqref{eq:KrasikovJ} with $(r^2-\mushift)^{1/4}$
\begin{align*}
|(r^2-\mushift)^{1/4} \besselj_m(r)-\sqrt{\frac{2}{\pi}} \cos(\phaseB(r)-\omshift)|
&\leq \frac{13\mushift}{12\sqrt{2\pi}(r^2-\mushift)^{3/2}}
%\leq \frac{13}{12\sqrt{2\pi}\tstar^{3/2}}
\end{align*}
and note that the new right hand-side is monotonically decreasing in $r\in(\sqrt{\mushift},+\infty)$.
Thus plugging $r=m+m^{1/3}\tstar$ into $\frac{13\mushift}{12\sqrt{2\pi}(r^2-\mushift)^{3/2}}$ and applying elementary estimates yield the desired estimate \eqref{eq:KrasikovJbound}.
Estimate \eqref{eq:KrasikovYbound} follows the very same way.
To establish \eqref{eq:KrasikovJpbound} we multiply \eqref{eq:KrasikovJp} with $\frac{r}{(r^2-\mushift)^{1/4}}$ and estimate
\begin{align*}
|\frac{r}{(r^2-\mushift)^{1/4}}& \besselj'_m(r)+\sqrt{\frac{2}{\pi}} \sin(\phaseB(r)-\omshift)|\\
&\leq
\frac{5}{\sqrt{8\pi}} \frac{r^2}{(r^2-\mushift)^{3/2}}
+\frac{13\mushift r^2}{24\sqrt{2\pi}(r^2-\mushift)^{3}}\\
&=
\frac{5}{\sqrt{8\pi}} \frac{1}{(r^2-\mushift)^{1/2}}
+\frac{5}{\sqrt{8\pi}} \frac{\mushift}{(r^2-\mushift)^{3/2}}
+\frac{13\mushift}{24\sqrt{2\pi}(r^2-\mushift)^{2}}
+\frac{13\mushift^2}{24\sqrt{2\pi}(r^2-\mushift)^{3}}\\
&\leq
\frac{5}{\sqrt{8\pi}} \frac{1}{m^{2/3}\tstar^{1/2}}
+\frac{5}{\sqrt{8\pi}} \frac{1}{\tstar^{3/2}}
+\frac{13}{24\sqrt{2\pi}m^{2/3}\tstar^2}
+\frac{13}{24\sqrt{2\pi}\tstar^3}\\
&\leq
\frac{5}{\sqrt{2\pi}} \frac{1}{\tstar^{1/2}}
+\frac{13}{12\sqrt{2\pi}\tstar^2}
\leq
\frac{7}{\sqrt{2\pi}} \frac{1}{\tstar^{1/2}},
\end{align*}
where in the step from the third to the fourth line we exploited the monotonicity in $r\in(\sqrt{\mushift},+\infty)$.
Estimate \eqref{eq:KrasikovYpbound} follows the very same way.
\end{proof}
Note that estimates for $F''_m(r)$, $F=\besselj,\bessely$ in $r>\sqrt{\mushift}$ can be obtained from
\begin{align}
\label{eq:JYppp}
F_m''=\frac{-1}{r}F_m'(r)-\frac{r^2-m^2}{r^2}F_m(r),\quad F=\besselj,\bessely,
\end{align}
which itself follows from $F_m''(r)=(r^{-1}rF_m'(r))'=\frac{-1}{r}F_m'(r)+\frac{1}{r}\big(\frac{m^2}{r}-r\big)F_m(r)$.

\subsection{Gaps of roots and extrema}

\begin{proposition}\label{prop:cdistbes}
Let
\begin{align*}
\cgap:=\min\Big\{\besseljz'_{1,1},1,\frac{2\besseljz_{1,0}}{2\besseljz_{1,0}+1}\Big\}.
\end{align*}
Then it holds that
\begin{align*}
|\besseljzc_{1,m}|\geq\cgap,\quad
|\besseljzc_{n+1,m}-\besseljzc_{n,m}|\geq\cgap,\quad
&\text{for all}\quad \indn\in\setN,\indm\in\Nz.
\end{align*}
\end{proposition}
\begin{proof}
First, note that $\min_{m\in\Nz}\besseljzc_{1,m}=\besseljz'_{1,1}$.
Due to \eqref{eq:interlacing} it suffices to estimate
\begin{align*}
\besseljz'_{n,0}-\besseljz_{n,0},\quad
|\besseljz'_{n,0}-\besseljz_{n+1,0}|\quad \text{ for }n\in\setN
\end{align*}
and
\begin{align*}
|\besseljz'_{1,m}-\besseljz_{1,m}|,\quad
\besseljz'_{n,m}-\besseljz_{n-1,m},\quad
|\besseljz'_{n,m}-\besseljz_{n,m}|\quad \text{ for } n\in\setN\setminus\{1\}, m\in\setN.
\end{align*}
We recall from \cite[(F5)]{LangowskiNowak25} that
\begin{align*}
\sum_{k\geq1} \frac{\besseljz_{k,m}^2}{\big((\besseljz'_{n,m})^2-\besseljz_{k,m}^2\big)^2}
=\frac{1}{4}-\frac{m^2}{4(\besseljz'_{n,m})^2}, \quad n\geq1.
\end{align*}
Thus
\begin{align*}
|\besseljz'_{n,m}-\besseljz_{k,m}| \geq \frac{2\besseljz_{k,m}}{\besseljz'_{n,m}+\besseljz_{k,m}}, \quad n,k\in\setN,m\in\Nz.
\end{align*}
If $\besseljz'_{n,m}<\besseljz_{k,m}$, then $|\besseljz'_{n,m}-\besseljz_{k,m}| \geq1$.
The case $\besseljz'_{n,m}>\besseljz_{k,m}$ and $\besseljz'_{n,m}-\besseljz_{k,m} \geq1$ does not require additional treatment.
Hence the case $\besseljz'_{n,m}>\besseljz_{k,m}$ and $\besseljz'_{n,m}-\besseljz_{k,m} <1$ remains, for which we estimate
$|\besseljz'_{n,m}-\besseljz_{k,m}| \geq \frac{2\besseljz_{k,m}}{2\besseljz_{k,m}+1}\geq\frac{2\besseljz_{1,0}}{2\besseljz_{1,0}+1}$.
Thus the proof is finished.
\end{proof}
\begin{remark}
Note that an alternative proof of \Cref{prop:cdistbes} (with a less explicit constant $\cgap$) follows along the lines of \Cref{lem:ev-conv} and \Cref{prop:evgap}.
\end{remark}

\subsection{Boundedness of Bessel function products}

% plot pi*n*x*besselj(n,n*x)*1/2*(bessely(n-1,n*x)-bessely(n+1,n*x)), x=0...1, n=30
% https://www.wolframalpha.com/input?i=plot+pi*n*x*besselj%28n%2Cn*x%29*1%2F2*%28bessely%28n-1%2Cn*x%29-bessely%28n%2B1%2Cn*x%29%29%2C+x%3D0...1%2C+n%3D30
% plot pi*m*x*besselj(m,m*x)*1/2*(bessely(m-1,m*x)-bessely(m+1,m*x)),pi*n*x*besselj(n,n*x)*1/2*(bessely(n-1,n*x)-bessely(n+1,n*x)), x=0.0...1, n=20, m=10
% https://www.wolframalpha.com/input?i=plot+pi*m*x*besselj%28m%2Cm*x%29*1%2F2*%28bessely%28m-1%2Cm*x%29-bessely%28m%2B1%2Cm*x%29%29%2Cpi*n*x*besselj%28n%2Cn*x%29*1%2F2*%28bessely%28n-1%2Cn*x%29-bessely%28n%2B1%2Cn*x%29%29%2C+x%3D0.0...1%2C+n%3D20%2C+m%3D10
%
% Plot[x*(BesselJ(100-1,x)-BesselJ(100+1,x))*BesselY(100,x)/100, {x,90,100}]
% Det[{{besselj[0,x*c],-besselj[0,x*c],-bessely[0,x*c]},{d/dx besselj[0,x*c],-d/dx besselj[0,x*c]*0.99,-d/dx bessely[0,x*c]*0.99}/c,{0,besselj[0,x],bessely[0,x]}

\begin{lemma}\label{lem:Crjy}
There exists a constant $\Crjy>0$ such that
\(\sup_{r>0,m\in\Nz}|r\besselj'_m(r)\bessely_m(r)|\leq\Crjy.\)
\end{lemma}
\begin{proof}
\emph{0.\ case $m=0$:}\quad
$r\besselj'_0(r)\bessely_0(r)$ is continuous functions on $r\in(0,+\infty)$.
The well-known asymptotics for $r\besselj'_0(r)$ and $\bessely_0(r)$ yields that $\limsup_{r\to+\infty} |r\besselj'_0(r)\bessely_0(r)| \leq 2\pi$.
In addition, $\lim_{r\to0+} |r\besselj'_0(r)\bessely_0(r)|=1/\pi$ \cite[Eq.~10.6.2., 10.7.3, 10.7.4]{NISTDLMF}.
Thus $\Crjy^{(0)}:=\sup_{r>0}|r\besselj'_0(r)\bessely_0(r)|<+\infty$.\\
To proceed we split $(0,+\infty)$ into $(0,m]$, $[m,m+m^{1/3}\tstar]$, $[m^{1/3}\tstar,+\infty)$ with an arbitrary constant $\tstar>1$.\\
\emph{1.\ case $r\in(0,m]$:}\quad
Let $\besselyz_{1,m}$ be the first positive zero of $\bessely_m$.
We recall \cite[p.133]{Watson95} that \[\bessely_m(r)=\frac{2}{\pi}\besselj_m(r) \int_{\besselyz_{1,m}}^r \frac{1}{z\besselj_m(z)^2}\,\dd z.\]
Thus for
$f_m(r):=\frac{\pi}{2}r\besselj'_m(r)\bessely_m(r)=r\besselj'_m(r)\besselj_m(r) \int_{\besselyz_{1,m}}^r \frac{1}{z\besselj_m(z)^2}\,\dd z$ we compute that
\[f_m'(r)=\frac{\besselj'_m(r)}{\besselj_m(r)}+\Big( \big(\frac{m^2}{r}-r\big)\besselj_m(r)^2+r\besselj'_m(r)^2 \Big) \int_{\besselyz_{1,m}}^r \frac{1}{z\besselj_m(z)^2}\,\dd z.\]
Hence for any root $r_*\in(0,m)$ of $f'_m$ it holds that
$\int_{\besselyz_{1,m}}^{r_*} \frac{1}{z\besselj_m(z)^2}\,\dd z=-\frac{\besselj'_m(r_*)}{\besselj_m(r_*)}\frac{1}{(\frac{m^2}{r_*}-r_*)\besselj_m(r)^2+r_*\besselj'_m(r_*)^2}$.
We deduce that $f_m(r_*)=\frac{-r_*\besselj'_m(r_*)^2}{(\frac{m^2}{r_*}-r_*)\besselj_m(r_*)^2+r_*\besselj'_m(r_*)^2}\in(-1,0)$.
With \cite[Eq.~10.6.2., 10.7.3, 10.7.4]{NISTDLMF} we compute that $\lim_{r\to0+}f_m(r)=-1/2$ for each $m\in\Nz$.
On the other hand it follows by means of \cite[p.232]{Watson95} that $\lim_{m\to+\infty} f_m(m)=\frac{3^{1/2}\Gammaf(1/3)\Gammaf(2/3)}{4}$ (with $\Gammaf$ meaning the Gamma function).
It follows that $\Crjy^{(1)}:=\sup_{r\in(0,m),m\in\setN}|r\besselj'_m(r)\bessely_m(r)|<+\infty)$.\\
\emph{2.\ case $r\in[m,m+m^{1/3}\tstar]$:}\quad
In this regime it holds that $r/m\in [1,1+\tstar]$.
Hence by means of \eqref{eq:SherJp} and \eqref{eq:SherY} it follows that $\Crjy^{(2)}:=\sup_{r\in[m,m+m^{1/3}\tstar],m\in\setN}|r\besselj'_m(r)\bessely_m(r)|<+\infty$.\\
\emph{3.\ case $r\in[m+m^{1/3}\tstar],+\infty)$:}\quad
Applying the estimates \eqref{eq:KrasikovJpbound} and \eqref{eq:KrasikovYbound} yields that $\Crjy^{(3)}$ $:=$ $\sup_{r\in[m+m^{1/3}\tstar,+\infty),m\in\setN}|r\besselj'_m(r)\bessely_m(r)|<+\infty$.\\
The claim \(\sup_{r>0,m\in\Nz}|r\besselj'_m(r)\bessely_m(r)|\leq\Crjy\)
follows now with $\Crjy:=\max\{\Crjy^{(0)},\Crjy^{(1)},\Crjy^{(2)},\Crjy^{(3)}\}$.
\end{proof}
Having established \Cref{lem:Crjy} the next two lemmas follow without much effort.
\begin{lemma}\label{lem:Crjj}
There exists a constant $\Crjj>0$ such that
\(\sup_{r>0,m\in\Nz}|r\besselj'_m(r)\besselj_m(r)|\leq\Crjj.\)
\end{lemma}
\begin{proof}
The proof is similar to that of \ref{lem:Crjy}, where in the region $r\in[0,m]$ we use that $|\besselj_m(r)|,|\besselj'_m(r)|$ $\leq1$ \cite[Eq.~10.6.1,10.14.1]{NISTDLMF}.
\end{proof}
\begin{lemma}\label{lem:Crrjjyy}
Let $q_1,q_2$ be constants with $0<q_1\leq q_2\leq1$.
Then there exists a constant $\Crrjjyy>0$ such that
\(\sup_{q\in[q_1,q_2],r>0,m\in\Nz}|r^2\besselj'_m(qr)\besselj_m(qr)\bessely'_m(r)\bessely_m(r)|
\leq\Crrjjyy.\)
\end{lemma}
\begin{proof}
Let $q>0$ be fixed.
In the region $r\in[m,+\infty)$ we estimate
\begin{align*}
\sup_{r>m,m\in\Nz}|r^2\besselj'_m(qr)\besselj_m(qr)\bessely'_m(r)\bessely_m(r)|
&\leq
q^{-1}\sup_{r>m,m\in\Nz}|qr\besselj'_m(qr)\besselj_m(qr)|
\sup_{r>m,m\in\Nz}|r\bessely'_m(r)\bessely_m(r)|\\
&\leq q^{-1}\Crjj \sup_{r>m,m\in\Nz}|r\bessely'_m(r)\bessely_m(r)|
\end{align*}
with $\Crjj$ as in \Cref{lem:Crjj}.
The term $\sup_{r>m,m\in\Nz}|r\bessely'_m(r)\bessely_m(r)|$ can be uniformly bounded by repeating the technique used in the proof of \Cref{lem:Crjy} (second and third case).
On $r\in[0,m]$ we exploit that $r\mapsto r\besselj'_m(qr)\besselj_m(qr)$ is monoton increasing to estimate
\begin{align*}
\sup_{r\in[0,m],m\in\Nz}|r^2\besselj'_m(qr)\besselj_m(qr)\bessely'_m(r)\bessely_m(r)|
&\leq
\sup_{r\in[0,m],m\in\Nz}|r^2\besselj'_m(r)\besselj_m(r)\bessely'_m(r)\bessely_m(r)|\\
&\leq
\sup_{r\in[0,m],m\in\Nz}|r\besselj'_m(r)\bessely_m(r)|
\sup_{r\in[0,m],m\in\Nz}|r\besselj_m(r)\bessely'_m(r)|
\end{align*}
Due to \Cref{lem:Crjy} and \eqref{eq:Wronskian} we obtain that
$\sup_{r>0,m\in\Nz}|r^2\besselj'_m(qr)\besselj_m(qr)\bessely'_m(r)\bessely_m(r)|<+\infty$.
The final claim follows by a continuity argument and the fact that $[q_1,q_2]$ is compact.
\end{proof}

\section{The selfadjoint perturbation}\label{sec:perturbation-sa}

In this section study the properties of the eigenvalue problem \eqref{eq:evp-sa} (or rather the family \eqref{eq:evp-sa-m} respectively).
Points of interest are the gaps between eigenvalues (see \Cref{subsec:gaps}) and trace inequalities for the eigenfunctions (see \Cref{subsec:traceofefs}).

\subsection{Reduction of $\Aop_{11}\bv=\nu^2 \Kop_{11}\bv$ to a scalar problem}

To study the properties of the selfadjoint eigenvalue problem associated to \eqref{eq:evp-sa} we note that \eqref{eq:evp-sa} decouples into the two separate problems
\begin{subequations}
\begin{align}
\label{eq:evp-bv}
\text{find}\quad (\nu,\bv)\in\setC\times\Vspace\setminus\{0\}\quad\text{such that}\quad
\Aop_{11}\bv&=\nu^2 \Kop_{11}\bv,\\
\label{eq:evp-w}
\text{find}\quad (\nu,w)\in\setC\times W\setminus\{0\}\quad\text{such that}\quad
\Aop_{22} w&=\nu^2 \Kop_{22} w.
\end{align}
\end{subequations}
As a next step we reduce the analysis of the vectorial problem \eqref{eq:evp-bv} to a problem for the scalar potential.
Let
\begin{align*}
\Deltamu v:=\divg(\mu\gradg v)=-\curlg(\mu\Curlg v)
\quad\text{and}\quad
\Deltaeps w:=\divg(\epsilon\gradg w)=-\curlg(\epsilon\Curlg w).
\end{align*}
\begin{lemma}\label{lem:vector-to-scalar}
Let $\nu\in\setC$.
Then $\bv\in\Vspace\setminus\{\mathbf{0}\}$, $\bv=\mu\Curlg v$, $v\in\Vpotspace\setminus\{0\}$ solves \eqref{eq:evp-bv},
if and only if it's potential
$v\in V$ solves
\begin{align}\label{eq:evp-v}
-\divg(\mu\gradg v)&=\nu^2 \mu v \quad\text{in }\domcs \quad\text{and}\quad
\mu\partial_{\nvt}\gradg v=0 \quad\text{on }\partial\domcs.
\end{align}
In such a case, if in addition $\bv=\bv_m\in\Vspace_m$ or equivalently $v=v_m\in\Vm$, $m\in\setZ$, then
\begin{align}\label{eq:trace-bvm}
\nvt\cdot\mu^{-1}\bv_m=im v_m \quad\text{on }\intf.
\end{align}
\end{lemma}
\begin{proof}
Let $\bv=\mu\Curlg v$ solve \eqref{eq:evp-bv}.
Note that $\nu\neq0$.
It follows that
\begin{align}\label{eq:lem-reductionI}
\mu^{-1}\curlg\bv\in V% H^1_{\mu*}(\domcs)
\quad\text{and}\quad
\tv\cdot\mu(\mu^{-1}\curlg\bv)=0 \quad\text{on }\partial\domcs.
\end{align}
Let $\tilde v:=\mu^{-1} \Deltamu v$.
Then \eqref{eq:lem-reductionI} yields that $\tilde v\in V$
and $\mu\partial_{\nvt} \tilde v=0$ on $\partial\domcs$.
Applying integration by parts in the variational formulation of \eqref{eq:evp-bv} yields further that $-\Deltamu \tilde v=\nu^2\mu \tilde v$ in $\domcs$.
There exists a unique solution to the problem:
\begin{align}\label{eq:lem-reductionII}
\text{find } u\in V%H^1_{\mu*}(\domcs)
\text{ such that } -\Deltamu u=\mu\tilde v \text{ in } \domcs \quad\text{and}\quad \mu\partial_{\nvt}u=0 \text{ on } \partial\domcs.
\end{align}
It can be checked that the solution to \eqref{eq:lem-reductionII} equals $\nu^{-2}\tilde v$.
Thus $v=\nu^{-2} \tilde v$ solves \eqref{eq:evp-v}.\\
Vice-versa, let $v$ solve \eqref{eq:evp-v} and $\bv=\mu\Curl v$.
Then $\tv\cdot\bv=\mu\partial_{\nvt} v=0$ on $\partial\domcs$.
Furthermore, $\mu^{-1}\curlg\bv=\nu^2 v$ and thus
$\Curlg \mu^{-1}\curlg\bv=\nu^2 \Curlg v=\nu^2 \mu^{-1} \bv$ in $\domcs$.
Since $\bv\in\Vspace$ it follows that $\bv$ solves \eqref{eq:evp-bv}.\\
To prove the last statement we compute that
$\nvt\cdot\mu^{-1}\bv_m=\nvt\cdot\Curlg v_m=\partial_\theta v_m=im v_m$ on $\intf$.
\end{proof}

\subsection{Reduction of Laplace eigenvalue problems}\label{subsec:reductiontomatrix}

Our goal is to relate the eigenvalues of the problems
\begin{subequations}\label{eq:evps-scalar}
\begin{align}
\label{eq:evps-scalar-v}
\text{find } (\nu,v)\in\setC\times V%H^1_{\mu*}(\domcs)
\setminus\{0\}
\text{ such that }
-\Deltamu v&=\nu^2 \mu v\quad\text{in }\domcs,&\mu\partial_{\nvt}v&=0\quad\text{on } \partial\domcs,\\
\label{eq:evps-scalar-w}
\text{find } (\nu,w)\in\setC\times W%H^1_0(\domcs)
\setminus\{0\}
\text{ such that }
-\Deltaeps w&=\nu^2 \epsilon w\quad\text{in }\domcs,&\hspace{-25mm} \phantom{\mu\partial_{\nvt}}w&=0\quad\text{on } \partial\domcs,
\end{align}
\end{subequations}
to those of the homogeneous cases $\epsilon=\epsilon_0$ and $\mu=\mu_0$ respectively.
Since the perturbation of $\epsilon$ and $\mu$ in \eqref{eq:evps-scalar} is not small in any subordinate sense (see, e.g., \cite{Markus88,MityaginSiegl19} for the concept of subordination), convenient abstract perturbation techniques do not provide us with any fruitful means.
Note also that the Courant–Fischer–Weyl min-max principle would only provide (for our purpose) insufficient estimates
\begin{align*}
\nu_{\mathrm{Neu},\mu_0}^2 \frac{\mumin}{\mumax}
\leq \nu_{\mathrm{Neu},\mu}^2
\leq \nu_{\mathrm{Neu},\mu_0}^2 \frac{\mumax}{\mumin},\qquad
\nu_{\mathrm{Dir},\epsilon_0}^2 \frac{\epsmin}{\epsmax}
\leq \nu_{\mathrm{Dir},\epsilon}^2
\leq \nu_{\mathrm{Dir},\epsilon_0}^2 \frac{\epsmax}{\epsmin}.
\end{align*}
Instead we exploit our assumptions on $\domcs$ and $\epsilon,\mu$ to reduce \eqref{eq:evps-scalar} to the eigenvalue problems of $2\indL-1$-dimensional holomorphic matrix functions.
First we note that \eqref{eq:evps-scalar} decouples into the family ($m\in\setZ$) of problems
\begin{subequations}\label{eq:evps-scalar-m}
\begin{align}
\label{eq:evps-scalar-vm}
\text{find } (\nu,v_m)\in\setC\times \Vm\setminus\{0\}
\text{ such that }
-\Deltamu v_m&=\nu^2 \mu v_m\quad\text{in }\domcs,&\mu\partial_{\nvt}v&=0\quad\text{on } \partial\domcs,\\
\label{eq:evps-scalar-wm}
\text{find } (\nu,w_m)\in\setC\times \Wm\setminus\{0\}
\text{ such that }
-\Deltaeps w_m&=\nu^2 \epsilon w_w\quad\text{in }\domcs,&\hspace{-25mm} \phantom{\mu\partial_{\nvt}}w_m&=0\quad\text{on } \partial\domcs.
\end{align}
\end{subequations}
Let $(\nu,v_m)$ solve \eqref{eq:evps-scalar-vm}.
Recall that $\mu$ is constant on each interval $(r_l,r_{l+1})$, $l=1,\dots,\indL$ (ball $B_{r_2}$, annuli $B_{r_{l+1}}\setminus B_{r_l}$).
It follows that the function $e^{-im\theta}v_m|_{(r_l,r_{l+1})}$ depends only on $r$ and solves a scaled Bessel equation with index $m$.
Hence $e^{-im\theta}v_m|_{(r_l,r_{l+1})}(r)=a_l\besselj_m(\nu r)+b_l\bessely_m(\nu r)$ with constants $a_l,b_l\in\setC$.
Considering the boundary and transmission conditions implied by \eqref{eq:evps-scalar} it follows that $\vvec$ $:=$ $(a_1,a_2,b_2,\dots,a_{\indL},b_{\indL})^\top$ $\in$ $\setC^{2\indL-1}$ solves $\AmatNeu_{\indL}(\nu)\vvec=0$, where
\begin{subequations}
\begin{align}\label{eq:AmatNeu}
\AmatNeu_{\indL}(\nu):=
\bpm \besselj_{m2} & -\besselj_{m2} & -\bessely_{m2}\\
\frac{\mu_1}{\mu_2}\besselj'_{m2} & -\besselj'_{m2} & -\bessely'_{m2}\\
& \besselj_{m3} & \bessely_{m3} & -\besselj_{m3} & -\bessely_{m3}\\
& \frac{\mu_2}{\mu_3}\besselj'_{m3} & \frac{\mu_2}{\mu_3}\bessely'_{m3} & -\besselj'_{m3} & -\bessely'_{m3}\\
&&& \besselj_{m4} & \bessely_{m4} & -\besselj_{m4} & -\bessely_{m4}\\
&&& \frac{\mu_3}{\mu_4}\besselj'_{m4} & \frac{\mu_3}{\mu_4}\bessely'_{m4} & -\besselj'_{m4} & -\bessely'_{m4}\\
&&&&&\dots&\dots&\dots&\dots\\
&&&&&\dots&\dots&\dots&\dots\\
&&&&&&& \besselj'_{m(\indL+1)} & \bessely'_{m(\indL+1)} \epm
\end{align}
with the abbreviations
$\besselj_{ml}:=\besselj_m(\nu r_l)$,
$\besselj'_{ml}:=\besselj'_m(\nu r_l)$,
$\bessely_{ml}:=\bessely_m(\nu r_l)$,
$\bessely'_{ml}:=\bessely'_m(\nu r_l)$.
%\begin{align*}
%&\Amat:=\\
%&\bpm \besselj_m(r_2 z) & -\besselj_m(r_2 z) & -\bessely_m(r_2 z)\\
%\besselj'_m(r_2 z) & -\besselj'_m(r_2 z) & -\bessely'_m(r_2 z)\\
%& \besselj_m(r_3 z) & \bessely_m(r_3 z) & -\besselj_m(r_3 z) & -\bessely_m(r_3 z)\\
%& \besselj'_m(r_3 z) & \bessely'_m(r_3 z) & -\besselj'_m(r_3 z) & -\bessely'_m(r_3 z)\\
%&&& \besselj_m(r_4 z) & \bessely_m(r_4 z) & -\besselj_m(r_4 z) & -\bessely_m(r_4 z)\\
%&&& \besselj'_m(r_4 z) & \bessely'_m(r_4 z) & -\besselj'_m(r_4 z) & -\bessely'_m(r_4 z)\\
%&&&&&\dots&\dots&\dots&\dots\\
%&&&&&\dots&\dots&\dots&\dots\\
%&&&&&&& \besselj_m(r_{\indL+1} z) & \bessely_m(r_{\indL+1} z) \epm
%\end{align*}
%%}
%
Similarly, if $(\nu,w_m)$ solves \eqref{eq:evps-scalar-wm}, then $e^{-im\theta}w_m|_{(r_l,r_{l+1})}(r)=a_l\besselj_m(\nu r)+b_l\bessely_m(\nu r)$ with constants $a_l,b_l\in\setC$ and $\wvec:=(a_1,a_2,b_2,\dots,a_{\indL},b_{\indL})^\top\in\setC^{2\indL-1}$ solves $\AmatDir_{\indL}(\nu)\wvec=0$, where
\begin{align}\label{eq:AmatDir}
\AmatDir_{\indL}(\nu):=
\bpm \besselj_{m2} & -\besselj_{m2} & -\bessely_{m2}\\
\frac{\epsilon_1}{\epsilon_2}\besselj'_{m2} & -\besselj'_{m2} & -\bessely'_{m2}\\
& \besselj_{m3} & \bessely_{m3} & -\besselj_{m3} & -\bessely_{m3}\\
& \frac{\epsilon_2}{\epsilon_3}\besselj'_{m3} & \frac{\epsilon_2}{\epsilon_3}\bessely'_{m3} & -\besselj'_{m3} & -\bessely'_{m3}\\
&&& \besselj_{m4} & \bessely_{m4} & -\besselj_{m4} & -\bessely_{m4}\\
&&& \frac{\epsilon_3}{\epsilon_4}\besselj'_{m4} & \frac{\epsilon_3}{\epsilon_4}\bessely'_{m4} & -\besselj'_{m4} & -\bessely'_{m4}\\
&&&&&\dots&\dots&\dots&\dots\\
&&&&&\dots&\dots&\dots&\dots\\
&&&&&&& \besselj_{m(\indL+1)} & \bessely_{m(\indL+1)} \epm,
\end{align}
i.e., compared to $\AmatNeu_{\indL}$ the last row is changed and $\mu,\mu_0$ are replaced by $\epsilon,\epsilon_0$ respectively.
Altogether, $\nu$ is an eigenvalue to \eqref{eq:evps-scalar-vm} or \eqref{eq:evps-scalar-wm}, if and only if $\det\AmatNeu_{\indL}(\nu)=0$ or $\det\AmatDir_{\indL}(\nu)=0$ respectively.

\subsection{Computation of $\det\AmatNeu_{\indL}(\nu)$ and $\det\AmatDir_{\indL}(\nu)$}

Note that the case $\indL=1$ corresponds to a homogeneous material and is thus trivial.
Hence it suffices to assume $\indL\geq2$.
In order to express $\det\AmatNeu_{\indL}(\nu),\det\AmatDir_{\indL}(\nu)$ we introduce the auxiliary matrices $\BmatNeu_{\indL}(\nu), \BmatNeu_{\indL}(\nu) \in\setC^{(2\indL-3)\times(2\indL-3)}$ as follows:
%\begin{align*}
%\Bmat_{\indL}:=
%\bpm \bessely_{m2} & -\besselj_{m2} & -\bessely_{2m}\\
%\frac{\epsilon_1}{\epsilon_2}\bessely'_{m2} & -\besselj'_{m2} & -\bessely'_{m2}\\
%& \besselj_{m3} & \bessely_{m3} & -\besselj_{m3} & -\bessely_{m3}\\
%& \frac{\epsilon_2}{\epsilon_3}\besselj'_{m3} & \frac{\epsilon_2}{\epsilon_3}\bessely'_{3m} & -\besselj'_{3m} & -\bessely'_{3m}\\
%&&& \besselj_{m4} & \bessely_{4m} & -\besselj_{m4} & -\bessely_{4m}\\
%&&& \frac{\epsilon_3}{\epsilon_4}\besselj'_{4m} & \frac{\epsilon_3}{\epsilon_4}\bessely'_{4m} & -\besselj'_{4m} & -\bessely'_{4m}\\
%&&&&&\dots&\dots&\dots&\dots\\
%&&&&&\dots&\dots&\dots&\dots\\
%&&&&&&& \besselj_{m(\indL+1)} & \bessely_{m(\indL+1)} \epm
%\end{align*}
\begin{align}\label{eq:BmatNeu}
\BmatNeu_{\indL}(\nu):=
\bpm \bessely_{m2} & -\besselj_{m2} & -\bessely_{m2}\\
\frac{\mu_1}{\mu_2}\bessely'_{m2} & -\besselj'_{m2} & -\bessely'_{m2}\\
& \besselj_{m3} & \bessely_{m3} & -\besselj_{m3} & -\bessely_{m3}\\
& \frac{\mu_2}{\mu_3}\besselj'_{m3} & \frac{\mu_2}{\mu_3}\bessely'_{m3} & -\besselj'_{m3} & -\bessely'_{m3}\\
&&& \besselj_{m4} & \bessely_{m4} & -\besselj_{m4} & -\bessely_{m4}\\
&&& \frac{\mu_3}{\mu_4}\besselj'_{m4} & \frac{\mu_3}{\mu_4}\bessely'_{m4} & -\besselj'_{m4} & -\bessely'_{m4}\\
&&&&&\dots&\dots&\dots&\dots\\
&&&&&\dots&\dots&\dots&\dots\\
&&&&&&& \besselj'_{m(\indL+1)} & \bessely'_{m(\indL+1)} \epm,
\end{align}
and
\begin{align}\label{eq:BmatDir}
\BmatDir_{\indL}(\nu):=
\bpm \bessely_{m2} & -\besselj_{m2} & -\bessely_{m2}\\
\frac{\epsilon_1}{\epsilon_2}\bessely'_{m2} & -\besselj'_{m2} & -\bessely'_{m2}\\
& \besselj_{m3} & \bessely_{m3} & -\besselj_{m3} & -\bessely_{m3}\\
& \frac{\epsilon_2}{\epsilon_3}\besselj'_{m3} & \frac{\epsilon_2}{\epsilon_3}\bessely'_{m3} & -\besselj'_{m3} & -\bessely'_{m3}\\
&&& \besselj_{m4} & \bessely_{m4} & -\besselj_{m4} & -\bessely_{m4}\\
&&& \frac{\epsilon_3}{\epsilon_4}\besselj'_{m4} & \frac{\epsilon_3}{\epsilon_4}\bessely'_{m4} & -\besselj'_{m4} & -\bessely'_{m4}\\
&&&&&\dots&\dots&\dots&\dots\\
&&&&&\dots&\dots&\dots&\dots\\
&&&&&&& \besselj_{m(\indL+1)} & \bessely_{m(\indL+1)} \epm,
\end{align}
\end{subequations}
which are obtained from $\AmatNeu_{\indL}(\nu)$ and $\AmatDir_{\indL}(\nu)$ respectively by a subsequent modification of the first column.
Note that we have no intrinsic interest in $\BmatNeu_{\indL}(\nu)$ and $\BmatDir_{\indL}(\nu)$, but we have to include them into the statement of the forthcoming \Cref{lem:det} for technical reasons.
\begin{lemma}\label{lem:det}
Let $p:=r_2/r_{\indL+1}$.
There exist $\Cdet>0$ and $f^\mathrm{Dir}_m,f^\mathrm{Neu}_m\in C(\Rp)$ such that
\begin{subequations}
\begin{align}\label{eq:induction-Dir}
\begin{aligned}
\Pi_{l=2}^{\indL} 
\Big(\frac{\pi zr_l}{2}\Big)
\det\AmatDir_{\indL}(z)&=\besselj_m(zr_{\indL+1})+f^\mathrm{Dir}_m(z),\\
|f^\mathrm{Dir}_m(z)| &\leq \Cdet\deltaem\max\{|\besselj_m(z)|,|\bessely_m(z)|\},\\
%|f^\mathrm{Dir}_m(z)| &\leq \Cdet\deltaem\max\{|\besselj_m(z r_{\indL+1})|,|\bessely_m(z r_{\indL+1})|\},\\
|z^{\indL}
\besselj_m(z q)\besselj'_m(z q)\det\BmatDir_{\indL}(z)|
&\leq \Cdet\max\{|\besselj_m(z)|,|\bessely_m(z)|\},
%&\leq \Cdet\max\{|\besselj_m(z r_{\indL+1})|,|\bessely_m(z r_{\indL+1})|\},
\end{aligned}
\end{align}
and
\begin{align}\label{eq:induction-Neu}
\begin{aligned}
\Pi_{l=2}^{\indL} 
\Big(\frac{\pi zr_l}{2}\Big)
\det\AmatNeu_{\indL}(z)&=\besselj_m(zr_{\indL+1})+f^\mathrm{Neu}_m(z),\\
|f^\mathrm{Neu}_m(z)| &\leq \Cdet\deltaem\max\{|\besselj'_m(z)|,|\bessely'_m(z)|\},\\
%|f^\mathrm{Neu}_m(z)| &\leq \Cdet\deltaem\max\{|\besselj'_m(z r_{\indL+1})|,|\bessely'_m(z r_{\indL+1})|\},\\
|z^{\indL}
\besselj_m(z q)\besselj'_m(z q)\det\BmatNeu_{\indL}(z)|
&\leq \Cdet\max\{|\besselj'_m(z)|,|\bessely'_m(z)|\},
%&\leq \Cdet\max\{|\besselj'_m(z r_{\indL+1})|,|\bessely'_m(z r_{\indL+1})|\},
\end{aligned}
\end{align}
\end{subequations}
for all $q\in [pr_2,r_2]$, $z>0$ and $m\in\Nz$.
\end{lemma}
\begin{proof}
We only present the proof for \eqref{eq:induction-Dir}.
To improve the readability we omit the super indices ``$\mathrm{Dir}$''.
We are going to prove the claim by induction in $\indL$.
For $\indL=2$ we compute that
\begin{align*}
\det\Amat_{2}=\frac{2}{\pi zr_2}\besselj_{m3}
+(1-\frac{\epsilon_1}{\epsilon_2})\besselj'_{m2}\bessely_{m2}\besselj_{m3}
+(\frac{\epsilon_1}{\epsilon_2}-1)\besselj_{m2}\besselj'_{m2}\bessely_{m3}
\end{align*}
and
\begin{align*}
\det\Bmat_{2}=\frac{2}{\pi zr_2}\bessely_{m3}
+(\frac{\epsilon_1}{\epsilon_2}-1)\besselj_{m2}\bessely'_{m2}\bessely_{m3}
+(1-\frac{\epsilon_1}{\epsilon_2})\bessely_{m2}\bessely'_{m2}\besselj_{m3},
\end{align*}
for which the claim follows with \Cref{lem:Crjy,lem:Crjj} and the Wronskian identity $\besselj_m(z)\bessely'_m(z)-\besselj'_m(z)\bessely_m(z)=2/(\pi z)$~\cite[Eq.~10.5.2]{NISTDLMF}.
Let now \eqref{eq:induction-Dir} be satisfied for $\indL-1\geq2$.
We refer to the $(2\indL-1)\times(2\indL-1)$ matrices build by the formulas \eqref{eq:AmatDir}, \eqref{eq:BmatDir} with the evaluation points $r_3,\dots,r_{\indL+2}$ (instead of $r_2,\dots,r_{\indL+1}$) as $\tilde\Amat_{\indL}$ and $\tilde\Bmat_{\indL}$ respectively.
We compute that
\begin{align*}
\det\Amat_{\indL+1}
=(\frac{\epsilon_1}{\epsilon_2}-1)\besselj_{m2}\besselj'_{m2}\det\tilde\Bmat_{\indL}(z)
+\frac{2}{\pi zr_2}\det\tilde\Amat_{\indL}(z)
+(1-\frac{\epsilon_1}{\epsilon_2})\besselj'_{m2}\bessely_{m2}\det\tilde\Amat_{\indL}(z)
\end{align*}
and
\begin{align*}
\det\Bmat_{\indL+1}
=(1-\frac{\epsilon_1}{\epsilon_2})\bessely_{m2}\bessely'_{m2}\det\tilde\Amat_{\indL}(z)
+\frac{2}{\pi zr_2}\det\tilde\Bmat_{\indL}(z)
+(\frac{\epsilon_1}{\epsilon_2}-1)\besselj_{m2}\bessely'_{m2}\det\tilde\Bmat_{\indL}(z).
\end{align*}
To prove the induction step for $\Amat_{\indL+1}$ we compute that
\begin{align*}
\Pi_{l=2}^{\indL+1}& \Big(\frac{\pi zr_l}{2}\Big)
\det\Amat_{\indL+1}(z) -\besselj_m(zr_{\indL+1})=
\Pi_{l=2}^{\indL+1} \Big(\frac{\pi zr_l}{2}\Big)
\big(\det\Amat_{\indL+1}(z) -\frac{2}{\pi zr_2}\det\Amat_{\indL}(z)\big)\\
&=\Pi_{l=2}^{\indL+1} \Big(\frac{\pi zr_l}{2}\Big)
\bigg( (\frac{\epsilon_1}{\epsilon_2}-1)\besselj_{m2}\besselj'_{m2}\det\tilde\Bmat_{\indL}(z)
+(1-\frac{\epsilon_1}{\epsilon_2})\besselj'_{m2}\bessely_{m2}\det\tilde\Amat_{\indL}(z) \bigg)
=:f_m(z)
\end{align*}
and estimate
\begin{align*}
|f_m(z)|&\leq 4\epsilon_0^{-1} \big(\frac{\pi r_{\indL+1}}{2}\Big)^{\indL} \deltaem
|z^{\indL} \besselj_m(z)\besselj'_m(z)\det\tilde\Bmat_{\indL}(z)|\\
&+2\epsilon_0^{-1}\pi \deltaem |zr_2\besselj_m'(zr_2)\bessely_m(zr_2)|
|\Pi_{l=3}^{\indL+1} \Big(\frac{\pi zr_l}{2}\Big)
\det\tilde\Amat_{\indL}(z)|.
\end{align*}
The claim follows now by applying \Cref{lem:Crjy} and the induction assumption.
The prove the induction step for $\Bmat_{\indL+1}$ we compute that
\begin{align*}
&|z^{\indL+1} \besselj_m(zq)\besselj'_m(zq)\det\Bmat_{\indL+1}(z)|\\
&\leq \frac{4\deltaem}{\epsilon_0} \Big(\frac{\pi r_2}{2}\Big)^{-\indL+1} 
|z^2\besselj_m(zq)\besselj'_m(zq)\bessely_m(zr_2)\bessely'_m(zr_2)|
|\Pi_{l=2}^{\indL} \Big(\frac{\pi zr_l}{2}\Big) \det\tilde\Amat_{\indL}(z)|\\
&+\frac{2}{\pi r_2} |z^{\indL} \besselj_m(zq)\besselj'_m(zq)\det\tilde\Bmat_{\indL}(z)|
+\frac{4\deltaem}{\epsilon_0} |z\besselj_m(zr_2)\bessely'_m(zr_2)|
|z^{\indL} \besselj_m(zq)\besselj'_m(zq)\det\tilde\Bmat_{\indL}(z)|.
\end{align*}
Note that $q\in[pr_2,r_2]\subset[p\frac{r_2}{r_3}r_3,r_3]$ (which is the resonon we introduced the $\indL$ independent constant $p$).
The claim follows now by applying \Cref{lem:Crjy,lem:Crjj,lem:Crrjjyy}, \eqref{eq:Wronskian} and the induction assumption.
Thus the proof is finished.
\end{proof}

\subsection{Eigenvalue gaps}\label{subsec:gaps}

Note that the eigenvalues $\nu>0$ of \eqref{eq:evps-scalar-vm} and \eqref{eq:evps-scalar-wm} are precisely the positive roots of $\det\AmatNeu_{\indL}(\cdot)$ and $\det\AmatDir_{\indL}(\cdot)$ respectively.
For each $m\in\Nz$ let $(\tilde\besseljz_{n,m})_{n\in\setN}$ be the non-decreasing sequence of the roots of
$\tilde\besselj_m(z):=\Pi_{l=2}^{\indL} \Big(\frac{\pi zr_l}{2}\Big) \det\AmatDir_{\indL}(z)$,
$(\tilde\besseljpz_{n,m})_{n\in\setN}$ be the non-decreasing sequence of the roots of
$\tilde\besselj'_m(z):=\Pi_{l=2}^{\indL} \Big(\frac{\pi zr_l}{2}\Big) \det\AmatNeu_{\indL}(z)$,
and $(\evsac_{n,m})_{n\in\setN}$ be the compound sequence of the former two ordered in a non-decreasing way.
Our objective is to prove a gap estimate for $(\evsac_{n,m})_{n\in\setN}$ being uniform in $m\in\Nz$ similar to the one stated in \Cref{prop:cdistbes} for $(\besseljzc_{n,m})_{n\in\setN}$.
The claim would follow easily, if $\evsac_{n,m}$ would converge to $\besseljzc_{n,m}$ uniformly in $n\in\setN$ and $m\in\Nz$ as $\deltaem\to0$.
However, this uniform convergence is actually not true.
Instead we have a uniform convergence in a distorted sense, as stated in \Cref{lem:ev-conv}, which suffices to establish \Cref{prop:evgap}.
As preparation we recall the following result from \cite{Hethcote70}.
\begin{lemma}[Thm.~1 of \cite{Hethcote70}]\label{lem:Hethcote}
In the interval $[b-s,b+s]$, suppose $\phi(x)=\psi(x)+\tau(x)$, where $\phi(x)$ is continuous, $\psi(x)$ is differentiable, $\psi(b)=0$, $\min|\psi'(x)|>0$, and $\max|\tau(x)|<\min\{\psi(b-s)|,|\psi(b+s)|\}$.
Then there exists a zero $c\in[b-s,b+s]$ of $\phi(x)$ such that $|c-b|\leq \max|\tau(x)|/\min|\psi'(x)|$.
\end{lemma}
Recall that the phase function $\phaseB$ and the negative roots $(\Aiz_n)_{n\in\setN}$, $(\Aipz_n)_{n\in\setN}$ of $\Ai(-\cdot)$ and $\Ai'(-\cdot)$ had been introduced in \Cref{sec:bessel}.
In addition, let
\begin{align*}
\rasy_{n,m}:=\omshift-\frac{\pi}{2}+\pi n=\Big(n+\frac{m}{2}-\frac{1}{4}\Big)\pi, \quad n\in\setN, m\in\Nz.
\end{align*}
\begin{lemma}\label{lem:ev-conv}
For each $\eta>0$ there exist $\tstar>1, \delta_0>0$ such that if $\deltaem<\delta_0$, then
\begin{subequations}
\begin{align}
\label{eq:ev-conv-a}
|\besseljzc_{n,0}-\evsac_{n,0}|&<\frac{\eta}{2} \quad
&&\forall n\in\setN,\\
\label{eq:ev-conv-b}
|\rasy_{n,m}-\phaseB(\besseljzc_{n,m})|, |\rasy_{n,m}-\phaseB(\evsac_{n,m})|&<\frac{\eta}{2} \quad
&&\forall n\in\setN,m\in\setN \colon\hspace{2.5mm}\quad \besseljzc_{n,m} > m+m^{1/3}\tstar.
\end{align}
For each $\tstar>1$, $\Col>1$, $\tilde\eta>0$ there exist $\mstar>1,\tilde\delta_0>0$ such that if $\deltaem<\tilde\delta_0$, then
\begin{align}
\label{eq:ev-conv-c}
\Big| \frac{\Aizc_n}{2^{1/3}}-\frac{\besseljzc_{n,m}-m}{m^{1/3}} \Big|,
\Big| \frac{\Aizc_n}{2^{1/3}}-\frac{\evsac_{n,m}-m}{m^{1/3}} \Big|&<\frac{\tilde\eta}{2} 
&&\forall n\in\setN,m>\mstar \colon\quad \besseljzc_{n,m} < m+m^{1/3}\tstar\Col,\\
\label{eq:ev-conv-d}
|\besseljzc_{n,m}-\evsac_{n,m}|&<\frac{\tilde\eta }{2}
&&\forall n\in\setN,m\leq\mstar \colon\quad \besseljzc_{n,m} < m+m^{1/3}\tstar\Col.
\end{align}
\end{subequations}
\end{lemma}
\begin{proof}
We only present the proof for the Dirichlet eigenvalues $\tilde\besseljz_{n,m}$.
Our main approach to obtain uniform estimates is to apply \Cref{lem:Hethcote} to treat all but a finite number of eigenvalues.
We derive the asymptotic values $\rasy_{n,m}$, bound $|\besseljz_{n,m}-\rasy_{n,m}|$ and $|\tilde\besseljz_{n,m}-\rasy_{n,m}|$ with the same technique and occasionally estimate $|\besseljz_{n,m}-\tilde\besseljz_{n,m}|\leq |\besseljz_{n,m}-\rasy_{n,m}|+|\tilde\besseljz_{n,m}-\rasy_{n,m}|$.
The uniform convergence of the remaining eigenvalues then follows due to their finite number.\\
\textit{1.\ case: $m=0$.}\quad 
Recall the existence of $C_0>0$ such that $|\besselj_0(z)-\big(\frac{2}{\pi z}\big)^{1/2}\cos(z-\bbomega_0)|\leq \frac{C_0}{z}$ for all $z>1$ (which follows, e.g., from \cite[Eq.~10.17.13]{NISTDLMF}).
We apply \Cref{lem:Hethcote} simultaneously with
\begin{align*}
\psi(z)&=\sqrt{2/\pi}\cos(z-\bbomega_0),&
\tilde\psi(z)&=\psi(z),\\
\tau(z)&=z^{1/2} \besselj_0(z) -\psi(z),&
\tilde\tau(z)&=z^{1/2} \besselj_0(z) -\psi(z) +z^{1/2} f^\mathrm{Dir}_0(z),\\
b&=\rasy_{n,0},\,\, n\in\Nz,&
%\bbomega_0+\pi/2+\pi n, n\in\Nz,&
\tilde b&=b,\\
s&\in (0,\min\{\cgap/2, \pi/8\}),&
\tilde s&=s,
\end{align*}
where $f_0^\mathrm{Dir}$ is as in \Cref{lem:det} and $\cgap$ is as in \Cref{prop:cdistbes}.
Note that $|\psi(\rasy_{n,0}\pm s)|=\cos(\pi/2+s)$ and $|\psi'(x)|\geq \sin(\pi/2+s)$, $x\in [\rasy_{n,0}-s,\rasy_{n,0}+s]$.
Furthermore
\begin{align*}
|\tau(z)|, |\tilde\tau(z)| \leq \frac{C_0}{z^{1/2}}
+\Cdet\deltaem \sup_{t>1}\max\{|t^{1/2}\besselj_0(t)|,|t^{1/2}\bessely_0(t)|\}
\quad\text{for all } z>1.
\end{align*}
Let $z_0>1$, $\delta_1>0$ be such that
\begin{align*}
\bigg(\frac{C_0}{z_0^{1/2}}
+\Cdet\delta_1 \sup_{t>1}\max\{|t^{1/2}\besselj_0(t)|,|t^{1/2}\bessely_0(t)|\}\bigg)
&<\cos(\pi/2+s),\\
\bigg(\frac{C_0}{z_0^{1/2}}
+\Cdet\delta_1 \sup_{t>1}\max\{|t^{1/2}\besselj_0(t)|,|t^{1/2}\bessely_0(t)|\}\bigg)
&<\frac{\sin(\pi/2+s)}{4}\eta,
\end{align*}
and $\deltaem<\delta_1$ henceforth.
Let $\nstar:=\min\{n\in\setN\colon \rasy_{n,0}\geq z_0+s\}$.
Then for each $n\geq\nstar$ \Cref{lem:Hethcote} yields the existence of
a root $\jtmp_{n,0} \in [\rasy_{n,0}-\eta/4,\rasy_{n,0}+\eta/4]$ of $\besselj_{n,0}$
and a root $\jttmp_{n,0} \in [\rasy_{n,0}-\eta/4,\rasy_{n,0}+\eta/4]$ of $\tilde\besselj_{n,0}$.
It also follows that $\besselj_{n,0}, \tilde\besselj_{n,0}$ have no roots in
$\{z\in\Rp\colon z\geq z_0+s\}\setminus \bigcup_{n\geq\nstar} [\rasy_{n,0}-s,\rasy_{n,0}+s]$.
Due to \Cref{prop:cdistbes} $\jtmp_{n,0}$ is the only root of $\besselj_{n,0}$ in $[\rasy_{n,0}-s,\rasy_{n,0}+s]$, $n\geq\nstar$.
(This could alternatively be reasoned with the monotonicity of $\besselj_{n,0}$ in $[\rasy_{n,0}-s,\rasy_{n,0}+s]$ obtained by an asymptotic approximation of $\besselj'_{n,0}$.)
Assume that $\tilde\besselj_{n,0}$ admits more than one root in $[\rasy_{n,0}-s,\rasy_{n,0}+s]$.
Since these roots are eigenvalues they are continuous with respect to $\epsilon\in L^\infty$ (see, e.g., \cite[Thm.~2,Thm.~3]{Karma96a} or \cite{Kato95}) and hence remain in $[\rasy_{n,0}-s,\rasy_{n,0}+s]$ as $\deltaem\to0$.
Note that the eigenvalues $\tilde\besseljz_{n,0}$ of [\eqref{eq:evps-scalar-m}, $m=0$] converge locally to the eigenvalues $\besseljz_{n,0}$ of [\eqref{eq:evps-scalar-m}, $m=0$, $\epsilon=\epsilon_0$] as $\deltaem\to0$.
Thus as $\deltaem\to0$ the only possible limit for roots of $\tilde\besselj_{n,0}$ in $[\rasy_{n,0}-s,\rasy_{n,0}+s]$ is $\jtmp_{n,0}$.
However, this would lead to a contradiction on the multiplicity of $\jtmp_{n,0}$.
Hence $\jttmp_{n,0}$ is the only root of $\tilde\besselj_{n,0}$ in $[\rasy_{n,0}-s,\rasy_{n,0}+s]$, $n\geq\nstar$.
Let $\nshift:=\#\{\besseljz_{n,0}\colon n\in\setN, \besseljz_{n,0}<z_0+s\}$.
Classical eigenvalue perturbation theory yields the existence of $\delta_2\in(0,\delta_1]$ such that if $\deltaem<\delta_2$, then [\eqref{eq:evps-scalar-m}, $m=0$] admits exactly $\nshift$ eigenvalues in $(0,z_0+s)$.
Henceforth let $\deltaem<\delta_2$.
The known asymptotic behavior of $\besseljz_{n,0}$ \cite[Eq.~10.21.19,10.21.20]{NISTDLMF} then yields that $\jtmp_{n,0}=\besseljz_{n,0}$ and $\jttmp_{n,0}=\tilde\besseljz_{n,0}$ for $n\geq\nstar$, i.e., there is no index shift in $n$.
The triangle inequality yields that $\besseljz_{n,0}-\tilde\besseljz_{n,0}<\eta/2$ for $n\geq\nstar$.
Now classical eigenvalue perturbation theory allows us to choose $\delta_3\in(0,\delta_2]$ such that $|\besseljz_{n,0}-\tilde\besseljz_{n,0}|<\eta/2$ for all $n=1,\dots,\nshift=\nstar-1$, $\deltaem<\delta_3$.
Hence the proof of the first case is finished.
\\
\textit{2.\ case: $n\in\setN,m\in\setN\colon \besseljz_{n,m} > m+m^{1/3}\tstar$.}\quad
We repeat the approach used for the previous case.
We apply \Cref{lem:Hethcote} simultaneously with
\begin{align*}
\psi(x)&=\tilde\psi(x)=\sqrt{2/\pi}\cos(x-\omshift),\\
\tau(x)&=\big((\phaseB^{-1}(x))^2-\mushift\big)^{1/4} \besselj_m\circ\phaseB^{-1}(x)
-\psi(x),\\
\tilde\tau(x)&=\big((\phaseB^{-1}(x))^2-\mushift\big)^{1/4} \besselj_m\circ\phaseB^{-1}(x)
-\psi(x)
+\big((\phaseB^{-1}(x))^2-\mushift\big)^{1/4} f^\mathrm{Dir}_m\circ\phaseB^{-1}(x),\\
b&=\tilde b=\rasy_{n,m},\,\, n\in\setN,\\
%&=\tilde b_{n,m}=\omshift+\pi/2+\pi n,\,\, n\in\setN,\\
s&=\tilde s=\frac{\pi}{8}.
\end{align*}
\Cref{lem:besselOscZoneII,lem:det} yield that
\begin{align*}
|\tau(x)|,|\tilde\tau(x)| \leq 
\frac{13}{12\sqrt{2\pi}\tstari^{3/2}}(1+\Cdet\deltaem)+\sqrt{\frac{2}{\pi}}\Cdet\deltaem
\quad\text{for all } x>m+m^{1/3}\tstari,\,\, \tstari>1.
\end{align*}
We choose $\tstari>1$ and $\delta_4\in(0,\delta_3]$ such that
\begin{align*}
\frac{13}{12\sqrt{2\pi}\tstari^{3/2}}(1+\Cdet\delta_4)+\sqrt{\frac{2}{\pi}}\Cdet\delta_4
&<\cos(\pi5/8),\\
\frac{13}{12\sqrt{2\pi}\tstari^{3/2}}(1+\Cdet\delta_4)+\sqrt{\frac{2}{\pi}}\Cdet\delta_4
&<\frac{\sin(\pi5/8)\eta}{2},
\end{align*}
and assume $\deltaem<\delta_4$ henceforth.
Let $\nstar>0$ be such that $\rasy_{\nstar,m}-s \geq \phaseB(m+m^{1/3}\tstari)$.
Thence for each $n\geq\nstar$ \Cref{lem:Hethcote} yields that $\besselj_m\circ\phaseB^{-1}$ and $\tilde\besselj_m\circ\phaseB^{-1}$ admit roots $\xtmp_{n,m},\xttmp_{n,m}\in [\rasy_{n,m}-s,\rasy_{n,m}+s]$ which satisfy $|\rasy_{n,m}-\xtmp_{n,m}|, |\rasy_{n,m}-\xttmp_{n,m}|<\eta$.
It also follows that $\besselj_m\circ\phaseB^{-1}, \tilde\besselj_m\circ\phaseB^{-1}$ have no roots in
$(\rasy_{\nstar,m}-s,+\infty)\setminus \bigcup_{n\geq\nstar} [\rasy_{n,m}-s,\rasy_{n,m}+s]$.
It can be deduced with \Cref{lem:besselOscZoneII}
that $\besselj_m\circ\phaseB^{-1}$ is monotone on $[\rasy_{n,m}-s,\rasy_{n,m}+s]$ and hence $\xtmp_{n,m}$ is the only root of $\besselj_m\circ\phaseB^{-1}$ in $[\rasy_{n,m}-s,\rasy_{n,m}+s]$.
(For the Neumann case we can use \eqref{eq:JYppp} and have to adjust $\tstar,\delta_4$.)
As applied in the previous case, a continuity and convergence argument for $\deltaem\to0$ shows that also $\xttmp_{n,m}$ is the only root of $\tilde\besselj_m\circ\phaseB^{-1}$ in $[\rasy_{n,m}-s,\rasy_{n,m}+s]$.
Thus all roots of $\besselj_m\circ\phaseB^{-1}$, $\tilde\besselj_m\circ\phaseB^{-1}$ in $(\rasy_{\nstar,m}-s,+\infty)$ are given by $\xtmp_{n,m}$ and $\xttmp_{n,m}$, $n\geq\nstar$ respectively.
The known asymptotic behavior of $\besseljz_{n,m}$ \cite[Eq.~10.21.19,10.21.20]{NISTDLMF} then yields that $\xtmp_{n,m}=\phaseB(\besseljz_{n,m})$ and $\xttmp_{n,m}=\phaseB(\tilde\besseljz_{n,m})$ for $n\geq\nstar$, i.e., there is no index shift in $n$.
It remains to compute that
\begin{align*}
\phaseB^{-1}(\rasy_{\nstar,m}-s)&=
\phaseB^{-1}(\rasy_{\nstar,m}-s)-\phaseB^{-1}\circ\phaseB(m+m^{1/3}\tstari)+m+m^{1/3}\tstari\\
&=\int_{\phaseB(m+m^{1/3}\tstari)}^{\rasy_{\nstar,m}-s} (\phaseB^{-1})'(z) \,dz +m+m^{1/3}\tstari\\
&\leq\int_{\phaseB(m+m^{1/3}\tstari)}^{\phaseB(m+m^{1/3}\tstari)+\pi} \frac{1}{\phaseB'\circ\phaseB^{-1}(z)} \,dz +m+m^{1/3}\tstari\\
&\leq \frac{\pi}{\phaseB'(m+m^{1/3}\tstari)} +m+m^{1/3}\tstari\\
&\leq \frac{\pi(m+m^{1/3}\tstari)}{m^{2/3}\tstari^{1/2}} +m+m^{1/3}\tstari
\leq 2\pi\tstari^{1/2}m^{1/3} +m+m^{1/3}\tstari.
\end{align*}
Hence the claim for the second case follows with $\tstar:=\tstari+2\pi\tstari^{1/2}$.\\
\textit{3.\ case: $n\in\setN,m>\mstar\colon \besseljz_{n,m} < m+m^{1/3}\tstar\Col$.}\quad
We repeat the approach used for the previous case.
Let $\nstar:=\max\{n\in\setN\colon \Aiz_n <\tstar\Col\}$,
$\hat n_*:=\max\{n\in\setN\colon \Aizc_n \leq\tstar\Col\}$,
and
$\cgapai:=\min\{\Aizc_1,\min\{\Aizc_{n+1}-\Aizc_n\colon n=1,\dots,\hat n_*-1\}\}$.
We apply \Cref{lem:Hethcote} simultaneously with
\begin{align*}
\psi(t)&= \tilde\psi(t)=2^{1/3}\Ai(-2^{1/3}t),\\
\tau(t)&= m^{1/3} \besselj_m(m+m^{1/3}t) -\psi(x),\\
\tilde\tau(x)&= m^{1/3} \besselj_m(m+m^{1/3}t) -\psi(x)
+m^{1/3} f^\mathrm{Dir}_m(m+m^{1/3}t),\\
b&=\tilde b=2^{-1/3}\Aiz_n,\,\, n=1,\dots,\nstar,\\
s&=\tilde s=\min\Big\{ \frac{\cgapai}{2^{1/3}4}, \frac{\tstar\Col-\Aiz_{\nstar}}{2^{1/3}}\Big\}.
\end{align*}
\Cref{lem:besselTransZone,lem:det} yield that
\begin{align*}
|\tau(t)|,|\tilde\tau(t)| \leq 
\frac{C_{[0,\tstar\Col]}}{m^{2/3}}
+\Cdet\deltaem \sup_{\tilde t\in[0,\tstar\Col]} \max \{|2^{1/3}\Ai(-2^{1/3}\tilde t)|,|2^{1/3}\Bi(-2^{1/3}\tilde t)|\}
\,\,\forall t\in [0,\tstar\Col].
\end{align*}
We choose $\mstar>1$ and $\tilde\delta_1>0$ such that
\begin{align*}
\frac{C_{[0,\tstar\Col]}}{m^{2/3}}
&<\min\{\psi(2^{-1/3}\Aiz_n\pm s)\colon n=1,\dots,\nstar\},\\
\frac{C_{[0,\tstar\Col]}}{m^{2/3}}
&<\frac{\tilde\eta}{2} \min\{\psi'(t)\colon t\in[2^{-1/3}\Aiz_n-s,2^{-1/3}\Aiz_n+s]\colon n=1,\dots,\nstar\}.
\end{align*}
and assume $m>\mstar$, $\deltaem<\tilde\delta_1$ henceforth.
Thence for each $n\leq\nstar$ \Cref{lem:Hethcote} yields that $\psi+\tau$ and $\tilde\psi+\tilde\tau$ admit roots $\troot_{n,m},\ttroot_{n,m}\in [2^{-1/3}\Aiz_n-s,2^{-1/3}\Aiz_n+s]$ which satisfy $|2^{-1/3}\Aiz_n-\troot_{n,m}|,|2^{-1/3}\Aiz_n-\ttroot_{n,m}|<\tilde\eta/2$.
By arguing the same way as in the previous cases we obtain that $\troot_{n,m}$ and $\ttroot_{n,m}$, $n=1,\dots,\nstar$ are the only roots of $\psi+\tau$ and $\tilde\psi+\tilde\tau$ in $[0,\tstar\Col]$.
The claim follows now from $\besseljz_{n,m}=m+m^{1/3}\troot_{n,m}$ and $\tilde\besseljz_{n,m}=m+m^{1/3}\ttroot_{n,m}$ for $n=1,\dots,\nstar$.
\\
\textit{4.\ case: $n\in\setN,m\leq\mstar\colon \besseljz_{n,m} \leq m+m^{1/3}\tstar$.}\quad
The index set for $(n,m)$ is finite and thus conventional eigenvalue perturbation theory yields the existence of $\tilde\delta_2\in(0,\tilde\delta_1]$ such that if $\deltaem<\tilde\delta_2$, then $|\besseljz_{n,m}-\tilde\besseljz_{n,m}|<\tilde\eta/2$.
\end{proof}
\begin{proposition}\label{prop:evgap}
There exist constants $\cgapt,\delta_0>0$ such that if $\deltaem<\delta_0$, then
\begin{align*}
|\evsac_{1,m}|\geq\cgapt,\quad
|\evsac_{n+1,m}-\evsac_{n,m}|\geq\cgapt,\quad
&\text{for all}\quad \delta\in(0,\delta_0),\indn\in\setN,\indm\in\Nz.
\end{align*}
\end{proposition}
\begin{proof}
The claim essentially follows from \Cref{lem:ev-conv} and the triangle inequality.
Let $\eta:=\min\{\frac{\cgap}{2},\frac{\pi}{2}\}$.
The first part of \Cref{lem:ev-conv} allows us to choose $\tstar>1$ and $\delta_1>0$ such that if $\deltaem<\delta_1$, then \eqref{eq:ev-conv-a}, \eqref{eq:ev-conv-b} hold.
Let $\deltaem<\delta_1$ henceforth.
Thus we estimate
\begin{align*}
|\evsac_{1,0}|
\geq |\besseljzc_{1,0}|-|\besseljzc_{1,0}-\evsac_{1,0}|
\geq \cgap-\eta/2 \geq \cgap/2,
\end{align*}
and
\begin{align*}
|\evsac_{n+1,0}-\evsac_{n,0}|
\geq 
|\besseljzc_{n+1,0}-\besseljzc_{n,0}|
-|\besseljzc_{n+1,0}-\evsac_{n+1,0}|
-|\besseljzc_{n,0}-\evsac_{n,0}|
\geq \cgap-\eta \geq \cgap/2.
\end{align*}
Furthermore,
\begin{align*}
\evsac_{n+1}-\evsac_{n,m}
&\geq \int^{\evsac_{n+1,m}}_{\evsac_{n,m}} \phaseB'(z)\,dz
= \phaseB(\evsac_{n+1,m})-\phaseB(\evsac_{n,m})\\
&\geq |\rasy_{n+1,m}-\rasy_{n,m}|-|\rasy_{n+1,m}
-\phaseB(\besseljzc_{n+1,m})|
-|\rasy_{n,m}-\phaseB(\besseljzc_{n,m})|
\geq \pi-\eta \geq \pi/2
\end{align*}
for all $n,m\in\setN$ such that $\besseljzc_{n,m}>m+m^{1/3}\tstar$.
Due to \cite[Eqs.~(9.9.6),(9.9.8)]{NISTDLMF} we can choose $\Col>1$ such that for each $m\in\setN$ there exists $n_+\in\setN$ with $m+m^{1/3}\tstar<\besseljzc_{n+,m}<m+m^{1/3}\tstar\Col$.
Let $\cgapai>0$ be defined as in the proof (3.\ case) of \Cref{lem:ev-conv} and $\tilde\eta:=\min\{2^{-1/3}\cgapai/2, \cgap/2\}$.
With $\tstar$ remaining it's role the second part of \Cref{lem:ev-conv} allows us to choose $\mstar>1$ and $\delta_2>0$ such that if $\deltaem<\delta_v$, then \eqref{eq:ev-conv-c}, \eqref{eq:ev-conv-d} hold.
Let $\deltaem<\delta_v$ henceforth.
We estimate that
\begin{align*}
\evsac_{1,m}
&\geq 2^{-1/3}m^{1/3}|\Aizc_{1}|
-m^{1/3}\big|2^{-1/3}\Aizc_{1}-\frac{\evsac_{1,m}-m}{m^{1/3}}\big|\\
&\geq (2^{-1/3}\cgapai-\tilde\eta) m^{1/3}
\geq 2^{-1/3}\cgapai/2,
\end{align*}
and
\begin{align*}
\evsac_{n+1,m}&-\evsac_{n,m}\\
&\geq 2^{-1/3}m^{1/3}|\Aizc_{n+1}-\Aizc_{n}|
-m^{1/3}\big|2^{-1/3}\Aizc_{n+1}-\frac{\evsac_{n+1,m}-m}{m^{1/3}}\big|
-m^{1/3}\big|2^{-1/3}\Aizc_{n}-\frac{\evsac_{n,m}-m}{m^{1/3}}\big|\\
&\geq (2^{-1/3}\cgapai-\tilde\eta) m^{1/3}
\geq 2^{-1/3}\cgapai/2
\end{align*}
for all $n\in\setN,m>\mstar$ such that $\besseljzc_{n,m}<m+m^{1/3}\tstar\Col$.
Since for each $m>\mstar$ there exists $n_+\in\setN$ with $m+m^{1/3}\tstar<\besseljzc_{n+,m}<m+m^{1/3}\tstar\Col$ it follows that
\begin{align*}
|\evsac_{1,m}|, |\evsac_{n+1,m}-\evsac_{n,m}| \geq \min\{\pi/2, 2^{-1/3}\cgapai/2\}
\quad\text{for all } n\in\setN,m>\mstar.
\end{align*}
It remains to note that
\begin{align*}
|\evsac_{1,m}|
\geq |\besseljzc_{1,m}|-|\besseljzc_{1,m}-\evsac_{1,m}|
\geq \cgap-\eta/2 \geq \cgap/2,
\end{align*}
and
\begin{align*}
|\evsac_{n+1,m}-\evsac_{n,m}|
\geq 
|\besseljzc_{n+1,m}-\besseljzc_{n,m}|
-|\besseljzc_{n+1,m}-\evsac_{n+1,m}|
-|\besseljzc_{n,m}-\evsac_{n,m}|
\geq \cgap-\eta \geq \cgap/2
\end{align*}
for all $n\in\setN,m\leq\mstar$ such that $\besseljzc_{n,m}<m+m^{1/3}\tstar\Col$.
The claim follows now with $\delta_0:=\min\{\delta_1,\delta_2\}$ and
$\cgapt:=\min\{\cgap/2,\pi/2, 2^{-1/3}\cgapai/2\}$.
\end{proof}

\subsection{Traces of eigenfunctions}\label{subsec:traceofefs}

\begin{lemma}\label{lem:vector}
There exist $\Cvec,\delta_0>0$ such that if $\deltaem<\delta_0$ and $(\nu,\wvec)$, $\nu>0$ is a nontrivial solution to $\AmatDir_{\indL}(\nu)\wvec=0$ or $(\nu,\vvec)$, $\nu>0$ is a nontrivial solution to $\AmatNeu_{\indL}(\nu)\vvec=0$, then $a_{\indL}\neq0$ and after normalization of $\wvec$ or $\vvec$ respectively with $a_{\indL}:=1$ it holds that
$|a_1-1|,|(a_l,b_l)-(1,0)|\leq \Cvec \deltaem$ for all $l=2,\dots,\indL$.
\end{lemma}
\begin{proof}
We only present the proof for the Dirichlet case.\\

\emph{1.~step: $a_{\indL},b_{\indL}$.}\quad
Let $\delta_1\in(0,1/\Cdet)$ and $\deltaem<\delta_1$ henceforth.
Recall that $r_{\indL+1}=1$.
Assume that $|\bessely_m(\nu)|<|\besselj_m(\nu)|$.
Then due to \Cref{lem:det} it follows that $|\besselj_m(\nu)| \leq \Cdet \deltaem |\besselj_m(\nu)|$ $<$ $|\besselj_m(\nu)|$, which is a contradiction.
Hence it holds that $|\bessely_m(\nu)|\geq|\besselj_m(\nu)|$ and thus $|b_{\indL}|\leq \Cdet\deltaem |a_{\indL}|$.
We normalize $a_{\indL}=1$ and hence $|(a_{\indL},b_{\indL})-(1,0)|\leq \Cdet\deltaem$.
\\
\emph{2.~step: $a_{l},b_{l},l=2,\dots,\indL-1$.}\quad
Having already obtained $(a_{\indL},b_{\indL})$ we can compute $(a_{\indL-1},b_{\indL-1})$ from $\AmatDir_{\indL}(\nu)$:
\begin{align*}
\bpm \besselj_{m\indL} & \bessely_{m\indL} \\
\frac{\epsilon_{\indL-1}}{\epsilon_{\indL}} \besselj'_{m\indL} & \frac{\epsilon_{\indL-1}}{\epsilon_{\indL}}\bessely'_{m\indL} \epm
\bpm a_{\indL-1} \\ b_{\indL-1} \epm &= 
\bpm \besselj_{m\indL} & \bessely_{m\indL} \\
\besselj'_{m\indL} & \bessely'_{m\indL} \epm
\bpm a_{\indL} \\ b_{\indL} \epm,
\end{align*}
i.e.,
\begin{align*}
\bpm a_{\indL-1} \\ b_{\indL-1} \epm &=
\frac{\epsilon_{\indL}}{\epsilon_{\indL-1}} \bpm a_{\indL} \\ b_{\indL} \epm +
\frac{\epsilon_{\indL}}{\epsilon_{\indL-1}}
\Big( \frac{\epsilon_{\indL-1}}{\epsilon_{\indL}}-1 \Big)
\frac{\pi}{2}
\bpm r\besselj_{m\indL}\bessely'_{m\indL} \\ r\besselj_{m\indL}\besselj'_{m\indL} \epm
a_{\indL}.
\end{align*}
By the first step, \Cref{lem:Crjy,lem:Crjj} and induction we obtain the existence of a constant $\Cvec>0$ (which is independent of $m$ and $\nu$) such that
\begin{align*}
|(a_l,b_l)-(1,0)|\leq \Cvec \deltaem \quad\text{for all }l=2,\dots,\indL.
\end{align*}
\\
\emph{3.~step: $a_{1}$.}\quad
If $\nu r_2\in [m,m+m^{1/3}\tstar]$, then \eqref{eq:SherJ}, \eqref{eq:SherJp} yield the existence of a constant $\clb$ such that $\max\{|r^{1/3}\besselj_{m2}|,|r^{2/3}\besselj'_{m2}|\}>\clb$.
If $|r^{1/3}\besselj_{m2}|>\clb$, we can express $a_1=a_2+\frac{r^{1/3}\bessely_{m2}}{r^{1/3}\besselj_{m2}}b_2$ and then \eqref{eq:SherY}, \eqref{eq:SherYp} yield that $|a_1-1|\leq \Cvec \deltaem$ (with an updated constant $\Cvec$).
The reasoning in the case $|r^{2/3}\besselj'_{m2}|>\clb$ is similar.
Like-wise, for $\nu r_2\in [m+m^{1/3}\tstar,+\infty)$ (and large enough $\tstar$) we obtain the claim by means of \Cref{lem:besselOscZoneII}.
It remains to discuss the case $\nu r_2\in (0,m]$.
Note that $\bessely_{m2}\neq0$ and hence the first row of $\AmatDir_{\indL}(\nu)\wvec=0$ yields $b_2=\frac{\besselj_{m2}}{\bessely_{m2}}(a_1-a_2)$.
Plugging this into the second row of $\AmatDir_{\indL}(\nu)\wvec=0$ we obtain by means of \eqref{eq:Wronskian} and \Cref{lem:Crjy} that
\begin{align*}
|a_1-a_2|&=|1-\frac{\epsilon_1}{\epsilon_2}|
\bigg|\frac{\epsilon_1}{\epsilon_2}-\frac{\besselj_{m2}\bessely_{m2}'}{\besselj'_{m2}\bessely_{m2}} \bigg|^{-1}|a_2|
\leq 4\epsilon_0\deltaem \bigg|\frac{\epsilon_1}{\epsilon_2}-1+\frac{2}{\pi}\frac{1}{\nu r_2\besselj_{m2}'\bessely_{m2}}\bigg|^{-1}|a_2|\\
&\leq 4\epsilon_0\deltaem \bigg|\frac{2}{\pi}\frac{1}{\Crjy}-4\epsilon_0\deltaem \bigg|^{-1}|a_2|
\leq 4\epsilon_0\pi\Crjy\deltaem|a_2|
\end{align*}
for all $\deltaem<\delta_2:=\min\{\delta_1,\frac{1}{4\epsilon_0\pi\Crjy}\}$, what we assume henceforth.
Thus in all cases $|a_1-1|\leq\Cvec\deltaem$ (with an updated constant $\Cvec$).\\
\end{proof}
\begin{proposition}\label{prop:eftraces}
There exist $\Ceftrace,\delta_0>0$ such that if $\deltaem<\delta_0$, then for each $m\in\Nz$ and each eigenpair
$(\evsaneu_{n,m},v_{n,m})\in\Rp\times \Vm\setminus\{0\}$ to \eqref{eq:evps-scalar-vm} and
$(\evsadir_{n,m},w_{n,m})\in\Rp\times \Wm\setminus\{0\}$ to \eqref{eq:evps-scalar-wm} it holds that
$\|v_{n,m}\|_{\intf} \leq \Ceftrace n^{1/6} \|v_{n,m}\|_{\domcs}$ and
$\|w_{n,m}\|_{\intf} \leq \Ceftrace n^{1/6}\|w_{n,m}\|_{\domcs}$.
\end{proposition}
\begin{proof}
We start with analyzing the Dirichlet case, i.e., \eqref{eq:evps-scalar-wm}.
Recall from \Cref{subsec:reductiontomatrix} that
\begin{align*}
e^{-im\theta}w_{n,m}|_{(r_1,r_{2})}(r)=a_{n1}\besselj_m(\evsadir_{n,m} r),\qquad
e^{-im\theta}w_{n,m}|_{(r_l,r_{l+1})}(r)=a_{nl}\besselj_m(\evsadir_{n,m} r)+b_{nl}\bessely_m(\evsadir_{n,m} r),
\end{align*}
$l=2,\dots,\indL$ with constants $a_{nl},b_{nl}\in\setC$,
and $\wvec$ $:=$ $(a_{n1},a_{n2},b_{n2},\dots,a_{n\indL},b_{n\indL})^\top$ $\in$ $\setC^{2\indL-1}$ solves $\AmatDir_{\indL}(\evsadir_{n,m})\wvec=0$ with $\AmatDir_{\indL}(\evsadir_{n,m})$ as defined in \eqref{eq:AmatDir}.
Let $(\Cvec,\delta_1)$ be as in \Cref{lem:vector} and $\deltaem<\delta_1$ henceforth, such that after a normalization we have $|a_{n1}-1|,|(a_{nl},b_{nl})-(1,0)|\leq\Cvec\deltaem$, $l=2,\dots,\indL$.
We consider $m\in\setN$ and will treat the case $m=0$ separately.
We distinguish the three cases $n<\Cnl m$, $n\in [\Cnl m,\Cnu m+1]$ and $n>\Cnu m+1$, where the constants $\Cnu>\Cnl>0$ will be specified in the course of our analysis.
To this end let us note that there exist constants $\cAil,\CAiu>0$ such that
\begin{align*}
	\Aiz_n,\Aiz'_n \in [\cAil n^{2/3}, \CAiu n^{2/3}] \quad\text{for all } n\in\setN,
\end{align*}
see, e.g., \cite[Sect.\ 9.9(iv)]{NISTDLMF}, \cite{PittalugaSacripante91}, \cite{Krasikov14a} and note the interlacing of $\Aiz_n$ and $\Aiz'_n$.
Furthermore, recall from \cite{QuWong99} that
\begin{align}\label{eq:QuWong99}
m+\frac{\Aiz_n}{2^{1/3}}m^{1/3}
< \besseljz_{m,n}
< m+\frac{\Aiz_n}{2^{1/3}}m^{1/3} +\frac{3\cdot2^{1/3}}{20}\frac{\Aiz_n^2}{m^{1/3}}
\quad\text{for all }n,m\in\setN.
\end{align}
Let $\Cnl>0$ be such that
$\big(1+\frac{\CAiu \Cnl^{2(3}}{2^{1/3}} +\frac{3\cdot2^{1/3}}{20} \CAiu^2 \Cnl^{4/3}\big) r_{\indL}<1$ (which is possible because $r_{\indL}<1$).
We employ the abbreviations
$\besselj_{nml}:=\besselj_{m}(\evsadir_{n,m}r_l)$ and
$\bessely_{nml}:=\bessely_{m}(\evsadir_{n,m}r_l)$.\\

\emph{1.~case: $n<\Cnl m$}.\quad
First, we show with complete induction that there exists a constant $\Cab>0$ such that
\begin{align}\label{eq:pointvalues-aux}
%|\bessely_{ml}b_l| \leq \Cab|\besselj_{ml}\deltaem|\\
|\bessely_{nml}b_l| \leq \Cab|\besselj_{nml}|\deltaem
%|\bessely_{m}(\evsadir_{n,m}r_l)b_l| \leq \Cab|\besselj_{m}(\evsadir_{n,m}r_l)|\deltaem
\end{align}
for all $l\in\{2,\dots,\indL\}$ and $n<\Cnl m$, $m\in\setN$.
For $l=2$ the first row of $\AmatDir_{\indL}(\evsadir_{n,m})$ yields $|\bessely_{nm2}b_2| \leq |\besselj_{nm2}(a_{n2}-a_{n1})| \leq 2\Cvec \besselj_{nm2} \deltaem$.
For an index $l>2$ the $(2l-3)$th row of $\AmatDir_{\indL}(\evsadir_{n,m})$ yields
\begin{align*}
|\bessely_{nml}b_{nl}| \leq |\besselj_{nml}(a_{nl}-a_{n(l-1)})|+|\bessely_{nml}b_{n(l-1)}|
&\leq 2\Cvec \besselj_{nml}\deltaem + \frac{|\bessely_{nml}|}{|\bessely_{nm(l-1)}|} \besselj_{nm(l-1)} \Cab\deltaem\\
&\leq 2\Cvec \besselj_{nml}\deltaem + \frac{|\bessely_{nml}|}{|\bessely_{nm(l-1)}|} \frac{\besselj_{nm(l-1)}}{\besselj_{nml}} \besselj_{nml} \Cab\deltaem\\
&\leq (2\Cvec+\Cab) \besselj_{nml}\deltaem,
\end{align*}
due to the induction assumption and the monotonicity of $\besselj_m,\bessely_m$ in $[0,m]$.
Thus \eqref{eq:pointvalues-aux} is established.
For $l=2,\dots,\indL-1$ we estimate that
\begin{align}
\label{eq:infl-est}
\|w_{n,m}\|_{\intf_l}
\leq r_l^{1/2} (|\besselj_{nml}a_{nl}|+|\bessely_{nml}b_{nl}|)
&\leq (1+(\Cab+1)\deltaem) r_l^{1/2} |\besselj_{nml}|\\
\nonumber
&\leq \frac{(1+(\Cab+1)\deltaem)}{r_{l+1}-r_l} \|\cdot^{1/2}\besselj_{m}(\evsadir_{n,m}\cdot)\|_{L^2(r_l,r_{l+1})}
\end{align}
and
\begin{align*}
\|w_{n,m}\|_{L^2(\domcs_l)}
&\geq \|\cdot^{1/2}\besselj_{m}(\evsadir_{n,m}\cdot)\|_{L^2(r_l,r_{l+1})} |a_{nl}|
-\|\cdot^{1/2}\bessely_{m}(\evsadir_{n,m}\cdot)\|_{L^2(r_l,r_{l+1})} |b_{nl}|\\
&\geq \|\cdot^{1/2}\besselj_{m}(\evsadir_{n,m}\cdot)\|_{L^2(r_l,r_{l+1})} (1-\deltaem)
-\|\cdot^{1/2}\|_{L^2(r_l,r_{l+1})} |\bessely_{nml}| |b_{nl}|\\
&\geq \|\cdot^{1/2}\besselj_{m}(\evsadir_{n,m}\cdot)\|_{L^2(r_l,r_{l+1})} (1-\deltaem)
-\|\cdot^{1/2}\|_{L^2(r_l,r_{l+1})} \Cab \deltaem |\besselj_{nml}|\\
&\geq \|\cdot^{1/2}\besselj_{m}(\evsadir_{n,m}\cdot)\|_{L^2(r_l,r_{l+1})} (1-(\Cab+1)\deltaem)
\end{align*}
due to the monotonicity properties of Bessel functions.
Let $\delta_2\in(0,\delta_1)$ be such that $1-(\Cab+1)\delta_2>0$ and $\deltaem\in(0,\delta_2)$ henceforth.
Then $\|w_{n,m}\|_{L^2(\intf_l)} \leq C\|w_{n,m}\|_{L^2(\domcs)}$ with $C=\frac{(1+(\Cab+1)\deltaem)}{(1-(\Cab+1)\deltaem)} \frac{1}{r_{l+1}-r_l}$, $l=2,\dots,\indL-1$.
For $l=\indL$ we estimate identically, but use the integration interval
$\Big(r_{\indL}, 1/\big(1+\frac{\CAiu \Cnl^{2(3}}{2^{1/3}} +\frac{3\cdot2^{1/3}}{20} \CAiu^2 \Cnl^{4/3}\big)\Big)$
instead of $(r_l,r_{l+1})$ (to stay in the monotonicity zone).
Thus the claim for the first case $n<\Cnl m$ follows.
\\

\emph{2.~case: $n>\Cnu m+1$}.\quad
We choose $\Cnu>1$ large enough to satisfy several conditions.
The first is that
\begin{subequations}
\begin{align}
\big(1+\frac{\cAil\Cnu^{2/3}}{2^{1/3}}\big)r_2>1.
%,\\
%\sqrt{\frac{2}{\pi}}\sin(\pi/4)-\frac{7}{\sqrt{2\pi}}\frac{1}{(\frac{\cAil \Cnu^{2/3}}{2^{1/3}})^{1/2}}>\frac{1}{\sqrt{\pi}}\sin(\pi/4).
\end{align}
Second, let $T>1$ be large enough such that
\begin{align}\label{eq:T}
\sqrt{\frac{2}{\pi}}\sin(\pi/4)-\frac{7}{\sqrt{2\pi}}\frac{1}{T^{1/2}}>\frac{1}{\sqrt{\pi}}\sin(\pi/4).
\end{align}
We choose $\delta_3\in(0,\delta_2)$ and $\tstar>T$ according to \Cref{lem:ev-conv} for $\eta<\pi/2$ with $\eta$ to be specified later,
and demand that
\begin{align}\label{eq:CnuTstar}
\frac{\cAil\Cnu^{2/3}}{2^{1/3}}>\tstar,
\end{align}
\end{subequations}
which guarantees $\besseljz_{n,m}>m+m^{1/3}\tstar$ (appearing in \eqref{eq:ev-conv-b}).
Due to our choice of $\Cnu$ it follows with \Cref{lem:besselOscZoneII} that
\begin{align*}
\|w_{n,m}\|_{\intf} \leq \frac{1}{\evsadir_{n,m}^{1/2}}
\frac{(\indL-1)\big(1+(1+\Cvec)\deltaem\big)r_{\indL}^{1/2}}{r_2^{1/2}}
\frac{\sqrt{\frac{2}{\pi}}+\frac{13}{12\sqrt{2\pi}\Big(\big(1+\frac{\cAil\Cnu^{2/3}}{2^{1/3}}\big)r_2-1\Big)^{3/2}}}%
{\left(1-\big(1+\frac{\cAil\Cnu^{2/3}}{2^{1/3}}\big)^{-2}r_2^{-2}\right)^{1/4}}.
\end{align*}
Note that in the forthcoming estimates whenever possible will use the constant $\Cnl$ (instead of $\Cnu$) to reuse the estimates for the third case later on.
To estimate $\|w_{n,m}\|_{\domcs}$ we note that the techniques applied for the ``1.~case: $n<\Cnl m$'' yield that
\begin{align*}
\|w_{n,m}\|_{B_{m/\evsadir_{n,m}}} \geq \big(1-(1+\Cvec)\deltaem\big) \|\cdot^{1/2}\besselj_m(\evsadir_{n,m}\cdot)\|_{L^2(0,m/\evsadir_{n,m})}.
\end{align*}
Next, we have that
\begin{align*}
\|w_{n,m}\|_{B_1 \setminus B_{m/\evsadir_{n,m}}} \geq
(1-\deltaem) \|\cdot^{1/2}\besselj_m(\evsadir_{n,m}\cdot)\|_{L^2(m/\evsadir_{n,m},1)}
-\Cvec\deltaem \|\cdot^{1/2}\bessely_m(\evsadir_{n,m}\cdot)\|_{L^2(m/\evsadir_{n,m},1)}.
\end{align*}
Combining the previous two estimates we obtain that
\begin{align*}
\|w_{n,m}\|_{\domcs}^2 \geq
\frac{\big(1-(1+\Cvec)\deltaem\big)^2}{2} \|\cdot^{1/2}\besselj_m(\evsadir_{n,m}\cdot)\|_{L^2(0,1)}^2
-\Cvec^2\deltaem^2 \|\cdot^{1/2}\bessely_m(\evsadir_{n,m}\cdot)\|_{L^2(m/\evsadir_{n,m},1)}^2.
\end{align*}
Using \cite[Eq.~10.22.5]{NISTDLMF} (together with \cite[10.6.2]{NISTDLMF})
\begin{align*}
\int^z rF_m(ar)^2 \, dr
&=\frac{z^2}{2} \big( F_m(az)^2-F_{m-1}(az)F_{m+1}(az) \big)\\
&=\frac{z^2}{2} \Big( F_m(az)^2-\big(F'_{m}(az)+\frac{m}{az}F_m(az)\big)\big(-F'_{m}(az)+\frac{m}{az}F_m(az)\big) \Big),\quad a>0,\quad F=J,Y,
%\int^z r\besselj_m(ar)^2 \, dr=\frac{z^2}{2} \big( \besselj_m(az)^2-\besselj_{m-1}(az)\besselj_{m+1}(az) \big), a>0
\end{align*}
we have that
%\cite[10.22.37]{NISTDLMF}: $\int_0^1 r\besselj_m(\besseljz_{m,n}r)^2 \, dr=\frac{\besselj'_m(\besseljz_{m,n})^2}{2}$\\
\begin{align*}
2\|\cdot^{1/2}&\besselj_m(\evsadir_{n,m}\cdot)\|_{L^2(0,1)}^2\\
&=2\big(\|\cdot^{1/2}\besselj_m(\besseljz_{n,m}\cdot)\|_{L^2(0,1)}^2
+\|\cdot^{1/2}\besselj_m(\evsadir_{n,m}\cdot)\|_{L^2(0,1)}^2
-\|\cdot^{1/2}\besselj_m(\besseljz_{n,m}\cdot)\|_{L^2(0,1)}^2\big)\\
&=\besselj'_m(\besseljz_{m,n})^2
+\big(1-\frac{m^2}{\evsadir_{n,m}^2}\big)\besselj_m(\evsadir_{n,m})^2
+\besselj'_m(\evsadir_{m,n})^2-\besselj'_m(\besseljz_{m,n})^2\\
%-\besselj_{m-1}(\evsadir_{n,m})\besselj_{m+1}(\evsadir_{n,m})
%-\besselj_m(\besseljz_{m,n})^2-\besselj_{m-1}(\besseljz_{m,n})\besselj_{m+1}(\besseljz_{m,n}).
&=\besselj'_m(\besseljz_{m,n})^2
+\big(1-\frac{m^2}{\evsadir_{n,m}^2}\big)
\big(\besselj_m(\evsadir_{n,m})-\besselj_m(\besseljz_{n,m})\big)
\big(\besselj_m(\evsadir_{n,m})+\besselj_m(\besseljz_{n,m})\big)\\
&\hspace{37mm}+\big(\besselj'_m(\evsadir_{m,n})-\besselj'_m(\besseljz_{m,n})\big)
\big(\besselj'_m(\evsadir_{m,n})+\besselj'_m(\besseljz_{m,n})\big).
\end{align*}
First we estimate $|\besselj'_m(\besseljz_{m,n})|$ from below, after which we will show that the perturbation terms are sufficiently small.
Applying \Cref{lem:besselOscZoneII} yields that
\begin{align*}
|\besselj'_m(\besseljz_{m,n})| &\geq \frac{1}{\besseljz_{n,m}^{1/2}}
\left(1-\frac{1}{1+\cAil\Cnl^{2/3}2^{-1/3}}\right)
%\frac{1+\cAil\Cnl^{2/3}2^{-1/3}}{\cAil\Cnl^{2/3}2^{-1/3}}
%\frac{(r^2-\mushift)^{1/4}}{r}
\left(\sqrt{\frac{2}{\pi}}\sin(\pi/4)-\frac{7}{\sqrt{2\pi}}\frac{1}{
%(\frac{\cAil (\Cnu m)^{2/3}}{2^{1/3}})
\tstar
^{1/2}}\right)\\
&\geq \frac{1}{\besseljz_{n,m}^{1/2}}
\left(1-\frac{1}{1+\cAil\Cnl^{2/3}2^{-1/3}}\right)
\frac{1}{\sqrt{\pi}} \sin(\pi/4).
%\\
%&\geq \frac{1}{\besseljz_{n,m}^{1/2}}
%\left(\sqrt{\frac{2}{\pi}}\sin(\pi/4)-\frac{7}{\sqrt{2\pi}}\frac{1}{(\frac{\cAil \Cnu^{2/3}}{2^{1/3}})^{1/2}}\right).
\end{align*}
To extract from \Cref{lem:besselOscZoneII} an $m$-uniform estimate we have just used that $\besseljz_{n,m}>m(1+\frac{\cAil \Cnl^{2/3}}{2^{1/3}})$ with the appearing constant being larger than one.
To obtain a similar estimate for $\evsadir_{n,m}$ we proceed as follows.
Our assumptions ensure that $\besseljz_{n-1,m}>m+m^{1/3}\tstar$.
Hence, using \eqref{eq:ev-conv-b} of \Cref{lem:ev-conv} we have
\begin{align*}
\evsadir_{n,m}-\besseljz_{n-1,m}\geq\int^{\evsadir_{n,m}}_{\besseljz_{n-1,m}} \phaseB'(z)\,dx
=\phaseB(\evsadir_{n,m})-\phaseB(\besseljz_{n-1,m})
\geq \pi-\eta \geq0.
\end{align*}
Now we can estimate
\begin{align*}
|\besselj_m(\besseljz_{n,m})-\besselj_m(\evsadir_{n,m})|
&\leq |\besseljz_{n,m}-\evsadir_{n,m}| \sup_{z\in [\min\{\besseljz_{n,m},\evsadir_{n,m}\},\max\{\besseljz_{n,m},\evsadir_{n,m}\}]} |\besselj_m'(z)|\\
&\leq \eta \sup_{t\in [\min\{\phaseB(\besseljz_{n,m}),\phaseB(\evsadir_{n,m})\},\max\{\phaseB(\besseljz_{n,m}),\phaseB(\evsadir_{n,m})\}]} |(\phaseB^{-1})'(t)|
\sup_{z>\besseljz_{n-1,m}} |\besselj_m'(z)|\\
&= \eta \sup_{z\in [\min\{\besseljz_{n,m},\evsadir_{n,m}\},\max\{\besseljz_{n,m},\evsadir_{n,m}\}] } \frac{z}{\sqrt{z^2-\mushift}} %|1/\phaseB'(z)|
\sup_{z>\besseljz_{n-1,m}} |\besselj_m'(z)|\\
&\leq \eta \frac{\besseljz_{n-1,m}}{\sqrt{\besseljz_{n-1,m}^2-\mushift}}
\sup_{z>\besseljz_{n-1,m}} |\besselj_m'(z)|\\
&\leq \eta \frac{1+\cAil\Cnl^{2/3}2^{-1/3}}{\sqrt{(1+\cAil\Cnl^{2/3}2^{-1/3})^2-1}}
\frac{\sqrt{2/\pi}+\frac{7}{\sqrt{2\pi}} \frac{1}{
%(\frac{\cAil \Cnu^{2/3}}{2^{1/3}})
\tstar
^{1/2}}
}{\besseljz_{n-1,m}^{1/2}}.
% ???
%\frac{\sqrt{\frac{2}{\pi}}+\frac{13}{12\sqrt{2\pi}\big(\frac{\cAil\Cnu^{2/3}}{2^{1/3}}\big)^{3/2}}}
\end{align*}
Since the previous obtained estimate contains the factor $\besseljz_{n-1,m}^{-1/2}$ instead of $\besseljz_{n,m}^{-1/2}$, we compute
\begin{align*}
\frac{\besseljz_{n,m}}{\besseljz_{n-1,m}}
=1+\frac{\besseljz_{n,m}-\besseljz_{n-1,m}}{\besseljz_{n-1,m}}
&\leq 1+\frac{\phaseB(\besseljz_{n,m})-\phaseB(\besseljz_{n-1,m})}{\besseljz_{n-1,m}}
\frac{\besseljz_{n-1,m}}{\sqrt{\besseljz_{n-1,m}^2-\mushift}}\\
&\leq 1+\frac{3\pi}{2}
%&\leq 1+\frac{\pi+\eta}{\besseljz_{1,0}}
\frac{1}{\sqrt{(1+\cAil\Cnl^{2/3}2^{-1/3})^2-1}}
=:\Cjrat.
\end{align*}
Further, note that \Cref{lem:besselOscZoneII} yields
\begin{align*}
|\besselj_m(\besseljz_{n,m})+\besselj_m(\evsadir_{n,m})| \leq
\besseljz_{n,m}^{-1/2}\big(1+\Cjrat^{1/2}\big)
\frac{\sqrt{\frac{2}{\pi}}+\frac{13}{12\sqrt{2\pi}
%\big(\frac{\cAil\Cnu^{2/3}}{2^{1/3}}\big)
\tstar
^{3/2}}}%
{\left(1-\big(1+\frac{\cAil\Cnl^{2/3}}{2^{1/3}}\big)^{-2}\right)^{1/4}}.
\end{align*}
Thus 
\begin{align*}
&|\big(1-\frac{m^2}{\evsadir_{n,m}^2}\big)
\big(\besselj_m(\evsadir_{n,m})-\besselj_m(\besseljz_{n,m})\big)
\big(\besselj_m(\evsadir_{n,m})+\besselj_m(\besseljz_{n,m})\big)|\\
&\leq 
\frac{\eta}{\besseljz_{n,m}}
\underbrace{\big(\Cjrat+\Cjrat^{3/2}\big)
\frac{\sqrt{\frac{2}{\pi}}+\frac{13}{12\sqrt{2\pi}
%\big(\frac{\cAil\Cnu^{2/3}}{2^{1/3}}\big)
\tstar
^{3/2}}}%
{\left(1-\big(1+\frac{\cAil\Cnl^{2/3}}{2^{1/3}}\big)^{-2}\right)^{1/4}}
\frac{1+\cAil\Cnl^{2/3}2^{-1/3}}{\sqrt{(1+\cAil\Cnl^{2/3}2^{-1/3})^2-1}}
\Big(\sqrt{\frac{2}{\pi}}+\frac{7}{\sqrt{2\pi}} \frac{1}{\tstar^{1/2}}\Big)}_
{C_1:=}.
\end{align*}
In a similar way and using \eqref{eq:JYppp} we are able to obtain that
\begin{align*}
|\besselj'_m(\evsadir_{n,m})+\besselj'_m(\besseljz_{n,m})|
\leq \besseljz_{n,m}^{-1/2}\big(1+\Cjrat^{1/2}\big)
\Big(\sqrt{\frac{2}{\pi}}+\frac{7}{\sqrt{2\pi}
%\big(\frac{\cAil\Cnu^{2/3}}{2^{1/3}}\big)
\tstar
^{1/2}}\Big)
\end{align*}
and
\begin{align*}
|\besselj'_m(\evsadir_{n,m})&-\besselj'_m(\besseljz_{n,m})|
\leq \eta
\frac{1+\cAil\Cnl^{2/3}2^{-1/3}}{\sqrt{(1+\cAil\Cnl^{2/3}2^{-1/3})^2-1}}
\sup_{z>\besseljz_{n-1,m}} |\besselj_m''(z)|\\
&\leq \frac{\eta}{\besseljz_{n,m}^{1/2}}
\frac{1+\cAil\Cnl^{2/3}2^{-1/3}}{\sqrt{(1+\cAil\Cnl^{2/3}2^{-1/3})^2-1}}
\Cjrat^{1/2}\bigg(
2\sqrt{\frac{2}{\pi}}+\frac{7}{\sqrt{2\pi}}\frac{1}{
%\big(\frac{\cAil\Cnu^{2/3}}{2^{1/3}}\big)
\tstar
^{1/2}}+\frac{13}{12\sqrt{2\pi}
%\big(\frac{\cAil\Cnu^{2/3}}{2^{1/3}}\big)
\tstar
^{3/2}}
\bigg).
\end{align*}
Thus
\begin{align*}
|\big(\besselj'_m(\evsadir_{n,m})-\besselj'_m(\besseljz_{n,m})\big)
\big(\besselj'_m(\evsadir_{n,m})+\besselj'_m(\besseljz_{n,m})\big)|
\leq \frac{\eta}{\besseljz_{n,m}} C_2
\end{align*}
with 
$C_2:=
\big(1+\Cjrat^{1/2}\big)
\Big(\sqrt{\frac{2}{\pi}}+\frac{7}{\sqrt{2\pi}\tstar^{1/2}}\Big)
\frac{1+\cAil\Cnl^{2/3}2^{-1/3}}{\sqrt{(1+\cAil\Cnl^{2/3}2^{-1/3})^2-1}}
\Cjrat^{1/2}\bigg(
2\sqrt{\frac{2}{\pi}}+\frac{7}{\sqrt{2\pi}}\frac{1}{
\tstar^{1/2}}+\frac{13}{12\sqrt{2\pi}
\tstar^{3/2}}\bigg)$.
Putting things together we have that
\begin{align*}
2\|\cdot^{1/2}&\besselj_m(\evsadir_{n,m}\cdot)\|_{L^2(0,1)}^2
\geq \frac{1}{\besseljz_{n,m}}
\Big(
\big(1-\frac{1}{1+\cAil\Cnl^{2/3}2^{-1/3}}\big)^2
\frac{\sin^2(\pi/4)}{\pi} -\eta(C_1+C_2) \Big).
\end{align*}
Next we estimate
\begin{align*}
2\|\cdot^{1/2}\bessely_m(\evsadir_{n,m}\cdot)\|_{L^2(m/\evsadir_{n,m},1)}^2
&=(1-\frac{m^2}{\evsadir_{n,m}^2}) \bessely_m(\evsadir_{n,m})^2
+\bessely'_m(\evsadir_{n,m})^2
-\frac{m^2}{\evsadir_{n,m}^2}\bessely'_m(m)^2\\
&\leq \bessely_m(\evsadir_{n,m})^2+\bessely'_m(\evsadir_{n,m})^2\\
&\leq
%\frac{\Cjrat}{\besseljz_{n,m}}
\besseljz_{n,m}^{-1}
\underbrace{\Cjrat
\Bigg(
\frac{\Big(\sqrt{\frac{2}{\pi}}+\frac{13}{12\sqrt{2\pi}\tstar^{3/2}}\Big)^2}%
{\Big(1-\big(1+\frac{\cAil\Cnl^{2/3}}{2^{1/3}}\big)^{-2}\Big)^{1/2}}
+\frac{\Big(\sqrt{\frac{2}{\pi}}+\frac{7}{\sqrt{2\pi}}\frac{1}{\tstar^{1/2}}\Big)^2}%
{\Big(1-\big(1+\frac{\cAil\Cnl^{2/3}}{2^{1/3}}\big)^{-2}\Big)^{-1/2}}
\Bigg)}_{2C_3:=}
\end{align*}
and thus
\begin{align*}
\|w_{n,m}\|_{\domcs}^2 \geq
\besseljz_{n,m}^{-1} \bigg(
\frac{\big(1-(1+\Cvec)\deltaem\big)^2}{2} 
\Big(
\big(1-\frac{1}{1+\cAil\Cnl^{2/3}2^{-1/3}}\big)^2
\frac{\sin^2(\pi/4)}{\pi} -\eta(C_1+C_2) \Big)
-\Cvec^2 C_3 \deltaem^2
\bigg).
\end{align*}
Hence we demand that
$\eta<\big(1-\frac{1}{1+\cAil\Cnl^{2/3}2^{-1/3}}\big)^2 \frac{\sin^2(\pi/4)}{2\pi(C_1+C_2)}$
and it follows that we can find $C_4>0$ and $\delta_4\in(0,\delta_3)$ such that
$\|w_{n,m}\|_{\domcs}^2 \geq C_4 \besseljz_{n,m}^{-1}$.
Hence we have proven the (more than) sufficient estimate $\|w_{n,m}\|_{\intf} \lesssim \|w_{n,m}\|_{\domcs}$.\\

\emph{3.~case: $n\in [\Cnl m, \Cnu m+1]$}.\quad
%\cite[10.21.42]{NISTDLMF}: $\besselj_m'(\besseljz_{n,m})=-\frac{2}{m^{2/3}}\frac{\Ai'(\Aiz_n)}{z(\zeta)h(\zeta)} \Big(1+O(m^{-2})\Big)$, $\zeta:=m^{-2/3}a_n$, $h(\zeta):=\big(\frac{4\zeta}{1-z(\zeta)^2}\big)^{1/4}$,
%$\frac{2}{3}(-\zeta)^{3/2}=\sqrt{z(\zeta)^2-1}-\arcsec z$
Since $\Cnu m+1 \leq \Cnu (m+1) \leq 2\Cnu m$ it suffices to consider $n\in [\Cnl m, 2\Cnu m]$.
Let $T>1$ and $\tstar>T$ be as in the 2.~case, i.e., such that \eqref{eq:T} is satisfied and $\tstar$ is chosen according to \Cref{lem:ev-conv}.
Let $\mstar\in\setN$ be such that
\begin{align}
\label{eq:MstarTstar}
\frac{\cAil(\Cnl \mstar)^{2/3}}{2^{1/3}} &> \tstar,
\end{align}
where \eqref{eq:MstarTstar} replaces the assumption \eqref{eq:CnuTstar} in the analysis of the 2.~case.
Hence it follows that
$\|w_{n,m}\|_{\domcs}^2 \geq \frac{C_4}{\besseljz_{n,m}}$ for all $(n,m)\in\{(n,m)\in\setN^2\colon m\geq\mstar,n\in [\Cnl m, 2\Cnu m]\}$.
Note that the analysis \eqref{eq:infl-est} of the 1.~case shows that
\begin{align*}
\|w_{n,m}\|_{\intf_l} \leq (1+(\Cab+1)\deltaem)|\besselj_m(\evsadir_{n,m}r_l)|
\quad\text{for }\quad\evsadir_{n,m}r_l<m.
\end{align*}
On the other hand we have that
\begin{align*}
\|w_{n,m}\|_{\intf_l} \leq (1+(\Cvec+1)\deltaem)\big(
\sup_{r>m} |\besselj_m(r)| + \sup_{r>m} |\bessely_m(r)| \big)
\quad\text{for }\quad\evsadir_{n,m}r_l\geq m.
\end{align*}
\Cref{lem:besselTransZone} and \Cref{lem:besselOscZoneII} yield that there exists a constant $C_5>0$ such that
\begin{align*}
\sup_{r>0} |\besselj_m(r)| + \sup_{r>m} |\bessely_m(r)| <\frac{C_5}{m^{1/3}},
\quad\forall m\in\setN.
\end{align*}
It follows that there exist $C_6>0$ such that $\|w_{n,m}\|_{\intf} \leq \frac{C_6}{m^{1/3}}$.
We conclude that
\begin{align*}
\|w_{n,m}\|_{\domcs}
\geq \frac{C_4^{1/2}}{\besseljz_{n,m}^{1/2}}
\geq \frac{C_4^{1/2}}{(1+\cAil\Cnl^{2/3}2^{-1/3})^{1/2}m^{1/2}}
&\geq \frac{C_4^{1/2}\Cnl^{1/6}}{(1+\cAil\Cnl^{2/3}2^{-1/3})^{1/2}m^{1/3}n^{1/6}}\\
&\geq \frac{C_4^{1/2}\Cnl^{1/6}}{(1+\cAil\Cnl^{2/3}2^{-1/3})^{1/2} C_6 n^{1/6}} 
\|w_{n,m}\|_{\intf},
\end{align*}
for all $m\geq\mstar$.
The remaining index set $\{(n,m)\in\setN^2\colon m=1,\dots,\mstar-1,n\in [\Cnl m, 2\Cnu m]\}$ is finite and hence it suffice to choose the constant $\Ceftrace$ sufficiently large.\\

\emph{The case $m=0$}.\quad
% \cite[p.~206]{Watson95}: p=2
% https://math.stackexchange.com/questions/1447137/conjectured-bound-on-bessel-function-of-the-first-kind?rq=1
For sufficiently large $\nstar>1$ we can use the asymptotic expansion of $\besselj_0$ and $\bessely_0$ \cite[p.~206]{Watson95} to proceed for $n>n\nstar$ as in the 2.~case.
For $n=1,\dots,\nstar$ we choose the constant $\Ceftrace$ sufficiently large.\\

\emph{Neumann boundary condition}.\quad
%\cite[10.22.38]{NISTDLMF}: $\int_0^1 r\besselj_m(\besseljz'_{m,n}r)^2 \, dr=\frac{\besseljz_{n,m}^2-m^2}{\besseljz_{n,m}^2} \frac{\besselj_m(\besseljz'_{m,n})^2}{2}$,
The treatment of this case is quite similar to the Dirichlet case.
However, a result for $(\Aipz_n)_{n\in\setN}$ corresponding to \eqref{eq:QuWong99} is not available.
Since we exploited \eqref{eq:QuWong99} only to obtain lower and upper bounds on $\besseljz_{n,m}$, it suffices to note that
\begin{align*}
\Aiz_{n} \inf_{l\geq2}\frac{\Aiz_{l-1}}{\Aiz_l} \leq \Aiz_{n-1} < \Aipz_n < \Aiz_n
\qquad\text{for } n=2,3,\dots,
\end{align*}
to obtain lower and upper bounds on $\besseljpz_{n,m}$.
\end{proof}
\begin{corollary}\label{cor:eftraces}
For each $m\in\Nz$ let $(\evsac_{n,m},(\bv_{n,m},w_{n,m}))\in\Rp\times \Xspace_m\setminus\{0\}$ be the eigenpairs to \eqref{eq:evp-sa-m}.
There exist $\Ceftrace,\delta_0>0$ such that if $\deltaem<\delta_0$, then for all $m\in\Nz$ and each $n\in\setN$ either $\bv_{n,m}=0$ or $w_{n,m}=0$, and in particular it holds that
$\|\bv_{n,m}\|_{\intf} \leq \Ceftrace n^{1/6} \|\bv_{n,m}\|_{\domcs}$ and
$\|w_{n,m}\|_{\intf} \leq \Ceftrace  n^{1/6}\|w_{n,m}\|_{\domcs}$.
\end{corollary}
\begin{proof}
Follows from \Cref{prop:eftraces} together with \Cref{lem:vector-to-scalar}.
\end{proof}

\section{The non-selfadjoint perturbation}\label{sec:perturbation-nsa}

\subsection{Local subordinate perturbations}

The main tool in the proof of our forthcoming main \Cref{thm:main} is the perturbation theory under a local subordinate condition of Mityagin and Siegl~\cite{MityaginSiegl19}.
For the comfort of the reader we provide in the next \Cref{lem:MS} a summary of the results applied from \cite{MityaginSiegl19} (accompanied by a slight extension~\cite[Thm.~A.9, Lem.~A.10]{DemkowiczHallaMelenk26}).
In order to stick to the notation of \cite{MityaginSiegl19}, but to avoid a conflict with previously used symbols, we equip all overlapping symbols in the context of \cite{MityaginSiegl19} with haceks.
\begin{lemma}\label{lem:MS}
Let $\check\Aop\colon \dom(\check\Aop)\subset\calH\to\calH$ be a closed, densely defined, (unbounded,) selfadjoint operator with compact resolvent in a separable, complex Hilbert space $\calH$.
Let $(\check\mu_k,\psi_k)_{k\in\setN}\in\setR\times\calH$ be the sequence of eigenpairs to $\check\Aop$ with $\calH$-normalized eigenelements:
\begin{align*}
\check\Aop \psi_k=\check\mu_k \psi_k, \quad \|\psi_k\|_{\calH}=1.
\end{align*}
We assume that the eigenvalues $(\check\mu_k)_{k\in\setN}$ of $\check\Aop$ are positive, simple and satisfy the gap condition
\begin{align*}
\exists\gamma,\kappa>0\colon\quad \check\mu_{k+1}-\check\mu_k\geq \kappa k^{\gamma-1} \quad\forall k\in\setN.
\end{align*}
Let $\check\bses(\cdot,\cdot)$ be a sesquilinear form which is defined and bounded on $\dom\check\Aop^{1/2}$ and satisfies the so-called local subordination condition
%Let $\check\Bop\colon \dom(\check\Bop)\subset\dom(\check\Aop)\to\calH$ satisfy the so-called local subordination condition
\begin{align*}
\exists \check\alpha\in\setR
\text{ with }
\begin{cases}
2\check\alpha+\gamma>\frac{3}{2}, &\text{if }\check\alpha\leq\frac{1}{2},\\
\gamma>\frac{1}{2}, &\text{if }\check\alpha>\frac{1}{2},
\end{cases}
\text{ and }
\MB>0\colon | \check\bses(\psi_m,\psi_n)|\leq \frac{\MB}{m^{\check\alpha} n^{\check\alpha}}
%\MB>0\colon |\spl \check\Bop\psi_m,\psi_n\spr_\calH|\leq \frac{\MB}{m^{\check\alpha} n^{\check\alpha}}
\quad\forall m,n\in\setN.
\end{align*}
Let $\check\Top$ be the operator induced by the elliptic sesquilinear form $\spl\check\Aop\cdot,\cdot\spr_{\calH}+\check\bses(\cdot,\cdot)$ and
$(\check\lambda_k,\phi_k)\in\setC\times\calH$, $k\in\setN$ be its eigenpairs
% of $\check\Aop+\check\Bop$
with $\phi_k$ being defined as in \cite[(3.17)]{MityaginSiegl19} (where admissible).
Then there exist $\cgka,\Cgka>0$ such that if $\MB<\cgka$, then all eigenvalues $\check\lambda_k$, $k\in\setN$ are simple, satisfy
\begin{align*}
|\check\mu_k-\check\lambda_k| &\leq \MB\bigg(\frac{1}{k^{2\check\alpha}}
+\Cgka\Big(\frac{\log(ek)}{k^{2\check\alpha+\gamma-1}}\Big)^2
\frac{1}{k^{2\check\alpha}}\bigg), &&\text{if }\check\alpha\leq\frac{1}{2},\\
|\check\mu_k-\check\lambda_k| &\leq \MB\bigg(\frac{1}{k^{2\check\alpha}}
+\Cgka\frac{1}{k^{2\gamma}}
\frac{1}{k^{2\check\alpha}}\bigg), &&\text{if }\check\alpha>\frac{1}{2},\\
\end{align*}
the formula \cite[(3.17)]{MityaginSiegl19} is admissible for all $k\in\setN$
and $(\phi_k)_{k\in\setN}$ forms a Riesz (and Bari) basis in $\calH$ with
\begin{align*}
\sum_{k\in\setN} \|\psi_k-\phi_k\|_{\calH}^2<\Cgka \MB.
\end{align*}
The constants $\cgka,\Cgka$ depend only on $\gamma,\kappa,\check\alpha$.
In addition, let $\calX:=\dom\check\Aop^{1/2}$ with $\spl \cdot,\cdot\spr_\calX:=\spl \check\Aop^{1/2}\cdot,\check\Aop^{1/2}\cdot\spr_\calH$ and $\tilde\psi_k:=\frac{1}{\check\mu_k^{1/2}}\psi_k$, $\tilde\phi_k:=\frac{1}{\check\mu_k^{1/2}}\phi_k$.
Then $(\tilde\phi_k)_{k\in\setN}$ forms a Riesz (and Bari) basis in $\calX$ with
\begin{align*}
\sum_{k\in\setN} \|\tilde\psi_k-\tilde\phi_k\|_{\calX}^2<\Cgka \MB.
\end{align*}
\end{lemma}
\begin{proof}
The statement regarding the eigenvalues and the properties of $(\phi_k)_{k\in\setN}$ in $\calH$ is nothing else, then a specification of \cite{MityaginSiegl19}.
In particular, a review of the proof of \cite[Thm~3.2]{MityaginSiegl19} reveals the dependency of the appearing estimates on the constants $\MB$, $\cgka$ and $\Cgka$.
Note that by assumption on $\MB<\cgka$ being small enough we have no preasymptotic, non-simple eigenvalues $\mu_k$ (characterized by $N_0$ in \cite{MityaginSiegl19}).
Hence there are no eigenvalues $\check\lambda_k$ with preasymptotic behavior and we can omit the patch $\Pi_0$ in \cite[Prop.~3.1]{MityaginSiegl19} around such.
Note that for the eigenvalue estimate we apply \cite[Thm~3.2]{MityaginSiegl19} with $j=1$ and estimate $\lambda_n^{(1)}$ with \cite[Rem~3.3]{MityaginSiegl19}.
To estimate $\sum_{k\in\setN} \|\psi_k-\phi_k\|_\calH^2$ we follow \cite[Rem~3.5]{MityaginSiegl19} and apply \cite[(3.19)]{MityaginSiegl19}.
The statement regarding properties in $\calX$ is taken from \cite[Thm.~A.9]{DemkowiczHallaMelenk26}.
\end{proof}
\begin{lemma}\label{lem:BariGram}
Let $Y$ be a Hilbert space, $(\psi_{n})_{n\in\setN}$ be an orthonormal basis of $Y$ and  $(\phi_{n})_{n\in\setN}$, $\phi_n\in Y, n\in\setN$ with $\qclose:=\sum_{n\in\setN} \|\psi_n-\phi_n\|_Y^2<1/2$.
Then $(\phi_{n})_{n\in\setN}$ is a Riesz basis of $Y$ with Gram constants $\frac{1}{2(1+\qclose)}$, $\frac{2}{1-2\qclose}$, i.e.,
for each $u\in Y$ it holds that $u=\sum_{n\in\setN} \spl u,\phi_n\spr_Y \phi_n$ and
\begin{align*}
\frac{1}{2(1+\qclose)} \sum_{n\in\setN} |\spl u,\phi_n\spr_Y|^2
\leq \|u\|_Y^2 \leq \frac{2}{1-2\qclose} \sum_{n\in\setN} |\spl u,\phi_n\spr_Y|^2.
\end{align*}
\end{lemma}
\begin{proof}
The claim follows from
\begin{align*}
\|u\|_Y^2 = \sum_{n\in\setN} |\spl u,\psi_n\spr_Y|^2
\leq 2\sum_{n\in\setN} |\spl u,\phi_n\spr_Y|^2+|\spl u,\psi_n-\phi_n\spr_Y|^2
\leq 2\sum_{n\in\setN} |\spl u,\phi_n\spr_Y|^2+2\qclose\|u\|_Y^2
\end{align*}
and
\begin{align*}
\sum_{n\in\setN} |\spl u,\phi_n\spr_Y|^2
\leq 2\sum_{n\in\setN} |\spl u,\psi_n\spr_Y|^2 +|\spl u,\psi_n-\phi_n\spr_Y|^2
\leq 2\|u\|_Y^2 +2\qclose\|u\|_Y^2
= 2(1+\qclose) \|u\|_Y^2.
\end{align*}
\end{proof}
\begin{remark}
Note that the assumption $\sum_{n\in\setN} \|\psi_n-\phi_n\|_Y^2<1/2$ in \Cref{lem:BariGram} is not sharp and could be mitigated to $\sum_{n\in\setN} \|\psi_n-\phi_n\|_Y^2<1$.
\end{remark}
\begin{lemma}\label{lem:evsquared}
Let $\kappa>0$ and $(\lambda_n)_{n\in\setN}$, $\lambda_n>0$, $n\in\setN$ be a monotone increasing sequence satisfying $\lambda_1\geq\kappa$ and $\lambda_{n+1}-\lambda_n \geq \kappa$ for all $n\in\setN$.
Then $\lambda_{n+1}^2-\lambda_n^2 \geq 2\kappa^2 n$ for all $n\in\setN$.
\end{lemma}
\begin{proof}
Note that $\lambda_{n+1}=\lambda_{n+1}-\lambda_{n}+\lambda_{n}\geq \kappa+\lambda_n$ and hence $\lambda_n\geq\kappa n$ for all $n\in\setN$.
Now we compute that
$
\lambda_{n+1}^2-\lambda_n^2=(\lambda_{n+1}-\lambda_n)(\lambda_{n+1}+\lambda_n)
\geq \kappa(\lambda_{n}+\lambda_n)
\geq 2\kappa^2 n,
$
which proves the claim.
\end{proof}

\subsection{Main results}

\begin{theorem}\label{thm:main}
Let the assumptions specified in \Cref{subsec:setting} be satisfied.
There exists $\delta_0>0$, such that if $\deltaem<\delta_0$, then:
If $(\alpha,(\Et,\Ez,\Ht,\Hz))$ is a solution to \eqref{eq:evp-full}, then $\alpha\neq0$ and $(-\alpha,(\Et,-\Ez,-\Ht,\Hz))$ is also a solution to \eqref{eq:evp-full}.
So let
\begin{align*}
(\alpha_k,(\Etk,\Ezk,\Htk,\Hzk)),\quad (-\alpha_k,(\Etk,-\Ezk,-\Htk,\Hzk)),\quad k\in\setN
\end{align*}
(with $\alpha_k\neq-\alpha_{k'}$ for all $k,k'\in\setN$)
be all solutions to \eqref{eq:evp-full}.
\begin{enumerate}[I.]
\item\label{enum:I} Each $\pm\alpha_k$, $k\in\setN$ is simple.
\item\label{enum:II} If $\omega\in\setR\setminus\{0\}$, then $\alpha_k^2\in\setR$ for all $k\in\setN$.
\item\label{enum:III} If $\omega\in\setR\setminus\{0\}$ and $\beta_k:=i\alpha_k\in\setR$ for a $k\in\setN$, then $\frac{\beta_k(\omega)}{\omega}\frac{d\beta_k(\omega)}{d\omega}>0$.
\item\label{enum:IV} After a normalization $(\Etk)_{k\in\setN}$ forms a Riesz basis in $\Hspace_0(\curlg;\domcs)$.
\item\label{enum:V} %If $\mu=\mu_0$, then
After a normalization $(\Etk)_{k\in\setN}$ forms a Riesz basis in $\Hspace_0^{-1/2}(\curlg;\domcs)$.
\item\label{enum:VI} After a normalization $(\Htk)_{k\in\setN}$ forms a Riesz basis in $\Hspace(\curlg;\domcs)$.
\item\label{enum:VII} %If $\epsilon=\epsilon_0$, then
After a normalization $(\Htk)_{k\in\setN}$ forms a Riesz basis in $\Hspace^{-1/2}(\curlg;\domcs)$.
\end{enumerate}
\end{theorem}
\begin{proof}
\textit{Preliminary.}\quad
By assumption $\omega$ is not a cut-off frequency for the homogeneous case $\epsilon=\epsilon_0$, $\mu=\mu_0$.
Due to \Cref{lem:cutoff} there exists $\delta_1>0$ such that if $\deltaem<\delta_1$, then $\omega$ is not cut-off wavenumber for the heterogeneous case.
Henceforth let $\deltaem<\delta_1$.
Thus $\alpha\neq0$.
If $(\alpha,(\Et,\Ez,\Ht,\Hz))$ is a solution to \eqref{eq:evp-full}, then $(\Et,-\Ez,-\Ht,\Hz))\neq0$ and it can easily be checked that $(-\alpha,(\Et,-\Ez,-\Ht,\Hz))$ solves \eqref{eq:evp-full}.\\

\textit{Separated eigenvalue problems.}\quad
Let $\Hspace(\divg\mu^{-1}0;\domcs):=\{\bv\in\bL^2(\domcs)\colon \divg(\mu^{-1}\bv)=0\}$,$\calH:=\Hspace(\divg\mu^{-1}0;\domcs)\times L^2(\domcs)$ be equipped with the scalar product $\spl (\bv,w),(\bv\tf,w\tf)\spr_{\calH}:=\spl \mu^{-1}\bv,\bv\tf\spr_{\bL^2(\domcs)}+\spl \epsilon w,w\tf\spr_{L^2(\domcs)}$ and $\calX:=\Xspace$.
For each $m\in\setZ$ consider the following.
Let $\calH_m:=\Hspace_m(\divg\mu^{-1}0;\domcs)\times L^2_{m}(\domcs)$ (equipped with the inherited weighted scalar product) and $\calX_m:=\Xspace_m$.
The operators $\Aop_m,\Bop_m\in\BLO(\calX_m)$ induce unbounded sesquilinear forms $\ases_m(\cdot,\cdot),\bses_m(\cdot,\cdot)$ on $\calH_m\times\calH_m$ with domains $\dom(\ases_m):=\dom(\bses_m):=\Xspace_m\subset\calH_m$.
%Following \cite[Ch.~6.2]{Kato95}, \cite{MityaginSiegl19} $\ases_m(\cdot,\cdot)$, $\bses_m(\cdot,\cdot)$ induce unbounded operators $\check\Aop_m\colon\dom(\check\Aop_m)\subset\calH_m\to\calH_m$,
Following \cite[Ch.~6.2]{Kato95}, \cite{MityaginSiegl19} $\ases_m(\cdot,\cdot)$, $\ases_m(\cdot,\cdot)+\bses_m(\cdot,\cdot)$ induce unbounded operators $\check\Aop_m\colon\dom(\check\Aop_m)\subset\calH_m\to\calH_m$,
$\check\Top_m\colon\dom(\Top_m)\subset\calH_m\to\calH_m$
%$\check\Bop_m\colon\dom(\Bop_m)\subset\calH_m\to\calH_m$
with $\dom(\check\Aop_m^{1/2})=\dom(\ases_m)$.
In a similar way $\Kop_m\in\BLO(\calX_m)$ induces a bounded sesquilinear form $\kses_m(\cdot,\cdot)\colon\calH_m\times\calH_m\to\setC$, which induces the identity operator $\check\Kop_m=I_{\calH_m}$.
Let $\cgapt,\delta_2>0$ be as in \Cref{prop:evgap}, $\delta_3:=\min\{\delta_1,\delta_2\}$ and $\deltaem<\delta_3$ henceforth.
Set $\gamma:=2$, $\kappa:=2\cgapt^2$ and $\check\alpha:=-1/6$.
Let $\Ceftrace,\delta_4>0$ be as in \Cref{cor:eftraces} and $\CBop,\delta_5>0$ be as in \Cref{lem:Bbound}.
Let $\delta_6\in(0,\min\{\delta_3,\delta_4,\delta_5\})$ be such that $6\Ceftrace\CBop(\epsilon_0+\mu_0)\delta_6<\cgka$ (with $\cgka$ as in \Cref{lem:MS}) and $\deltaem<\delta_6$ henceforth.
Let $(\evsac_{n,m}^2,\psi_{n,m})_{n\in\setN}$, $\psi_{n,m}\in\calX_m$ be the normalized eigenpairs of $\check\Aop_m$.
Then by means of \Cref{lem:evsquared} the assumptions of \Cref{lem:MS} are satisfied.
Thus the eigenpairs $(\check\lambda_{n,m},\phi_{n,m})_{n\in\setN}$ of
%$\check\Aop_m+\check\Bop_m$
$\check\Top_m$ (with $\phi_{n,m}$ as defined in \Cref{lem:MS}) satisfy:
each eigenvalue $\check\lambda_{n,m}$, $n\in\setN$ is simple,
\begin{align*}
|\evsac_{n,m}^2-\check\lambda_{n,m}|
&\leq 6\Ceftrace\CBop(\epsilon_0+\mu_0)\deltaem \bigg(1+\Cgka\bigg(\frac{\log(en)}{n}\bigg)^2\bigg)\\
&\leq 6\Ceftrace\CBop(\epsilon_0+\mu_0) (1+\Cgka) \deltaem,
\end{align*}
and $(\phi_{n,m})_{n\in\setN}$ and $(\evsac_{n,m}^{-1}\phi_{n,m})_{n\in\setN}$ form a Riesz (even Bari) basis in $\calH_m$ and $\calX_m$ respectively with
\begin{align*}
\sum_{n\in\setN} \|\psi_{n,m}-\phi_{n,m}\|_{\calH_m}^2&<6 \Cgka \Ceftrace\CBop(\epsilon_0+\mu_0)\deltaem,\\
\sum_{n\in\setN} \|\evsac_{n,m}^{-1}\psi_{n,m}-\evsac_{n,m}^{-1}\phi_{n,m})\|_{\calX_m}^2&<6 \Cgka \Ceftrace\CBop(\epsilon_0+\mu_0)\deltaem.
\end{align*}
Since $\alpha_{n,m}^2=\check\lambda_{n,m}-\omega^2\epsilon_0\mu_0$, the first claim~\ref{enum:I} follows.\\

\textit{\ref{enum:II}.}\quad
If $\omega\in\setR\setminus\{0\}$ and $\alpha_{n,m}$ is an eigenvalue to \eqref{eq:evp-full} with $\alpha_{n,m}^2\in\setC\setminus\setR$, then $\ol{\alpha}_{n,m}\neq\alpha_{n,m}$ is an eigenvalue too.
However, in the derivation of \Cref{lem:MS} each patch $\Pi_k$ (around $\besseljzc_{n,m}^2$) defined in \cite[(3.1)]{MityaginSiegl19} contains only a single eigenvalue, which is a contradiction to $|\besseljzc_{n,m}^2-\alpha^2_{n,m}|=|\besseljzc_{n,m}^2-\ol{\alpha}^2_{n,m}|$.
Note that the patch $\Pi_0$ is omitted as discussed in the proof of \Cref{lem:MS}.
\\

\textit{\ref{enum:III}.}\quad
For brevity we surpress the index $k$.
Due to \cite[Thm.~3.3]{Halla26Calderon} we have that
%We compute by means of \eqref{eq:EH-curl-div} that
\begin{align*}
\frac{\beta}{\omega\epsilon_0\mu_0} \frac{d\beta(\omega)}{d\omega}
&=\frac{\spl\epsilon\Et,\Et\spr_\domcs+\spl\mu\Ht,\Ht\spr_\domcs - \frac{\beta}{\omega\epsilon_0\mu_0} \Re(\spl(\epsilon\mu-\epsilon_0\mu_0)\rotm\Ht,\Et\spr_\domcs)}%
{\spl\epsilon\Et,\Et\spr_\domcs+\spl\mu\Ht,\Ht\spr_\domcs +\frac{\omega}{\beta} \Re(\spl(\epsilon\mu-\epsilon_0\mu_0)\rotm\Ht,\Et\spr_\domcs)}
>0.
\end{align*}
%Hence, we can find $\delta_7\in(0,\delta_6)$ such that the former denominator is actually positive for all $\deltaem\in(0,\delta_7)$ (which we assume henceforth) and
%\begin{align*}
%\frac{\beta}{\omega\epsilon_0\mu_0} \frac{d\beta(\omega)}{d\omega}\in \bigg[
%\frac{1-\deltaemt\beta/\omega}%
%{1+\deltaemt\omega\epsilon_0\mu_0/\beta},
%\frac{1+\deltaemt\beta\epsilon_0\mu_0/\omega}%
%{1-\deltaemt\omega\epsilon_0\mu_0/\beta} \bigg],
%\end{align*}
%from which the claim follows.
\\

\textit{\ref{enum:IV}.}\quad
Let $\delta_7\in(0,\delta_6)$ be such that $6 \Cgka \Ceftrace\CBop(\epsilon_0+\mu_0)\delta_7<1/4$ and $\deltaem<\delta_7$ henceforth.
Then \Cref{lem:BariGram} and \Cref{lem:compoundbasis} yield that $(\phi_{n,m})_{n\in\setN,m\in\setZ}$ forms a Riesz basis in $\calH_m$
and $(\evsac_{n,m}^{-1}\phi_{n,m})_{n\in\setN,m\in\setZ}$ forms a Riesz basis in $\calX_m$, both times with Gram constants $c>1/4$ and $C<4$.
Note that $\frac{|\lambda_{n,m}|}{|\evsac_{n,m}|}\to1$ as $n\to+\infty$ or $m\to\pm\infty$.
Hence we can trade a normalization constant for $\evsac_{n,m}^{-1}$.
Recall the bijection
$\Vspace_m\times \Wm\to\Hspace_{0m}(\curlg;\domcs)\colon \phi_m=(\bv_m,w_m)^\top\mapsto \bv_m+\omega^{-1}\gradg w_m$ \cite[Lem.~A.13]{Halla26Calderon}.
It follows with $\phi_{n,m}=(\bv_{n,m},w_{n,m})^\top$ and \Cref{lem:compoundbasis} that after normalization $((\Et)_{n,m}=\bv_{n,m}+\omega^{-1}\gradg w_{n,m})_{n\in\setN,m\in\setZ}$ forms a Riesz basis in $\Hspace_0(\curlg;\domcs)$.
\\

\textit{\ref{enum:V}.}\quad
Let $m\in\setZ$ and $(\phi_{n,m})_{n\in\setN}$ be $\calH_m$-normalized.
Then
\begin{align*}
\frac{1}{4} \sum_{n\in\setN} |\spl \phi_m,\phi_{n,m} \spr_\calH|^2
\leq \|\phi_m\|_\calH^2
\leq 4 \sum_{n\in\setN} |\spl \phi_m,\phi_{n,m} \spr_\calH|^2
\end{align*}
and
\begin{align*}
\frac{1}{4} \sum_{n\in\setN} \besseljzc_{n,m}^2 |\spl \phi_m,\phi_{n,m} \spr_\calH|^2
\leq \|\phi\|_\calX^2
\leq 4 \sum_{n\in\setN} \besseljzc_{n,m}^2 |\spl \phi_m,\phi_{n,m} \spr_\calH|^2
\end{align*}
for all $\phi_m\in\calX_m$.
By Hilbert space interpolation \cite{BerghLoefstroem76} the norm of the space $[\calH_m,\calX_m]_{1/2}$ is equivalent (with $m$-uniform bounds) to $\sqrt{\sum_{n\in\setN} \besseljzc_{n,m} |\spl \cdot,\phi_{n,m} \spr_\calH|^2}$ and after normalization $(\phi_{n,m})_{n\in\setN}$ forms a Riesz basis of $[\calH_m,\calX_m]_{1/2}$.
With 
\begin{align*}
\calP^{0}_\mu(\domcs)&:=H^1_*(\domcs),\\
\calP^{1/2}_\mu(\domcs)&:=\{u\in H^1_*(\domcs)\colon \quad \div_{\domcs}\mu\nabla_{\domcs}u\in H^{1/2}(\domcs)' \quad\text{and}\quad \tr_{\nvt}\mu\nabla_{\domcs}u=0\},\\
\calP^{1}_\mu(\domcs)&:=\{u\in H^1_*(\domcs)\colon \quad \div_{\domcs}\mu\nabla_{\domcs}u\in L^{2}(\domcs) \hspace{4.65mm}\quad\text{and}\quad \tr_{\nvt}\mu\nabla_{\domcs}u=0\},
\end{align*}
we have that $\Hspace(\divg\mu^{-1}0;\domcs)=\mu\Curlg\calP^{0}_\mu(\domcs)$, $\Vspace=\mu\Curlg\calP^{1}_\mu(\domcs)$ and hence by means of \cite[Lem.~A.14]{Halla26Calderon} that $[\Hspace(\divg\mu^{-1}0;\domcs),\Vspace]_{1/2}=\mu\Curlg\calP^{1/2}_\mu(\domcs)$.
Further, it is well known that $[L^2(\domcs),H^1_0(\domcs)]_{1/2}=H^{1/2}_{00}(\domcs)$ \cite{AgranovichBook}.
Thus $[\calH_m,\calX_m]_{1/2}=\mu\Curlg\calP^{1/2}_{\mu,m}(\domcs)\times H^{1/2}_{00,m}(\domcs)$ and the claim follows from the regularity of the decomposition $\Hspace_0^{-1/2}(\curlg;\domcs)=\mu\Curlg\calP^{1/2}_{\mu}(\domcs)\oplus\gradg H^{1/2}_{00}(\domcs)$ \cite[Lem.~A.12]{Halla26Calderon}.
\\

\textit{\ref{enum:VI}+\ref{enum:VII}.}\quad
We derive the same results for $\Ht$ by trading $\epsilon,\mu,\Et,\Ez$ for $\mu,\epsilon,\Ht,\Hz$ and adapting the boundary conditions.
\end{proof}
\begin{remark}\label{rem:LtwoBasis}
We do not expect that the basis properties \ref{enum:IV}-\ref{enum:VI} of \Cref{thm:main} could possibly be extended to $\bL^2(\domcs)$.
Indeed, for the setting of two materials ($\indL=2$) \cite[Thm.~2.2.3]{SheSmirSmo22} proves that $(\Etk,\Htk)_{k\in\setN}\cup(\Etk,-\Htk)_{k\in\setN}$ does not form a basis in $\bL^2(\domcs)\times\bL^2(\domcs)$ (and hence in particular no Riesz basis).
Making a two-by-two block transformation it follows that $(\Etk,0)_{k\in\setN}\cup(0,\Htk)_{k\in\setN}k$ is not a Riesz basis of $\bL^2(\domcs)\times\bL^2(\domcs)$.
Thus both $(\Etk)_{k\in\setN}$ and $(\Htk)_{k\in\setN}$ do not form a Riesz basis of $\bL^2(\domcs)$.
\end{remark}

\printbibliography %[heading=bibintoc]
\end{document}